\documentclass[twoside,a4paper,reqno]{amsart}
\usepackage[top=3.75cm,bottom=3.75cm,left=3.25cm,right=3.25cm]{geometry}
\usepackage{graphicx}
\usepackage{float}
\usepackage{fancyhdr}
\usepackage{tikz}
\usepackage{bbm}
\usepackage{amsmath}
\usepackage{tikz}
\usepackage{physics}
\usepackage{subcaption}
\usepackage{pgfplots}
\usepackage{amsmath,amssymb}
\usetikzlibrary{angles,quotes}
\usepackage{xcolor}
\usetikzlibrary{decorations.markings}
\usetikzlibrary{patterns}
 \usetikzlibrary{automata,topaths}
\usepackage[greek,english]{babel}
\usepackage{alphabeta}
\usepackage{amssymb}
\usetikzlibrary{shapes.geometric}
\usetikzlibrary{shapes.geometric,fit,backgrounds,arrows.meta}
\usepackage{pgfplots}
\pgfplotsset{width=10cm,compat=1.9}

\usepgfplotslibrary{external}

\usetikzlibrary{decorations.markings}
\usetikzlibrary{calc,fadings,decorations.pathreplacing}

\tikzset{%
  >=latex, 
  inner sep=0pt,%
  outer sep=2pt,%
  mark coordinate/.style={inner sep=0pt,outer sep=0pt,minimum size=3pt,
    fill=black,circle}%
}

\usepackage[outline]{contour} 
\usetikzlibrary{angles,quotes} 
\usetikzlibrary{bending} 
\contourlength{1.0pt}
\usetikzlibrary{3d}
\tikzset{>=latex} 
\usepackage{xcolor}
\colorlet{myblue}{blue!65!black}
\colorlet{mydarkblue}{blue!50!black}
\colorlet{myred}{red!65!black}
\colorlet{mydarkred}{red!40!black}
\colorlet{veccol}{green!70!black}
\colorlet{vcol}{green!70!black}
\colorlet{xcol}{blue!85!black}
\tikzstyle{vector}=[->,very thick,xcol,line cap=round]
\tikzstyle{xline}=[myblue,very thick]
\tikzstyle{yzp}=[canvas is zy plane at x=0]
\tikzstyle{xzp}=[canvas is xz plane at y=0]
\tikzstyle{xyp}=[canvas is xy plane at z=0]

\tikzset{mynode/.style={inner sep=2pt,fill,outer sep=0,circle}}

\newcommand{\setN}{\mathds N}

\newcommand{\R}{\mathds R}

\usetikzlibrary{decorations.markings}
\tikzset
  {marking1/.style=
     {decoration=
       {markings,mark=at position 0.6 with {\arrow[line width=1pt]{to}}},
      postaction=decorate
     },
   marking1r/.style=
     {decoration=
        {markings,
         mark=at position 0.6 with {\arrow[line width=1pt,xscale=-1]{to}}
        },
      postaction=decorate
     },
   marking2/.style=
     {decoration=
        {markings,
         mark=at position 0.18 with {\arrow[line width=1pt]{to}},
           },
      postaction=decorate
     }
  }

\def\eps{\varepsilon}

\def\e{{\rm e}}

\newcommand{\epsi}{\eps}

\usepackage{color}

\usetikzlibrary{patterns}

\tikzset{
    arc arrow/.style args={%
    to pos #1 with length #2}{
    decoration={
        markings,
         mark=at position 0 with {\pgfextra{%
         \pgfmathsetmacro{\tmpArrowTime}{#2/(\pgfdecoratedpathlength)}
         \xdef\tmpArrowTime{\tmpArrowTime}}},
        mark=at position {#1-\tmpArrowTime} with {\coordinate(@1);},
        mark=at position {#1-2*\tmpArrowTime/3} with {\coordinate(@2);},
        mark=at position {#1-\tmpArrowTime/3} with {\coordinate(@3);},
        mark=at position {#1} with {\coordinate(@4);
        \draw[-{Stealth[length=#2,bend]}]
        (@1) .. controls (@2) and (@3) .. (@4);},
        },
     postaction=decorate,
     }
}

\usepackage{xcolor}

\usepackage{amsmath}
\newcommand\numberthis{\addtocounter{equation}{1}\tag{\theequation}}

\usepackage{amssymb}
\usepackage{amsmath}
\usepackage{amsthm}
\usepackage{pgf}
\usepackage{color}
\usepackage{graphicx}
\usepackage{hyperref}

\newtheorem{theorem}{Theorem}[section]
\newtheorem{lemma}[theorem]{Lemma}
\newtheorem{corollary}[theorem]{Corollary}
\newtheorem{proposition}[theorem]{Proposition}

\theoremstyle{definition}
\newtheorem{definition}[theorem]{Definition}

\numberwithin{equation}{section}

\theoremstyle{definition}

\usepackage[LGR, T1]{fontenc}

\def\eps{\varepsilon}

\def\e{{\rm e}}

\usepackage[top=3.75cm,bottom=3.75cm,left=3.25cm,right=3.25cm]{geometry}
\usepackage{graphicx}

\usepackage{fancyhdr}
\usepackage{tikz}

\usepackage{xcolor}

\usepackage[utf8]{inputenc}
\usepackage{tikz}
\usetikzlibrary{positioning}

\usepackage{amsmath}
\usepackage{amsfonts}
\usepackage{amssymb}
\usepackage{amsthm}
\usepackage{ dsfont }
\theoremstyle{remark}

\usepackage{amsmath}
\usepackage{amsfonts}
\usepackage{amssymb}
\usepackage{amsthm}
\usepackage{ dsfont }
\usepackage{ esint }

\renewcommand{\d}{\mathrm{d}}

\usepackage{color}

\begin{document}
\title[Fractional diffusion as the limit of a Rayleigh gas]{Fractional diffusion as the limit of a short range potential Rayleigh gas}
\author{Karsten Matthies \and  Theodora Syntaka*}
\address{University of Bath,
Department of Mathematical Sciences,
Bath,
BA2 7AY,
United Kingdom}
\email{k.matthies@bath.ac.uk; theodora.syntaka@bath.edu}

\keywords{}

\begin{abstract}
 The  fractional diffusion equation is rigorously  derived as a scaling limit from a deterministic Rayleigh gas, where
 particles interact via short range potentials with support of size $\varepsilon$ and the background is distributed in space $\R^3$ according to a Poisson process with intensity $N$ and in velocity according to some fat-tailed distribution. As an intermediate step a  linear Boltzmann equation is obtained  in the Boltzmann-Grad limit as $\varepsilon$ tends to zero and $N$ tends to infinity with $N \varepsilon^2 =c$.   The convergence of the empiric particle dynamics to the Boltzmann-type dynamics is shown using semigroup methods to describe probability measures on collision trees associated to physical trajectories in the case of a Rayleigh gas. The fractional diffusion equation is a hydrodynamic limit  for times $t \in [0,T]$,  where $T$ and inverse mean free path $c$ can  both be chosen as some negative rational power $\varepsilon^{-k}$.
\end{abstract}

\maketitle
\tableofcontents
*Corresponding author: Theodora Syntaka, theodora.syntaka@bath.edu
\section{Introduction}

A fundamental problem  is the relation of continuum descriptions of matter on the macroscopic level --such as the Navier-Stokes or diffusion equations-- to the microscopic description of matter  by deterministic particle systems. 
For ideal  gases  Boltzmann introduced a mesoscopic level of description in the form of  kinetic theory, 
which provides probability distributions for the particles in their phase space. These models of the physical system can have contrasting properties on different scales. Typical macroscopic and kinetic equations make the system irreversible in time, while in the microscopic level the system is described by time-reversible Newtonian mechanics. The general area can be understood as modern interpretation of Hilbert's Sixth Problem. 

The full derivation from micro  to kinetic and to other macro dynamics with or without  using kinetic equations as an intermediate step, is still an open question for many equations. 
The rigorous derivation of the Boltzmann equation from a system of particles was first obtained by Lanford \cite{lanford75}, in the case of hard spheres and by King \cite{King1975} for more general potentials using the BBGKY hierarchy. The convergence in these results is valid for short times, i.e. in a time interval $[0,T']$, where $T'$ is a fraction to the free flight time. Complete proofs are given in \cite{Gallagher2013, pul13}. Illner and Pulvirenti \cite{Illner1986, Illner1989} obtained global in time results of convergence, when the initial density is small and the positions are in $\mathds{R}^d$. There has been recent, substantial progress in the parallel problem of deriving the kinetic wave equation for the interaction of waves within nonlinear Schr\"odinger equations
\cite{Lukkarinen2011,Deng2021,Deng2023,Deng2023a}. Recently, Hani, Deng and Ma announced substantial progress for the long-term derivation of the Boltzmann equation \cite{deng2024longtimederivationboltzmann} and then for derivation of Euler and Navier-Stokes equations \cite{deng2025hilbertssixthproblemderivation} using the Boltzmann equation as an intermediate step. 

Building on \cite{Gallagher2013, pul13} fluctuations around the Maxwellian equilibrium of the Boltzmann equation can be understood and used to derive macroscopic equations and processes. Bodineau, Gallagher and Saint-Raymond  \cite{MR3455156} obtain Brownian motion as the limit of a deterministic system of hard-spheres for arbitrary large times and for an initial distribution which is close to the equilibrium state of a tagged particle. 
The Ornstein-Uhlenbeck process is derived for a similar setting in \cite{Bodineau2018}. Further descriptions  for fluctuations around the equilibrium of the Boltzmann equation are provided in \cite{Bodineau2023a, Bodineau2023}.

In linear settings, variants of the Boltzmann equation can be justified for arbitrary long times, see e.g. \cite{spohn91} and references therein. Even in the nonlinear setting  of \cite{MR3455156} the authors  use the linear Boltzmann equation as an intermediate step, to pass to the Brownian motion as the limit of a deterministic system of hard-spheres, as a corollary a diffusion equation for the macroscopic distribution is obtained. The time-scale for convergence is of the order of $(\log \log N)^{-1}$. In this paper we will consider a linear setting of a tagged particle interacting with a background in a variant of the Rayleigh gas to obtain linear Boltzmann equations and then other macroscopic behaviour. A related recent result deriving diffusive behaviour of the tagged particle in a Rayleigh gas particle model with a Maxwellian background is provided in \cite{fougeres2024derivation}. In this article, Foug\`eres managed to find a rate of convergence which is of the order of $\exp ( -c_{\beta}  | \log N|^{1-\alpha} )$, $\forall \alpha >0$, with $N$ be the number of particles, for times of the order $(\log|c_{\beta} \log \varepsilon| )^{\frac{1}{2}-\alpha}$ using the BBGKY hierarchy approach. This rate improves the one of \cite{MR3455156} which is of the order $(\log \log N)^{-1}$. In \cite{MatthiesSyntaka2024surrey}, we show, starting also from a Rayleigh gas particle system with a Maxwellian background, which takes advantage of semigroup techniques as in \cite{Matthies2018} to study the evolution of collision trees instead of the BBGKY hierarchy, that the rate of convergence of the distribution of the tagged particle to the solution of the linear Boltzmann equation is of the order $\varepsilon^{\alpha}$, with $\alpha \in (0,\frac{1}{8})$ for time scales that are proportional to some negative power of $\varepsilon$.

The present paper provides a first result of deriving a \emph{fractional diffusion} equation as the limit of a deterministic particle system again on time scales that are proportional to some negative power of $\varepsilon$. Here, we are giving a specific example of short range potential Rayleigh gas and with a specific fat-tailed background, using a linear Boltzmann equation as an intermediate step.

\subsection*{Particle dynamics}
We start by describing the particle dynamics. A wide, relevant class are systems with short range potentials where the particles carry a force that affects only nearby particles, up to some distance proportional to $\varepsilon$. Our current work focuses on specific short range potential models. When the particles interact via a short range potential, then typically only two particles are interacting.  
Thus the main analysis needs only the description of the collision of $N = 2$ particles. The Hamiltonian equations of motion are given by
\begin{align*}
\frac{d x_i(t)}{dt} = v_i, \quad  m_i \frac{d v_i(t)}{dt} =  - \nabla U_\eps (x_i - x_j),
\end{align*}
where $(x_i (t),v_i(t))\in \R_x^3 \times \R_v^3$ is the position and velocity of particle $i$ at time $t$, for $i = 1, 2$,  where $i \neq j=1,2$,  and $t \ge 0$. Here, $\varepsilon$ is the diameter of the particles, $m_i$ is the mass of the $i$-th particle and we consider the potential function
\begin{align} \label{eqn:U}
   U : \mathds{R}^3 \to \mathds{R}, \ U ( | x| ) :=   \bar{K}(|x|^{1-n}-1)\mathds{1}_{|x|\le 1}
\end{align}
to be the interaction potential for $n\in(2,5]$ and $\bar{K}>0$, which is radial, non-increasing, supported in the unit ball of $\mathds{R}^3$ and goes to zero at $|x|=1$. Since $\bar{K}$ is a positive constant, we can set $\bar{K}=1$ along the article for simplification. The scattering event for each particle can be parametrised by the scattering angle $\Theta$ defined as the angle between ingoing and outgoing relative velocities, see Fig. \ref{Fig1}. See \cite{Gallagher2013,cerci88, pul13} for extended discussions. 

\begin{figure}[h!]
\begin{center}
\begin{tikzpicture}[scale=0.7][
dot/.style={
  fill,
  circle,
  inner sep=1pt
  }
]
\begin{scope}[x=1.5cm,y=1.5cm]

\draw (0,0) circle [radius=2];

     \draw
    ({sqrt(2)},{sqrt(2)}) coordinate (a)
    -- (0,0) coordinate (b)
    -- (     -0.7652, 1.8476     ) coordinate (c)
    pic["{$\Theta$}", draw, ->, angle eccentricity=1.8, angle radius=0.5cm]
    {angle=a--b--c};

       \draw [-] (0,0) -- ( -2, 0 );

      \path[-] (-2.25,0.25) edge [bend right] node[below] {} ( {sqrt(2)},{sqrt(2)+0.35} );

        \draw [-,line width=0.9, color= blue ] (0,0) -- (  -0.7652, 1.8476  );

\end{scope}

\end{tikzpicture}
\end{center}\caption{Two particle interaction with scattering angle $\Theta$.} \label{Fig1}
\end {figure}

Particles are assumed to be small with the size of the support of  the forces parametrised by $\varepsilon$.~Then the particle dynamics are given by
$\nabla U_{\varepsilon} (x) = \frac{1}{\varepsilon} \nabla U (\frac{x}{\varepsilon})$. 
These dynamics preserve overall kinetic energy and momentum. The dynamics and its potential satisfy the conditions in  \cite{Gallagher2013}.


\subsection*{Tagged particle models}

One option to simplify the dynamics --both on the particle level and on the level of the continuum equations--  is to split the system of particles into two kinds of particles, the background and the tagged particles. The tagged particle is single and interacts among a system of background particles which are assumed not to interact among themselves. This can be motivated e.g. for systems that are in equilibrium.

If the background particles are of infinite relative mass to the tagged particle and the background particles are fixed, then this model is known as the Lorentz gas model. Long-term convergence was shown in \cite{spohn78}. For recent progress  about random Lorentz gas see the paper of Lutsko and Toth, \cite{Lutsko2020}. In this work, the authors prove the invariance principle for a random Lorentz-gas particle in three space dimension under the Boltzmann-Grad limit and simultaneous diffusion scaling. They do this by using a coupling of the mechanical trajectory and some controls on the efficiency of this coupling. Another recent result on the Lorentz gas is  by Marklof and Str\"ombergsson \cite{Marklof} which provides a derivation of transport processes from a deterministic Lorentz gas particle system for radial potentials with compact support in the low density limit .

A Rayleigh gas is related to the Lorentz gas, but the background particles are no longer of infinite mass. If the background particles are of equal mass with the tagged particle and the background particles interact only with the tagged particle and not with each other then this model is known as the Rayleigh gas. We consider a tagged particle with initial distribution $f_0$ and $g_0$ the distribution of the background particles, distributed via a spatial homogeneous Poisson process. We index the particles by $i \in \setN$. There is no interaction among the background particles.

In the main part of our analysis we can ignore multiple collision of the tagged particle at the same time. Such events have vanishing probability in the scaling limit. Due to the compact support of $\nabla U_\varepsilon$ the tagged particle will only interact with a finite number of particles with probability $1$.  The equations of motions are given by
\begin{align}
\begin{cases}
\dot{x}_i(t)& =v_i(t)  \text{ for  all } i\in \{0\} \cup \setN \\
    \dot{v}_0(t) &=- \sum_{j \in \setN} \nabla U_\varepsilon (x_0(t)  - x_j(t) )   \\
\dot{v}_j(t)  &=- \nabla U_\varepsilon (x_j(t)  - x_0(t) )  \text{ if }   j\in \setN. 
\end{cases}
 \label{eq:partic}
\end{align}
The initial condition $(x_0,v_0)$ of the tagged particle $0$ is chosen according to some $f_0 \in L^1(\R^3_x \times \R^3_v)$.
The background particles are placed spatially in $\R^3$ according to a Poisson process with intensity $N$, the velocities are chosen independently according to the probability density $g_0 \in  L^1(\R^3_v)$

\subsection*{Linear Boltzmann equations}
In an appropriate scaling limit this leads to a linear Boltzmann equation which is given by
\begin{align}\label{LBE}
\begin{cases}
\partial_t f_t(x,v) + v \cdot \nabla_x f_t(x,v) &= c Q[f_t] (x,v), \\
 \quad  \quad  \quad  \quad  \quad   \quad \ \ \ f_{t=0}(x,v) & = f_0(x,v),
\end{cases}
\end{align}
where $c$ is a parameter which is the inverse of the mean free path of the microscopic particles and it represents the rate of collisions.
The collision operator $Q$ is a linear operator describes the interactions of the particles with the surrounding medium and it is defined by $Q:= Q^+ - Q^-$, where the gain term $Q^+$ is given by
\begin{align*}
Q^+[f_t](x,v)=  \int_{\mathbb{S}^2} \int_{\mathds{R}^3}   f_t(x,v') g_0(\bar{v}')\mathcal{B}(v-\bar{v}, \omega )\, \d \bar{v}\, \d {\omega},
\end{align*}
where the pre-collisional velocities $v'$ and $\bar{v}'$ are given by $v'=v+\omega \cdot (\bar{v} -v)\omega$ and $\bar{v}'=\bar{v}-\omega \cdot (\bar{v} -v)\omega$ respectively and the loss term $Q^-$ is given by
\begin{align*}
Q^-[f_t](x,v)=  f_t(x,v) \int_{\mathbb{S}^2} \int_{\mathds{R}^3} g_0(\bar{v}) \mathcal{B}(v-\bar{v}, \omega )\, \d \bar{v}\,  \d {\omega}.
\end{align*}
Here $\mathcal{B}$ is the collision kernel which is derived from a two particle interaction and has the following form
\begin{align}\label{CK}
\mathcal{B}(v_1-v_2, \Theta)
&=
\frac{2^{\frac{2}{n-1}}}{a^2} \left[ |v_1-v_2|^{ \frac{n-5}{n-1} }
+
\mathcal{O}( |v_1-v_2|^{ \frac{n-7}{n-1} }) \right]  \Theta (\sin \Theta)^{-1},
\end{align}
for $n \in (2,5]$ and  $|v_1 - v_2|  \to \infty$.
Here, $a$ is a constant defined by $a:= \int_{0}^{1} \frac{\d w}{ \sqrt{1-(1-w)^{n-1}}}$, for $n\in (3,5]$. See section \ref{sec:kernel} for the derivation of this formula. Furthermore, we are making the assumption that the initial distribution of the background particles, $g_0: \R^3 \to \R$, is radial, continuous and has the following form
\begin{align}\label{g0}
g_0(v) :=
\begin{cases}
C_1, & |v|\le \bar{R},\\
C_2 |v| ^{-q}, & |v| \ge \bar{R},
\end{cases}
\end{align}
where $C_1, C_2$, $\bar{R}$ are positive, real constants and $q\in (4,\frac{4n}{n-1})$, $n\in (3,5]$ such that $\int_{\R^3} g_0(v) \d v  =1$. Notice that the solutions of the linear Boltzmann equation $\eqref{LBE}$ conserves mass but not energy.
The derivation of \eqref{LBE} is given  in \cite{Matthies2018} for the case of hard sphere dynamics, for finite, fixed times without any error estimates. For the derivation of the linear Boltzmann equation from a Lorentz gas particle system see \cite{Gal69}. Also in \cite{Egginton2017, Desv99} the linear Boltzmann equation \eqref{LBE} is derived from a long range particle evolution. For variants and further details see also \cite{Stone2017, Matthies2018a} and for a related model \cite{Nota2019, DesvR01}. Various scaling limits can be considered for systems with long-range potentials \cite{Nota2021}. For the derivation of the linear Boltzmann equation without cut-off from a particle system where particles interact by a infinite range potential see \cite{Ayi2017}.
\subsection*{Hydrodynamic limits}
There is a widely developed theory to derive continuum equations, such as the Navier-Stokes, Euler and Heat equations as scaling limits from the Boltzmann equation. For some review of methods and the substantial literature see e.g. \cite{cercignani94, Saint-Raymond2009, Slemrod2013, Gallagher2019}. For particular relevance for us is the work  \cite{Mellet2011}, where Mellet, Mischler and Mouhot provide a proof for the convergence of the linear Boltzmann equation to a fractional diffusion equation, where the equilibrium distribution function is a heavy-tailed distribution with infinite variance. For a recent review on the hydrodynamic limit toward Fractional Diffusion equations see \cite{Dechi2023}. We also refer to \cite{CdPG23} for the derivation of the fractional Porous Medium equation as the hydrodynamic limit of a random particle system with long range interactions.
\subsection*{Fractional diffusion equations}
The fractional diffusion equation of order $\gamma <2$ is given by
\begin{align}\label{FDE}
\begin{cases}
 \partial_{\tau} \rho + \kappa (-\Delta _x)^{\frac{\gamma}{2}} \rho &=0 \quad \quad \mathrm{in}\ (0,\infty)\times \mathds{R}^3,\\
\quad \quad \quad \quad \quad \rho (0, \cdot)&= \rho_0 \quad \quad \mathrm{in} \ \mathds{R}^3.
 \end{cases}
 \end{align}
 The function $\rho = \rho (\tau,x) $ represents the density  $\int_{\mathds{R}^3} f(\tau,x,v) \d v$.
 The fractional operator is defined as $(-\Delta _x)^{\frac{\gamma}{2}} \rho := \mathcal{F}^{-1}\big( |k|^{\gamma} \mathcal{F}(\rho)(k) \big)$, with $\mathcal{F}$ the Fourier transform in the space variable, and it is non-local operator since it is defined through the Fourier transform. The fractional operator of the order of $\frac{\gamma}{2} \in (0,1)$ generates a rotationally symmetric $\gamma$-stable L\'{e}vy process, see e.g. \cite{Schilling16} or \cite{Bruce-Grigolini97}. For a reference for the state of the art of the fractional diffusion equations we refer to \cite{Evangelista-Lenzi18}.

In this article, our main theorem concerns the convergence from a deterministic microscopic system to  a fractional diffusion \eqref{FDE} continuum limit. In this microscopic system, the background particles are distributed in position according to a spatial homogeneous Poisson process with intensity $N$ and in velocity according to a fat-tailed distribution $g_0$ as in \eqref{g0}. The derivation of the fractional diffusion equation depends on the interaction of two ingredients: the decay assumptions of this distribution $g_0$ and the specific choice  particle interaction given in \eqref{eqn:U}. 

\begin{figure}[H]
\begin{center}
\begin{tikzpicture}[
inner sep=5mm,
roundnode/.style={circle, draw=green!60, fill=green!5, very thick, minimum size=1cm},
squarednode/.style={rectangle, draw=red!60, fill=yellow!5, very thick, minimum size=1cm},
title/.style={font=\LARGE\scshape,node distance=16pt, text=black!40, inner sep=1cm},
]

 \node[]      (maintopic)                              {};
\node[squarednode,align = center]        (uppercircle)       [above=of maintopic] {\textbf{Microscopic description} \\
Rayleigh gas particle system };
\node[squarednode,align = center ]      (rightsquare)       [right=of maintopic] {\textbf{Mesoscopic description} \\
Linear Boltzmann equations };
\node[ squarednode, align = center]        (lowercircle)       [below=of maintopic] {\textbf{Macroscopic description}\\
Fractional diffusion equations};
\draw[->, align = center] (uppercircle.south) -- node[left] {$N \varepsilon ^{d-1}\gg 1$,\\  $N \varepsilon ^{d}\ll 1$ } (lowercircle.north);
\draw[->, align = center] (uppercircle.east) --  node[right] {Boltzmann-Grad limit\\ $N \gg 1, N \varepsilon ^{d-1} = c$} (rightsquare.north);
\draw[->,align = center] (rightsquare.south) -- node[right] { Hydrodynamic limit \\ $c \gg1$ }  (lowercircle.east);
\end{tikzpicture}
\end{center}
 \caption{Description of the problem}
    \label{Fig2}
\end{figure}
\subsection*{Statement of main results}
In this paper we are providing several extensions to \cite{Matthies2018}: We consider short range potential dynamics, instead of hard sphere dynamics. We are working on the phase space $\mathds{R}^3 \times \mathds{R}^3$ instead of $\mathds{T}^3 \times \mathds{R}^3$. We are extending the time-scale of the derivation with quantitative error estimates, and lastly, we describe long-term fractional diffusion behaviour. The main results we are going to prove are the following two theorems. The first one is a theorem for the derivation of the linear Boltzmann equation from a Rayleigh gas particle system.
\begin{theorem}\label{thm1} Let $f_0 \in L^1( \mathds{R}^3 \times \mathds{R}^3)$ the initial distribution of the tagged particle with $f_0(x,v) (1+|v|) \in L^1(\R^3\times \R^3)$ and $g_0 (v) (1+|v|) \in L^1(\mathds{R}^3)$, where $g_0$ is the distribution of the background particles and let $t \in [0,T_{\varepsilon}]$,
$c T_{\varepsilon} = \varepsilon^{\frac{4 m}{3} - \frac{2}{9}},$
$0<m < \frac{1}{6}$. Then the distribution of the tagged particle $\hat{f}^N_t$ converges in the $L^1$-norm to the solution of the linear Boltzmann equation \eqref{LBE} $f_t$, in the Boltzmann-Grad limit $N \varepsilon^2 = c$, for a time $T_{\varepsilon}$ which is diverging with $\varepsilon \to 0$.
That is, there exists $\varepsilon_0>0$  such that for every $\varepsilon >0$, with $ \varepsilon < \varepsilon _0 $ such that for any $t\in [0,T_{\varepsilon}]$
the error can be estimated by
\begin{align*}
\| \hat{f}^N_t(x,v) - f_t (x,v) \|_{L^1(\mathds{R}^3 \times \mathds{R}^3)} & \le
C \varepsilon ^{m}.
\end{align*}
\end{theorem}

We compare our result with those of \cite{Gallagher2013} and \cite{Matthies2018}. In \cite{Gallagher2013}, the authors rigorously derive the Boltzmann equation from a system of hard spheres or Newtonian particles interacting via a short-range potential. In \cite{Matthies2018}, the linear Boltzmann equation is derived in the Boltzmann–Grad limit from a Rayleigh gas particle system. The primary distinction lies in the methodological approach: \cite{Gallagher2013} relies on the BBGKY hierarchy, whereas both our work and \cite{Matthies2018} use semigroup methods for the comparison of probabilities  of genuine collision histories.

A key difference also concerns the time scale for which convergence to the Boltzmann equation is established. In \cite{Gallagher2013}, convergence holds for short times. In contrast, \cite{Matthies2018} proves convergence for arbitrarily long but finite, fixed times. Our result goes further, establishing convergence for times that diverge in the Boltzmann–Grad limit.

The next theorem establishes the rigorous justification  of a fractional diffusion equation from a Rayleigh gas particle system with the background particles being initially distributed according to a spatial homogeneous Poisson process. In the macroscopic limit, the trajectory of the tagged particle is defined by $\Xi (\tau) := x(\tilde{\beta} \tau) \in \mathds{R}^3$. The distribution of $\Xi (\tau)$ is given by $\hat{f}^N( \tilde{\beta} \tau, x, v)$, see Lemma \ref{timescale} where the admissible choice of $\tilde{\beta}$ is given, which is $\tilde{\beta} = \epsilon^{1-{\gamma}}$, and the relation of this with the parameter $c$.
\begin{theorem}\label{thm2}
Let $f_0 \in L^1 \cap L^2_{F^{-1}}(\R^3 \times \R^3)$ be the initial distribution of the tagged particle and consider the background particles that are Poisson distributed according to the density $N g_0(v) \d x \d v$. Define $\gamma(q):= \frac{(q-4)(n-1)}{4} +1$, for $4<q<\frac{4n}{n-1}$ and $n\in (3,5]$. Let $F$ be  the unique positive   equilibrium distribution of the linear Boltzmann equation. Let  $\rho(\tau,x)$ be the solution of the fractional diffusion equation \eqref{FDE}.
Then the distribution of the tagged particle $\hat{f}^N(\epsilon^{1-{\gamma}} \tau, x, v)$ converges in $L^{\infty}(0,T;L^2_{F^{-1}} \cap L^1(\R^3 \times \R^3))$-norm to $\rho(\tau,x) F(v)$, i.e. for any test function $\tilde{\varphi} \in L^1(0,T; L^2_{F^{-1}} \cap L^{\infty}(\R^3 \times \R^3))$
\begin{align*}
\int_{0}^{T} \int_{\R^3 \times \R^3} | (\hat{f}^N ( \epsilon^{1-{\gamma}} \tau,  x,v) - \rho(\tau,x) F(v) ) \cdot \tilde{\varphi}(\tau,x,v) | \d x  \d v \d \tau \to 0,
\end{align*}
in the limit $N \to \infty$, with $c = N \varepsilon^2 \to \infty.$
\end{theorem}

\subsection*{Plan of paper}
In the next section we provide an overview of the main ingredients for the proofs of both theorems.
In section \ref{sec:kernel} we provide a general introduction on how to obtain the collision kernel $\mathcal{B}$. Furthermore we find the formula of the collision kernel $\mathcal{B}$ for the case $n \in (3,5]$, for large relative velocities. In section \ref{pthm1} we prove Theorem \ref{thm1} by adapting and extending \cite{Matthies2018} with quantitative error estimates and by considering short range potential dynamics on the phase space $\R^3 \times \R^3$. The proof of Theorem \ref{thm2} is given in section \ref{pthm2}. We adapt results from \cite{Mellet2011} and combine them with Theorem \ref{thm1}.
\section{Collision Trees}
The main ingredient to prove Theorem \ref{thm1} is the representation of the dynamics through collision trees. This will allow us to estimate the difference between the idealised behaviour  as predicted by the Boltzmann equation and the empirical, genuine particle dynamics.  We use similar structures as in \cite{matthies10,matt12} developed for nonlinear gainless dynamics. In \cite{Matthies2018, Matthies2018a} these were adapted to a Rayleigh gas of hard-sphere dynamics, such that collision trees have height 2 and essentially become lists. Very detailed proofs and explanations in the Rayleigh cases are given in \cite {Stone2017}. 

A collision tree or collision history $\Phi$ is a set that includes the collisions that the tagged particle experiences. More precisely, it includes the initial position and velocity of the tagged particle $(x_0, v_0) \in \mathds{R}^3 \times \mathds{R}^3$ along with a list of collisions that the tagged particle experiences. Each collision is denoted by $(t_j, \nu _j, v_j) \in (0,T] \times \mathbb{S}^2 \times \mathds{R}^3$, where $t_j$ is the time that the $j$-th collision happens, $\nu _j$ is the collision parameter and $v_j$ is the incoming velocity of the background particle.
\begin{definition}
The set of all collision trees $\mathcal{MT}$ is defined by
\begin{align*}
\mathcal{MT} :=\{ ((x_0, v_0), (t_1, \nu_1, v_1),...,(t_n,\nu_n, v_n)  ) : & \, (x_0,v_0)\in \mathds{R}^3 \times \mathds{R}^3,\\  &  \, t_i\in[0,T],\ \nu_i\in \mathbb{S}^2 ,\ v_i\in \mathds{R}^3, \ n\in \mathds{N}_0 \}.
\end{align*}
 For a tree $\Phi \in \mathcal{MT}$ is defined the function $n(\Phi)$ to be the number of collisions in this tree and for $n\ge 1$ define $\bar{\Phi}$ as the collision history identical to $\Phi$ but with the final collision removed.
Furthermore, we define the maximum collision time $\tau \in [0,T]$ as
\begin{equation*}
\tau = \tau(\Phi) :=
\begin{cases}
0, &n(\Phi)=0\\
\max\limits_{1\le j\le n} t_j, & \mathrm{else},
\end{cases}
\end{equation*}
and we define the marker of the final collision as
$$(\tau, \bar{\nu}, \bar{v}):=(t_n, \nu_n, v_n).$$
\end{definition}
\definition For a collision history $\Phi \in \mathcal{MT}$, the maximum velocity $\mathcal{V}(\Phi) \in [0, \infty)$ in the history is defined as
$$\mathcal{V}(\Phi):= \max \left\{ \max_{j=1,...,n(\Phi)} |v_j|, \max_{s\in [0,T]} |v(s)|  \right\},$$
where $v_j$, $j=1,...,n(\Phi)$ is the velocity of each background particle in the collision history $\Phi$ and $v(s),$ $s \in [0,T]$ is the velocity of the tagged particle in the same history for different times. We denote the smallest  relative speed of the $n(\Phi)$ collisions as
$$ \iota (\Phi):= \min_{j=1,...,n(\Phi)} |v(t_j^-)-v_j|.   $$

We associate probabilities of finding a given collision tree. We first define the idealised distribution.
Let $\Phi \in \mathcal{MT}$, then $P_0(\Phi)$ is zero unless $\Phi$ involves no collisions, in which case $P_0(\Phi)$ is given by initial distribution $f_0(x_0,v_0)$. $P_t(\Phi)$ remains zero until $t=\tau$ when there is an instantaneous increase to a positive value depending on $P_\tau(\bar{\Phi})$ and the final collision in $\Phi$. For $t>\tau$, $P_t(\Phi)$ decreases at a rate that is obtained by considering all possible collisions. This can be expressed as
\begin{equation}\label{eq-id}
\begin{cases}
\partial_tP_t(\Phi)&=c [Q_t^+[P_t](\Phi) - Q^-_{t}[P_t](\Phi)],\\
P_0(\Phi) & = f_0 (x_0,v_0)\mathds{1}_{n(\Phi)=0},
\end{cases}
\end{equation}
where
\begin{equation} \label{eq:Q+}
Q_t^+[P_t](\Phi)=
\begin{cases}
\mathds{1}_{t=\tau(\Phi)}P_t(\bar{\Phi})g_0(\bar{v})  \mathcal{B}(v^{\varepsilon}(\tau^-) - \bar{v} ,\omega) & n(\Phi)>0, \\
0 & n(\Phi)=0,
\end{cases}
\end{equation}
and
\begin{equation}\label{eq:Q-}
Q^-_{t}[P_t](\Phi)= P_t(\Phi)\int_{\mathds{R}^3} \int_{\mathbb{S}^2} g_0(\bar{v})  \mathcal{B}(v^{\varepsilon}(\tau) - \bar{v} ,\omega)\, \d{\omega}\, \d \bar{v},
\end{equation}
where $\bar{v}$ the velocity of the final colliding background particle and $v^{\varepsilon}(\tau^-)$ is the velocity of the tagged particle before the collision. The idealised equation should be thought as equivalent of the linear Boltzmann equation but written on the space $\mathcal{MT}$ instead of $\mathds{R}^3 \times \mathds{R}^3$. The evolution equation \eqref{eq-id} is well-posed by the same arguments as in \cite[Thm 3.1]{Matthies2018}.
Before we can introduce the empirical distribution related to the particle model, we need to introduce some further notation.
\begin{definition}
 A history $\Phi \in \mathcal{MT}$ is called non-grazing if
\begin{equation*}
\min_{1\le j \le n(\Phi)} \nu_j \cdot ( v(t_j^-) - v_j)>0.
\end{equation*}
This means that all the collisions in the history $\Phi$ are non-grazing, i.e., the tagged particle and each background particle $j$ are not flying parallel for long time.
\end{definition}
\begin{definition} We say that a collision history $\Phi \in \mathcal{MT}$ is free from initial overlap at diameter $\varepsilon$ if initially the tagged particle is at least $\varepsilon$ away from the centre of each background particle. That is to say, for all $j=1,...,n(\Phi),$ $$|x_0-x_j|>\varepsilon.$$
Define $S(\varepsilon) \subset \mathcal{MT}$ to be the set of all histories that are free from initial overlap at radius $\varepsilon$.
\end{definition}
\begin{definition} A collision history $\Phi \in \mathcal{MT}$ is called re-collision free at diameter $\varepsilon$ if for all $j=1,...,n(\Phi)$ and for all $t\in [0,T] \setminus (t_j, t_j + \tau_j^*)$, $$ |x_0(t) - x_j(t)|>\varepsilon.$$
That is, if the tagged particle and a background particle $j$ collide at time $(t_j, t_j + \tau_j^*)$ then the tagged particle has not previously collided and will not re-collide with the background particle $j$ up to time $T$.
Define the set
\begin{equation*}
R(\varepsilon):= \{ \Phi \in \mathcal{MT} :\ \Phi \text{\ is\ re-collision\ free\ at\ diameter}\ \varepsilon \}.
\end{equation*}
Here $\tau_j^*$ is the scattering time, i.e., the time it takes for the collision to take place, in the short range potential model. We also define the maximal collision time in a tree as
\begin{align*}
    \vartheta(\Phi)= \max_{j=1,...,n(\Phi)} \tau_j^*.
\end{align*}
\end{definition}
\begin{definition} A collision history $\Phi \in \mathcal{MT}$ is called binary at diameter $\varepsilon$ if for all $j=1,...,n(\Phi)$ the collision times $(t_j, t_j + \tau_j^*)$
are pairwise disjoint.

That is, the tagged particle will only interact with  at most a single background particle at any given time. 
Define the set
\begin{equation*}
B (\varepsilon):= \{ \Phi \in \mathcal{MT} :\ \Phi \text{\ is\ binary\ at\ diameter}\ \varepsilon \}.
\end{equation*}
\end{definition}

\begin{definition} \label{1.3.6.} The set of good histories $\mathcal{G}(\varepsilon)$ of diameter $\varepsilon$ is defined by
\begin{align*}
\mathcal{G}(\varepsilon) := \{ \Phi \in \mathcal{MT} : \ n(\Phi) \le M(\varepsilon),\ \mathcal{V}(\Phi) <&V(\varepsilon), \ \vartheta(\Phi) > I(\varepsilon), \  \iota(\Phi) > J(\Phi),\\
&\Phi \in R(\varepsilon) \cap S(\varepsilon) \cap B(\varepsilon)\ \mathrm{and} \ \Phi  \text{ is\ non-}\mathrm{grazing} \},
\end{align*}
for any strictly increasing $I, J: (0,\infty) \to [0, \infty)$ such that $\lim_{\varepsilon \to 0}I(\varepsilon) = \lim_{\varepsilon \to 0}J(\varepsilon) = 0$ and any decreasing functions $V, M : (0,\infty) \to [0, \infty)$ such that $\lim_{\varepsilon \to 0}V(\varepsilon) = \lim_{\varepsilon \to 0}M(\varepsilon) = \infty.$
\end{definition}
The functions $I$,  $J$, $V$ and $M$ will be defined in subsection \ref{largeT}.

\section{The collision kernel \texorpdfstring{$\mathcal{B}$}{B}}\label{sec:kernel}
\subsection{Deriving the collision kernel \texorpdfstring{$\mathcal{B}$}{\mathcal{B}}}
We will require a detailed understanding of the collision kernel $\mathcal{B}$ in \eqref{LBE}.
In \cite[Section 8.3]{Gallagher2013} an implicit formula for the collision kernel $\mathcal{B}$ is derived,  the process is as below. First, for fixed $x_1$ they calculate the surface measure on the sphere $ \{ y\in \R^3\ : \ |y-x_1| = \epsi \}$ which is defined by $d \sigma_1$ and also $x_2$ belongs in this sphere. Here $\epsi$ is the radius of the sphere. Now parametrise the sphere by ($\alpha , \psi$) $\in [0,\pi] \times [0,2\pi]$. Also, $\alpha$ is the angle between $\delta v = v_1-v_2$ and $\delta x = x_1-x_2$. Then
\begin{align*}
\d \sigma_1 = \epsi ^2 \sin \alpha \d \alpha  \d\psi.
\end{align*}
The direction of the apse line is $\omega = (\Theta, \psi) \in [0,\frac{\pi}{2}] \times [0,2\pi]$. Now we define the surface measure on the unit sphere by $\d \omega$. Then
\begin{align}\label{8.1.}
\d \omega = \sin \Theta  \d\Theta  \d\psi.
\end{align}
Also, by \cite[Definition 8.2.1]{Gallagher2013} $\alpha$ is defined by $\mathcal{J}_0 = \frac{|(x_1-x_2)\times(v_1-v_2)|}{\varepsilon |v_1-v_2|} = \sin \alpha$, where $|(x_1-x_2)\times(v_1-v_2)|$ is the magnitude of the cross product defined by $|(x_1-x_2)\times(v_1-v_2)| := |x_1-x_2||v_1-v_2| \sin \alpha$. Thus, taking the inner product of $x_1 - x_2$ and $v_1-v_2$,
\begin{align*}
(x_1-x_2)\cdot(v_1-v_2) := |x_1-x_2||v_1-v_2| \cos \alpha.
\end{align*}
Then, the above computations give
\begin{align*}
\frac{1}{\epsi} (x_1-x_2)\cdot(v_1-v_2) \d \sigma_1 
=& \epsi ^{2} |v_1-v_2| \cos \alpha \sin \alpha  \d \alpha \d \psi \\
=&  
   \varepsilon ^2 |v_1-v_2|  \mathcal{J}_0 \d \mathcal{J}_0\d\psi \\
=&  \varepsilon ^{2} |v_1-v_2|  \mathcal{J}_0 \partial _{\Theta} \mathcal{J}_0\d\Theta  \d\psi.
\end{align*}
The last equality holds true when $\mathcal{J}_0 \in [0,1)$ by \cite[Lemma 8.3.1]{Gallagher2013}.
Also, since
\begin{align*}
\d \Theta \, \d\psi  = (\sin \Theta)^{-1}   \d \omega
\end{align*}
we get
\begin{align*}
\frac{1}{\epsi} (x_1-x_2)\cdot(v_1-v_2) \d \sigma_1 =  \epsi ^{2} |v_1-v_2|  \mathcal{J}_0 \partial _{\Theta} \mathcal{J}_0 (\sin \Theta)^{-1}   \d\omega.
\end{align*}

\begin{definition}\label{def3.1} The scattering cross-section 
is defined for $|v_1-v_2|>0$ 
by $\mathcal{J}_0 \partial _{\Theta} \mathcal{J}_0 (\sin \Theta)^{-1} $. Therefore, the collision kernel $\mathcal{B}$ is defined by
\begin{align}\label{8.3}
\mathcal{B}(v_1-v_2, \Theta) : = &
|v_1-v_2|  \mathcal{J}_0 \partial _{\Theta} \mathcal{J}_0 (\sin \Theta)^{-1}.
\end{align}
By abuse of notation we may write $\mathcal{B}(v_1-v_2, \Theta)=\mathcal{B}(v_1-v_2, \omega)$. Here, $|v_1-v_2|$ is the kinetic part of the collision kernel and the remaining is the angular part of the collision kernel.
\end{definition}
\subsection{Formula of the collision kernel \texorpdfstring{$\mathcal{B}$}{\mathcal{B}} }
This subsection contains computations of finding the formula of the collision kernel $\mathcal{B}$ for two particles that are interacting via the short range potential introduced in the particle dynamics subsection in the introduction. We are doing so by finding the inverse of the function $\Theta$ with respect to $\mathcal{J}_0$.
Then we apply the above Definition \ref{def3.1} and analyse the behaviour for large relative velocities.  We start by observing that  $\mathcal{B}$ bounded for bounded velocities 
\begin{lemma} \label{lem:Bbounded}
    For any $n \in (3,5]$, the collision 
    kernel $\mathcal{B}$ is bounded  
    on bounded sets for $(v_1-v_2, \Theta) \in \R^3 \times [0,\pi/2]$.
\end{lemma}
\begin{proof}
    Following \cite{Gallagher2013}[Lemma 8.3.1], we observe that for fixed $|v_1-v_2|$, that  $\mathcal{J}_0 (\Theta) $ and $\partial _{\Theta} \mathcal{J}_0(\Theta)$ are well-defined, bounded and continuous for all $\Theta \in [0, \pi/2]$, such that  $\mathcal{J}_0 (0)=0 $ and
    $\partial _{\Theta} \mathcal{J}_0(0) < \infty$, which implies continuity of $\mathcal{B}$ for $|v_1-v_2|>0$. A simple inspection of the for $\Theta(|v_1-v_2|^2,\mathcal{J}_0) $ and $\frac{\partial \Theta}{\partial \mathcal{J}_0} (|v_1-v_2|^2,\mathcal{J}_0) $ yields that in the limit $|v_1-v_2|^2 \to 0$, the terms simplify to $\mathcal{J}_0 (\Theta) = \sin( \Theta) $ and $\partial _{\Theta} \mathcal{J}_0(\Theta)= \cos(\Theta).$ 
Hence $\mathcal{B}$ as given in \eqref{8.3} for our interaction potentials is also continuous at $|v_1-v_2|=0$ such that it is bounded for  
    bounded relative velocities $|v_1 - v_2|$ uniformly  in  $\Theta \in [0,\pi/2]$.  
\end{proof}
The main proposition of this section is the following.
\begin{proposition}\label{P3.2}
For any $n \in (3,5]$, the collision kernel $\mathcal{B}$ is bounded for large relative velocities $|v_1 - v_2|  \to \infty$ and 
$\mathcal{B}$ takes the form
\begin{align}\label{2.4}
\mathcal{B}(v_1-v_2, \Theta)
&=
\frac{2^{\frac{2}{n-1}}}{a^2}  \bigg[ |v_1-v_2|^{ \frac{n-5}{n-1}}  +  \mathcal{O}( |v_1-v_2|^{ \frac{n-7}{n-1} })\bigg]  \Theta (\sin \Theta)^{-1},
\end{align}
where $a$ is a fixed number.
Furthermore, it holds true that
\begin{align*}
\bigg| \mathcal{B}(v_1 - v_2, \Theta) - \frac{2^{\frac{2}{n-1}}}{a^2}  \Theta (\sin \Theta)^{-1} |v_1-v_2|^{ \frac{n-5}{n-1} } \bigg|
&=
\frac{2^{\frac{2}{n-1}}}{a^2} \Theta (\sin \Theta)^{-1}  \mathcal{O}( |v_1-v_2|^{ \frac{n-7}{n-1} }) \\
& = : \mathrm{error}(\Theta, |v_1 - v_2|).
\end{align*}
\end{proposition}
We note here that when we use the Landau notation in \eqref{2.4} for large relative velocities, we mean $\mathcal{O}(|v_1 - v_2|^{ \frac{n-7}{n-1} })$, as $|v_1 - v_2|\to \infty$ where we mean the standard Landau notation 
 $f= \mathcal{O}(g)$ as $x\to \infty$ if there exists a constant $C$ such that $|f(x)| \le C |g(x)|$ for all $x$ sufficiently large. The Proposition immediately implies upper and lower bounds. 
 \begin{corollary}\label{CorCOB}(Upper and Lower bound for the collision kernel $\mathcal{B}$)
For $n \in (3,5]$ and large $|v_1 - v_2| >0$ 
we have the following bounds.
\begin{align} \label{eqn:lowbdB}
 \frac{1}{2} \frac{2^{\frac{2}{n-1}}}{a^2}  |v_1-v_2|^{ \frac{n-5}{n-1} } \Theta (\sin \Theta)^{-1} &
\le
\mathcal{B}(v_1-v_2, \Theta)
\le
3 \frac{2^{\frac{2}{n-1}}}{a^2}  |v_1-v_2|^{ \frac{n-5}{n-1} } \Theta (\sin \Theta)^{-1}. 
\end{align}
\end{corollary}
\begin{proof}
Immediate by the way that the collision kernel $\mathcal{B}$ is defined.
\end{proof}
 
In order to give a proof of this proposition, we first need to prove Proposition  which is about finding a formula of the function $\mathcal{J}_0$ with respect to $\Theta$. Then we can compute the derivative of $\mathcal{J}_0$ with respect to $\Theta$ and then we use \eqref{8.3} to give the formula \eqref{2.4} of the collision kernel $\mathcal{B}$. We first give the proof of a lemma needed for Proposition \ref{P3.4.}.

\begin{lemma}\label{L3.5.} The derivative of $u_0$ with respect to $\mathcal{J}_0$ is zero, for $\mathcal{E}_0$ fixed and $\mathcal{J}_0 = 0$ near $u_0 (\mathcal{J}_0) = 1.$
\end{lemma}
\begin{proof}
With \eqref{1.2} we define
\begin{align*}
F(u_0, \mathcal{J}_0) := J_0^2 \bigg( 1 + \frac{ \mathcal{E}_0 }{2} \bigg)^{\frac{2}{n-1}} u_0^2 + u_0^{n-1} -1.
\end{align*}
We observe that $F(1,0) = 0$ and
\begin{align*}
\frac{\partial F}{\partial u_0} (u_0 , \mathcal{J}_0) = 2J_0^2 \bigg( 1 + \frac{ \mathcal{E}_0 }{2} \bigg)^{\frac{2}{n-1}} u_0 +(n-1) u_0^{n-2},
\end{align*}
 so $ \frac{\partial F}{\partial u_0} (1 , 0 ) = n-1 \neq 0$, for $n\neq 1. $
Also
\begin{align*}
\frac{\partial F}{\partial \mathcal{J}_0} (u_0 , \mathcal{J}_0) = 2J_0 \bigg( 1 + \frac{ \mathcal{E}_0 }{2} \bigg)^{\frac{2}{n-1}} u_0^2,
\end{align*}
 so $ \frac{\partial F}{\partial \mathcal{J}_0} (1 , 0 ) = 0. $
Therefore, by the Implicit Function Theorem we have
\begin{align*}
& \frac{\partial F}{\partial u_0} (1 , 0 ) \frac{\d u_0}{ \d \mathcal{J}_0} +\frac{\partial F}{\partial \mathcal{J}_0} (1 , 0 ) \frac{ \d \mathcal{J}_0}{\d \mathcal{J}_0} = 0
\end{align*}
which is equivalent to
\begin{align*}
\frac{\d u_0}{ \d \mathcal{J}_0}\bigg|_{\mathcal{J}_0 =0} = - \bigg(\frac{\partial F}{\partial u_0} (1 , 0 )\bigg)^{-1} \frac{\partial F}{\partial \mathcal{J}_0} (1 , 0 )
= 0,
\end{align*}
as required.
\end{proof}

\begin{proposition}\label{P3.4.}
For any $n \in (3,5]$ and for large relative velocities $|v_1-v_2|>0$, 
the formula of $\mathcal{J}_0(\Theta)$ is given by
\begin{align}\label{2.8.}
\mathcal{J}_0 (\Theta)
=
\frac{2^{\frac{1}{n-1}}}{a}\bigg(
 1
 + \mathcal{O}( \mathcal{E}_0^{-\frac{1}{n-1}} ) \bigg)
 \Theta \mathcal{E}_0^{-\frac{1}{n-1}}
 ,
\end{align}
where $\mathcal{E}_0 = |v_1 - v_2|^2$ and $a$ is a fixed real number.
\end{proposition}
\begin{proof}
By \cite[Chapter II]{cerci88}, the deviation angle $\Theta$ is given by the formula
\begin{align}\label{2.1.}
\Theta = \arcsin{J_0} + \int_{\lambda}^{x_0} \frac{\d x}{\sqrt{1-x^2-(\frac{x}{b})^{n-1}}},
\end{align}
where $x_0$, $\lambda$ and $b$ are  given by
\begin{align*}
1-x_0^2-\bigg(\frac{x_0}{b} \bigg)^{n-1} =0, \quad \lambda = J_0 \bigg( 1+ \frac{2}{ \mathcal{E}_0} \bigg)^{\frac{1}{2}}, \quad b = J_0 \bigg( 1 + \frac{ \mathcal{E}_0 }{2} \bigg)^{\frac{1}{n-1}},
\end{align*}
where $\mathcal{E}_0 := |v_1 - v_2|^2$. Following  \cite[Section 8.2]{Gallagher2013}, $\mathcal{J}_0 $ is defined by \[\mathcal{J}_0 : = \frac{|(x_1-x_2)\times(v_1-v_2)|}{\varepsilon |v_1-v_2|} =: \sin \alpha,\] where $|(x_1-x_2)\times(v_1-v_2)|$ is the magnitude of the cross product defined by $|(x_1-x_2)\times(v_1-v_2)| := |x_1-x_2||v_1-v_2| \sin \alpha$.

We use the change of variable $ u = \frac{x}{b}$, and \eqref{2.1.} becomes
\begin{align*}
\Theta = \arcsin{J_0} + \int_{\frac{\lambda}{b}}^{u_0} \frac{ b \, \d u}{\sqrt{1-b^2u^2 -u^{n-1}}},
\end{align*}
where $u_0$ is such that
\begin{align}\label{1.1}
1-b^2u_0^2 -u_0^{n-1} =0.
\end{align}
Again by the change of variable $u = u_0 y$ we have
\begin{align*}
\Theta = \arcsin{J_0} + \int_{\frac{\lambda}{bu_0}}^{1} \frac{ b u_0\, \d y}{\sqrt{1-b^2u_0^2 y^2 -(u_0y)^{n-1}}}
\end{align*}
and by \eqref{1.1} we get
\begin{align*}
\Theta = \arcsin{J_0} + \int_{\frac{\lambda}{bu_0}}^{1} \frac{ b u_0\, \d y}{\sqrt{1-(1-u_0^{n-1}) y^2 -(u_0y)^{n-1}}}.
\end{align*}
Here by equation \eqref{1.1} and by the definition of $b$ we get
\begin{align}\label{1.2}
 J_0^2 \bigg( 1 + \frac{ \mathcal{E}_0 }{2} \bigg)^{\frac{2}{n-1}} u_0^2 = 1 -u_0^{n-1},
\end{align}
where we note that when $\mathcal{J}_0 =0$, then $u_0 =1$.

Now, we define
\begin{align*}
\Sigma(\mathcal{J}_0) := \arcsin{J_0} + \int_{\frac{\lambda}{bu_0}}^{1} \frac{ b u_0\, \d y}{\sqrt{1-(1-u_0^{n-1}) y^2 -(u_0y)^{n-1}}}
\end{align*}
and we take the derivative of $\Sigma (\mathcal{J}_0)$ with respect to $\mathcal{J}_0$
\begin{align}
\Sigma'(\mathcal{J}_0) =\frac{1}{\sqrt{1- J_0^2}}
&+\bigg( 1 + \frac{ \mathcal{E}_0 }{2} \bigg)^{\frac{1}{n-1}} u_0  \int_{\frac{\lambda}{bu_0}}^{1} \frac{ \d y}{\sqrt{1-(1-u_0^{n-1}) y^2 -(u_0y)^{n-1}}} \nonumber \\
& + J_0 \bigg( 1 + \frac{ \mathcal{E}_0 }{2} \bigg)^{\frac{1}{n-1}} u_0
\frac{\d}{\d\mathcal{J}_0} \int_{\frac{\lambda}{bu_0}}^{1} \frac{ \d y}{\sqrt{1-(1-u_0^{n-1}) y^2 -(u_0y)^{n-1}}},\label{1.3.}
\end{align}
where
\begin{align*}
\frac{\d}{\d\mathcal{J}_0} \int_{\frac{\lambda}{bu_0}}^{1} \frac{ \d y}{\sqrt{1-(1-u_0^{n-1}) y^2 -(u_0y)^{n-1}}}
& = - \frac{ - \frac{\lambda}{b} \frac{1}{u_0^2}\frac{\d u_0}{\d \mathcal{J}_0}  }{\sqrt{1-(1-u_0^{n-1}) (\frac{\lambda}{bu_0})^2 -(\frac{\lambda}{b})^{n-1}}}\\
& \quad - \frac{1}{2}\int_{\frac{\lambda}{bu_0}}^{1} \frac{ (n-1) u_0^{n-2} (y^2- y^{n-1})\frac{\d u_0}{\d \mathcal{J}_0} \d y }{(1-(1-u_0^{n-1}) y^2 -(u_0y)^{n-1})^{\frac{3}{2}}} .
\end{align*}
Therefore, the derivative of $\Sigma(\mathcal{J}_0)$ for $\mathcal{J}_0 =0 $ is 
\begin{align*}
\Sigma'(0) = 1
&+\bigg( 1 + \frac{ \mathcal{E}_0 }{2} \bigg)^{\frac{1}{n-1}}  \int_{\frac{\lambda}{b}}^{1} \frac{ \d y}{\sqrt{1-y^{n-1}}}, 
\end{align*}
where we used that when $\mathcal{J}_0 =0$, then $u_0 =1$, by \eqref{1.2} and the fact that the derivative of $u_0$ with respect to $\mathcal{J}_0$ is zero, for $\mathcal{J}_0=0$, see Lemma \ref{L3.5.} below.
By using the change of variable $w : = 1-y$ the above integral becomes
\begin{align*}
 \int_{0}^{1- \frac{\lambda}{b}} \frac{\d w}{\sqrt{1-(1-w)^{n-1}}}
  =
 \int_{0}^{1} \frac{ \d w}{\sqrt{1-(1-w)^{n-1}}}
  -
   \int_{1- \frac{\lambda}{b}}^{1} \frac{ \d w}{\sqrt{1-(1-w)^{n-1}}},
\end{align*}
where the first term in the right hand side is a finite fixed number, since $n$ is fixed, so we define this finite fixed number by $a$. Furthermore, the second term in the right hand side is equal to $\mathcal{O}(\frac{\lambda}{b})$, as $\mathcal{E}_0 \to \infty$.
Therefore we have that
\begin{align*}
 \int_{0}^{1- \frac{\lambda}{b}} \frac{ \d w}{\sqrt{1-(1-w)^{n-1}}}
  = a + \mathcal{O}(\frac{\lambda}{b}), \quad \text{as} \ \mathcal{E}_0 \to \infty.
\end{align*}
Thus, we write
\begin{align*}
\Sigma'(0) = 1
&+\bigg( 1 + \frac{ \mathcal{E}_0 }{2} \bigg)^{\frac{1}{n-1}}
 \int_{0}^{1- \frac{\lambda}{b}} \frac{ \d w}{\sqrt{1-(1-w)^{n-1}}} =
  1
+
\bigg( 1 + \frac{ \mathcal{E}_0 }{2} \bigg)^{\frac{1}{n-1}}
\bigg( a + \mathcal{O}(\frac{\lambda}{b})  \bigg).
 \end{align*}
Equivalently,
\begin{align}\label{1.5}
\Sigma'(0) =
  1
+
 a \bigg( 1 + \frac{ \mathcal{E}_0 }{2} \bigg)^{\frac{1}{n-1}}
 +   \mathcal{O}(\frac{\lambda}{b})
\bigg( 1 + \frac{ \mathcal{E}_0 }{2} \bigg)^{\frac{1}{n-1}}.
  \end{align}
Also, the second term of $\eqref{1.3.}$, for $u_0 =1 - \delta$, $\delta >0$ sufficient small, becomes
\begin{align*}
\frac{b}{\mathcal{J}_0} (1-\delta)   \int_{\frac{\lambda}{b(1-\delta)}}^{1} \frac{ \d y}{\sqrt{1-(1-(1-\delta)^{n-1}) y^2 -(1-\delta)^{n-1}y^{n-1}}}.
\end{align*}
Let us denote the function $ j_{\delta}(y) : =  1-(1-(1-\delta)^{n-1}) y^2 -(1-\delta)^{n-1}y^{n-1}$. We notice that the function $ j_{\delta}(y) $ converges to $1 - y^{n-1}$, as $\delta \to 0$, pointwise. Thus
\begin{align*}
\frac{1}{ \sqrt{ j_{\delta}(y) } } \to
\frac{1}{ \sqrt{ 1-y^{n-1} } } \quad  \text{as} \ \delta \to 0.
\end{align*}
Here, we observe that for every $n>1$ the integral of the limiting function is finite
\begin{align*}
\int_{\bar c}^{1} \frac{\d y}{ \sqrt{ 1-y^{n-1} } }
 =
\int_{0}^{1- \bar c} \frac{\d z}{ \sqrt{ 1-(1-z)^{n-1} } }
=
\int_{0}^{1-\bar c} \frac{\d z}{ \sqrt{ (n-1)z } }
=
\frac{2(1-\bar c)^{\frac{1}{2}}}{\sqrt{n-1}} < \infty,
\end{align*}
where in the first equality we used the change of variable $z = 1-y$ and in the second equality we used Taylor expansion of first order of the function $ (1-z)^{n-1} $ around $z_0 = 0$. Here, $\bar c$ is a positive fixed number, less than $1$.
The function $j_{\delta}(y)$ can be written as
\begin{align*}
j_{\delta}(y)
&= 1-(1-(1-\delta)^{n-1}) y^2 -(1-\delta)^{n-1}y^{n-1}
= 1 - y^2 + (1-\delta)^{n-1}( y^2 - y^{n-1} )
\end{align*}
and we notice that the function $ \frac{1}{ \sqrt{ j_{\delta}(y) } } $ is an increasing function with respect to $\delta$, for $n>3$. Then, by Monotone Convergence Theorem,
we take that
\begin{align*}
\int_{  \frac{\lambda}{b(1-\delta)} }^{1 } \frac{\d y}{ \sqrt{ j_{\delta}(y) } } \to
\int_{\bar c}^{1} \frac{\d y}{ \sqrt{ 1-y^{n-1} } } \quad  \text{as} \ \delta \to 0.
\end{align*}
In this step, we  invert the function $\Sigma ' (0)$ in order to find the formula of $\mathcal{J}_0$. By the formula of $\Sigma'(0)$ and the definitions of $\lambda$ and $b$ we observe that
\begin{align*}
\lim_{\mathcal{E}_0 \to \infty}\frac{\lambda}{b} \bigg( 1 + \frac{ \mathcal{E}_0 }{2} \bigg)^{\frac{1}{n-1}}
=
1,\quad \textrm{i.e.} \quad
\mathcal{O} (\frac{\lambda}{b}) \bigg( 1 + \frac{ \mathcal{E}_0 }{2} \bigg)^{\frac{1}{n-1}}
=
\mathcal{O}(1), \quad \text{as} \ \mathcal{E}_0 \to \infty.
\end{align*}
Therefore
\begin{align*}
\Sigma'(0)
=1 + \frac{a}{2^{\frac{1}{n-1}}} \bigg( 2 + \mathcal{E}_0 \bigg)^{\frac{1}{n-1}}
+ \mathcal{O}(1)
&= 1+ a \bigg[ \bigg( \frac{ \mathcal{E}_0 }{2} \bigg)^{\frac{1}{n-1}}
+
 \frac{1}{n-1} \bigg( \frac{ \mathcal{E}_0 }{2} \bigg)^{\frac{2-n}{n-1}}  \bigg]
+
\mathcal{O}(1)\\
&=
a  \bigg( \frac{ \mathcal{E}_0 }{2} \bigg)^{\frac{1}{n-1}}
+
\underbrace{ \frac{a}{n-1} \bigg( \frac{ \mathcal{E}_0 }{2} \bigg)^{\frac{2-n}{n-1}}
+
\mathcal{O}(1)}_{ = \mathcal{O}(1)},
\end{align*}
where in the second equality we used the Taylor expansion of first order with respect to $x$ around $0$, evaluated at $x=2$ of the function
$h(x) := \big( x + \mathcal{E}_0 \big)^{\frac{1}{n-1}}$.
Thus, by the inverse function theorem
\begin{align*}
( \Sigma^{-1})'(\Sigma(0))
=
\frac{1}{
a  \big( \frac{ \mathcal{E}_0 }{2} \big)^{\frac{1}{n-1}}
+ \mathcal{O}(1)}
&=
a^{-1}  \bigg( \frac{ \mathcal{E}_0 }{2} \bigg)^{-\frac{1}{n-1}} \frac{1}{
1
+a^{-1}  \big( \frac{ \mathcal{E}_0 }{2} \big)^{-\frac{1}{n-1}} \mathcal{O}(1)}\\
& =
a^{-1}  \bigg( \frac{ \mathcal{E}_0 }{2} \bigg)^{-\frac{1}{n-1}}
\big( 1 + \mathcal{O}( \mathcal{E}_0^{-\frac{1}{n-1}} )  \big),
\end{align*}
equivalently
\begin{align*}
( \Sigma^{-1})'(\Sigma(0))
 =
\frac{2^{\frac{1}{n-1}}}{a}   \mathcal{E}_0^{-\frac{1}{n-1}}
 +
\frac{2^{\frac{1}{n-1}}}{a}  \mathcal{E}_0^{-\frac{1}{n-1}}
 \mathcal{O}( \mathcal{E}_0^{-\frac{1}{n-1}} ).
\end{align*}
Hence, the formula of $\mathcal{J}_0(\Theta)$ is the primitive of the above with respect to $\Theta$. That is
\begin{align*}
\mathcal{J}_0 (\Theta)
=
\frac{2^{\frac{1}{n-1}}}{a}  \Theta \mathcal{E}_0^{-\frac{1}{n-1}}
 +
\frac{2^{\frac{1}{n-1}}}{a} \Theta \mathcal{E}_0^{-\frac{1}{n-1}}
 \mathcal{O}( \mathcal{E}_0^{-\frac{1}{n-1}} ),
\end{align*}
where $\mathcal{E}_0 = |v_1 - v_2|^2$. This completes the proof of the proposition.
\end{proof}

We are now ready to give the proof of the main proposition of this section which is about finding the formula \eqref{2.4} of the collision kernel $B$.
\begin{proof}[Proof of Proposition \ref{P3.2}]

The formula of $\mathcal{J}_0$ by \eqref{2.8.} is
\begin{align*}
\mathcal{J}_0 (\Theta)
=
\frac{2^{\frac{1}{n-1}}}{a}\bigg(
 1
 + \mathcal{O}( \mathcal{E}_0^{-\frac{1}{n-1}} ) \bigg)
 \Theta \mathcal{E}_0^{-\frac{1}{n-1}}
 ,
\end{align*}
for any $n \in (3,5]$, for large relative velocities $|v_1-v_2|>0$ and 
$a$ is a fixed real number,
where $\mathcal{E}_0 = |v_1 - v_2|^2$.
Then the derivative with respect to $\Theta$ is
\begin{align*}
\partial_{\Theta}\mathcal{J}_0 (\Theta)
=
\frac{2^{\frac{1}{n-1}}}{a}\bigg(
 1
 + \mathcal{O}( \mathcal{E}_0^{-\frac{1}{n-1}} ) \bigg)
  \mathcal{E}_0^{-\frac{1}{n-1}}.
\end{align*}
Thus by the definition of the collision kernel in dimension $3$, which is the formula \eqref{8.3} and by replacing $\mathcal{E}_0 = |v_1 - v_2|^2$ we take the formula \eqref{2.4}.
Furthermore
\begin{align*}
\bigg| \mathcal{B}&(v_1 - v_2, \Theta) - \frac{2^{\frac{2}{n-1}}}{a^2}  \Theta (\sin \Theta)^{-1} |v_1-v_2|^{ \frac{n-5}{n-1} } \bigg| \\
& =
\bigg|  \frac{2^{\frac{2}{n-1}}}{a^2} \Theta (\sin \Theta)^{-1} \bigg[  |v_1-v_2|^{ \frac{n-5}{n-1} }
+
 \mathcal{O}( |v_1-v_2|^{ \frac{n-7}{n-1} })
 -   |v_1-v_2|^{ \frac{n-5}{n-1} }  \bigg] \bigg| \\
 & =
 \bigg|  \frac{2^{\frac{2}{n-1}}}{a^2} \Theta (\sin \Theta)^{-1}
 \mathcal{O}( |v_1-v_2|^{ \frac{n-7}{n-1} })
 \bigg| \\
 & =
 \frac{2^{\frac{2}{n-1}}}{a^2} \Theta (\sin \Theta)^{-1}
 \mathcal{O}( |v_1-v_2|^{ \frac{n-7}{n-1} }),
\end{align*}
as required. The boundedness for compact sets of $v_1-v_2$ and $\Theta$ follows from continuity as \cite {Gallagher2013}, this implies global boundedness with the estimates above.
\end{proof}
Additionally, we state a lemma for the asymptotic behaviour of the angle $\Theta$, when $\mathcal{E}_0 \to \infty$.
\begin{lemma}
\begin{align*}
\lim_{\mathcal{E}_0 \to \infty} \Theta ( \mathcal{E}_0, \mathcal{J}_0) = \frac{\pi}{2} \quad \forall \mathcal{J}_0 \in (0,1].
\end{align*}
\end{lemma}
\begin{proof}
Going through the two body reduction and using the formula of \cite[Section 8.3]{Gallagher2013}:
\begin{align*}
\Theta
=
\Theta (\mathcal{E}_0, \mathcal{J}_0)
:=
\arcsin{\mathcal{J}_0}
+
\mathcal{J}_0 \int_{\rho_*}^{1} \frac{ \d\rho}{\rho^2 \sqrt{1- \frac{4 U (\rho) }{\mathcal{E}_0} - \frac{\mathcal{J}^2_0 }{\rho^2}}},
\end{align*}
with
\begin{align*}
\rho_*= \rho_*(\mathcal{E}_0, \mathcal{J}_0):= \max \{ \rho \in (0,1)  :   4 U (\rho) + \frac{ \mathcal{E}_0 \mathcal{J}^2_0}{\rho^2} = \mathcal{E}_0 \}.
\end{align*}
This formula for $\Theta$ becomes
\begin{align} \label{2.2}
\Theta
=
\Theta (\mathcal{E}_0, \mathcal{J}_0) = \arcsin{\mathcal{J}_0}
+
\int_{ \arcsin{\mathcal{J}_0}}^{\frac{\pi}{2}}
\frac{ \rho \mathcal{J}_0 \mathcal{E}_0 \sin \varphi }{\mathcal{E}_0 \mathcal{J}_0^2 - 2 \rho^3 U ' (\rho)  }\d \varphi
\end{align}
after the change of variable
\begin{align*}
\sin^2 \varphi := \frac{4 U (\rho)}{\mathcal{E}_0} + \frac{\mathcal{J}_0^2}{\rho^2}.
\end{align*}
By replacing $\sin \varphi$ in $\eqref{2.2}$, by the above formula, we get
\begin{align*}
\Theta
=
\Theta (\mathcal{E}_0, \mathcal{J}_0) = \arcsin{\mathcal{J}_0}
+
\int_{ \arcsin{\mathcal{J}_0}}^{\frac{\pi}{2}}
\frac{ \rho \mathcal{J}_0 \mathcal{E}_0  \sqrt{  \frac{4 U (\rho)}{\mathcal{E}_0} + \frac{\mathcal{J}_0^2}{\rho^2}  } }{\mathcal{E}_0 \mathcal{J}_0^2 - 2 \rho^3 U ' (\rho)  }\d \varphi .
\end{align*}
Now, by taking the limit $\mathcal{E}_0 \to \infty$ we get
\begin{align*}
\lim_{\mathcal{E}_0 \to \infty} \Theta ( \mathcal{E}_0, \mathcal{J}_0)
=
\arcsin{\mathcal{J}_0}
+
\int_{ \arcsin{\mathcal{J}_0}}^{\frac{\pi}{2}} 1 \d \varphi
=
\arcsin{\mathcal{J}_0} + \frac{\pi}{2} - \arcsin{\mathcal{J}_0}
= \frac{\pi}{2},
\end{align*}
which concludes the proof of the lemma.
\end{proof}

\section{Quantitative errors estimates for the Rayleigh gas}\label{pthm1}
\subsection{Energy estimates}
To prove Theorem \ref{thm1} we
  extend the work of \cite{Matthies2018} by looking  at error estimates over longer times considering long-term dynamics. We are doing this by finding a quantitative error for the difference $ |P_t(S)- \hat{P_t}(S)|$, detailed expressions are  in Proposition $\ref{2.3.1.}$ below. We first provide quantitative estimates for the momentum of the tagged particle, the expected number of collisions and  estimates for recollisions. It will be enough to provide all these estimates for the idealised evolution. We  will use that $\mathcal{B}$ is the collision kernel which by $\eqref{2.4}$ has the form
\begin{align*}
\mathcal{B}(v_1-v_2, \Theta)
&=
\frac{2^{\frac{2}{n-1}}}{a^2} \left[ |v_1-v_2|^{ \frac{n-5}{n-1} }
+
 \mathcal{O}( |v_1-v_2|^{ \frac{n-7}{n-1} }) \right]  \Theta (\sin \Theta)^{-1},
\end{align*}
for $n \in (3,5]$ and large $|v_1 - v_2| >0$ as well as its boundedness for bounded $|v_1 - v_2| $ .
Consider
\begin{align*}
M_f(t):=\int_{{\R}^3 \times \R^3} f_t(x,v) (1+|v|)\, \mathrm{d}v\,  \d x
\end{align*}
and also define the momentum
\begin{align*}
D_f(t) := \int_{\mathds{R}^3 \times \R^3} f_t(x,v)|v| \, \mathrm{d}v\, \d x.
\end{align*}
Then, we observe that
\begin{align*}
M_f(t)=& \int_{\mathds{R}^3 \times \R^3}   f_t(x,v)\, \mathrm{d}v\, \d x +
\int_{\mathds{R}^3 \times \R^3}   f_t(x,v)|v| \, \mathrm{d}v\, \d x
= 1+ \int_{\mathds{R}^3 \times \R^3}   f_t(x,v)|v| \, \mathrm{d}v\, \d x ,
\end{align*}
since $f_t$ is a probability density. By taking the derivative of $D_f(t)$ we get
\begin{align*}
\frac{\d}{\d t}D_f(t) &= \frac{\d}{\d t}\int_{\mathds{R}^3 \times \R^3}   f_t(x,v)|v|\, \mathrm{d}v\, \d x
=\int_{\mathds{R}^3 \times \R^3}  \partial_t f_t(x,v)|v| \, \mathrm{d}v \, \d x\\
&= c\int_{\mathds{R}^3 \times \R^3} |v| \int_{\mathbb{S}^2 \times \R^3}   f_t(x,v') g_0(\bar{v}') \mathcal{B}(v-\bar{v}, \omega) \, \d \bar{v} \,  \d {\omega} \, \mathrm{d}v\, \d x \\
&\quad- c \int_{\mathds{R}^3 \times \R^3} |v| f_t(x,v) \int_{\mathbb{S}^2 \times \R^3}   g_0(\bar{v}) \mathcal{B}(v-\bar{v}, \omega)\, \d \bar{v}\, \d {\omega}\, \mathrm{d}v\, \d x,
\end{align*}
by integration by parts and by using the linear Boltzmann equation $\eqref{LBE}$.
\begin{proposition} The post-collisional velocity is bounded by
 \begin{align}\label{2.1}
|v| \le  |v'| + |\bar{v}'|,
\end{align}
where $ v'$ and $\bar{v}'$ are the pre-collisional velocities.
\end{proposition}
\begin{proof}
We have that
\begin{align*}
 |v|^2 \le |v'|^2 + |\bar{v}'|^2 \le  |v'|^2 + 2 |v'| |\bar{v}'| + |\bar{v}'|^2 = ( |v'| + |\bar{v}'|  )^2.
\end{align*}
Therefore, the inequality $\eqref{2.1}$ follows.
\end{proof}
\begin{proposition}\label{momentum}
The momentum $D_f(t)$ grows linearly with $t$, i.e. for any $t \ge 0$, for $q>4$ and for $2 < n \le 5$
\begin{align*}
D_f(t)
\le
\tilde{C}_{q,n} ct.
 \end{align*}
\end{proposition}
\begin{proof}
By using the inequality $\eqref{2.1}$, we get
\begin{align*}
\frac{\d}{\d t}D_f(t)
&\le c \int_{\mathds{R}^3 \times \R^3} (|v'| + |\bar{v}'|) \int_{\mathbb{S}^2 \times \R^3}   f_t(x,v') g_0(\bar{v}') \mathcal{B}(v-\bar{v}, \omega)\, \d \bar{v}\,   \d {\omega}\, \mathrm{d}v \, \d x \\
&\quad- c \int_{\mathds{R}^3 \times \R^3} |v| f_t(x,v) \int_{\mathbb{S}^2 \times \R^3 }  g_0(\bar{v}) \mathcal{B}(v-\bar{v}, \omega)\, \d \bar{v}\, \d {\omega}\, \mathrm{d}v\, \d x\\
& = \int_{\mathds{R}^3 \times \R^3}
   c (Q^+[f|v'|]- Q^-[f|v|]) \, \mathrm{d}v\, \d x \\
&\quad + c\int_{\mathds{R}^3 \times \R^3} \int_{\mathbb{S}^2 \times \R^3}  f_t(x,v') g_0(\bar{v}') |\bar{v}'| \mathcal{B}(v-\bar{v}, \omega)\, \d \bar{v}   \, \d {\omega}\, \mathrm{d}v \, \d x.
\end{align*}
In the last equality, the first term is equal to zero, by the conservation of mass.
Now, by using change of coordinates $\bar{v}, v \to \bar{v}',v'$ we get that the last integral becomes
\begin{align*}
&\int_{\mathds{R}^3 \times \R^3}  \int_{\mathbb{S}^2 \times \R^3}  f_t(x,v') g_0(\bar{v}') |\bar{v}'| \mathcal{B}(v-\bar{v}, \omega)\, \d \bar{v}   \, \d {\omega}\, \mathrm{d}v\, \d x \\
&= \int_{\mathds{R}^3 \times \R^3} \int_{\mathbb{S}^2 \times \R^3} f_t(x,v') g_0(\bar{v}') |\bar{v}'| \mathcal{B}(v' - \bar{v}' , \omega)\, \d \bar{v}' \,  \d {\omega}\, \mathrm{d}v'\, \d x.
\end{align*}
 Using the fact that the function $f_t$ is a probability density and that  $\mathcal{B}$ is uniformly bounded by Lemma \ref{lem:Bbounded} and Proposition \ref{P3.2} and using $g_0(.)(1+|.|) \in L^1(\R^3)$ we have that the integral can be bounded by a constant $\tilde{C}_{q,n}$ only depending on $g_0$ and $n$. Thus,
\begin{align*}
\frac{\d}{\d t}D_f(t)
\le c \tilde{C}_{q,n}.
 \end{align*}
We notice that the right hand side of the above differential inequality  does not depend on $t$.  Then, by e.g. \cite[Theorem 1.3]{Teschl2012} we get
\begin{align*}
D_f(t)
\le
\tilde{C}_{q,n} c t, \quad \textrm{for} \ q>4,\ 2< n \le5 \ \textrm{and\ for\ all} \ t\ge 0.
 \end{align*}
This means that the momentum $D_f(t)$ grows linearly with $t$.

\end{proof}
\subsection{Number of collisions}
In this section we want to find a bound for $\mathbb{E}(n(\Phi))(t)$, where $n(\Phi)$ is the number of collisions in the collision history $\Phi$. For this, we are going to find an evolution equation for $\mathbb{E}(n(\Phi))(t)$ which can be derived using the idealised equation \eqref{eq-id}.
\begin{proposition}\label{Prop:3.3}
 The expected number of collisions grows linearly with $t$. That is for any $t\ge 0$ and for $3<n\le 5$
\begin{align} \label{3.4}
\mathbb{E}(n(\Phi))(t)
& \le 1 + C c t.
\end{align}
\end{proposition}

\begin{proof}
 We observe that at the initial time $t=0$ there is only the tagged particle in the tree, since we do not have any collision yet, i.e., no initial overlap. Thus,
$\mathbb{E}(n(\Phi))(0)=1.$ Furthermore, $$\mathbb{E}(n(\Phi))(t) = \int_{\mathcal{MT}} n(\Phi) P_t(\Phi)\, \d \Phi.$$
Hence,
\begin{align*}
\frac{\d}{\d t}\mathbb{E}(n(\Phi))(t) &= \frac{\d}{\d t}\int_{\mathcal{MT}} n(\Phi) P_t(\Phi)\, \d \Phi
=\int_{\mathcal{MT}}  \frac{\d}{\d t} n(\Phi) P_t(\Phi) \, \d\Phi+ \int_{\mathcal{MT}} n(\Phi) \frac{\d}{\d t}P_t(\Phi)\, \d \Phi\\
&= \int_{\mathcal{MT}} n(\Phi)c[ Q_t^+[P_t](\Phi) - Q^-_{t}[P_t](\Phi) ]\, \d \Phi,
\end{align*}
here we used equation $\eqref{eq-id}$ and the fact that $\frac{\d}{\d t} n(\Phi)=0.$
Now,
\begin{align*}
\int_{\mathcal{MT}} n(\Phi) Q_t^+[P_t](\Phi)\, \d \Phi & = \int_{\mathcal{MT}} n(\Phi) \mathds{1}_{t=\tau(\Phi)} P_t(\bar{\Phi})g_0(\bar{v}) \mathcal{B}(v^{\varepsilon}(\tau^-) - \bar{v} ,\omega) \, \d \Phi \\
&= \int_{\mathcal{MT}} n(\bar{\Phi}) \mathds{1}_{t=\tau(\Phi)} P_t(\bar{\Phi})g_0(\bar{v}) \mathcal{B}(v^{\varepsilon}(\tau^-) - \bar{v} ,\omega)\, \d \Phi\\
&\quad+ \int_{\mathcal{MT}}\mathds{1}_{t=\tau(\Phi)} P_t(\bar{\Phi})g_0(\bar{v}) \mathcal{B}(v^{\varepsilon}(\tau^-) - \bar{v} ,\omega) \, \d \Phi
\end{align*}
and
\begin{align*}
\int_{\mathcal{MT}} n(\Phi)Q^-_{t}[P_t](\Phi)  \, \d\Phi = \int_{\mathcal{MT}} n(\Phi)P_t(\Phi) \int_{\mathds{R}^3} \int_{\mathbb{S}^2} g_0(\bar{v}) \mathcal{B}(v^{\varepsilon}(\tau) - \bar{v} ,\omega)\, \d \omega\, \d \bar{v}\, \d \Phi.
\end{align*}
Thus,
\begin{align*}
\frac{\d}{\d t}\mathbb{E}(n(\Phi))(t)
&= c \int_{\mathcal{MT}} n(\Phi) \mathds{1}_{t=\tau(\Phi)} P_t(\bar{\Phi})g_0(\bar{v}) \mathcal{B}(v^{\varepsilon}(\tau^-) - \bar{v} ,\omega) \, \d \Phi \\
& \quad- c \int_{\mathcal{MT}} n(\Phi)P_t(\Phi) \int_{\mathds{R}^3} \int_{ \mathbb{S}^2 } g_0(\bar{v}) \mathcal{B}(v^{\varepsilon}(\tau) - \bar{v} ,\omega)\, \d \omega\, \d \bar{v} \, \d\Phi\\
&\quad+c  \int_{\mathcal{MT}}\mathds{1}_{t=\tau(\Phi)} P_t(\bar{\Phi})g_0(\bar{v}) \mathcal{B}(v^{\varepsilon}(\tau^-) - \bar{v} ,\omega) \, \d\Phi.
\end{align*}
We define as $h_t(\Phi) := n(\Phi) P_t(\Phi)$ and then we get
\begin{align*}
\frac{\d}{\d t}\mathbb{E}(n(\Phi))(t)
= c \int_{\mathcal{MT}} [ Q^+ [h_t](\Phi) - & Q^-_{t}[h_t](\Phi)]\, \d \Phi \\
&+ c \int_{\mathcal{MT}}\mathds{1}_{t=\tau(\Phi)} P_t(\bar{\Phi})g_0(\bar{v})
 \mathcal{B}(v^{\varepsilon}(\tau^-) - \bar{v} ,\omega)  \,  \d \Phi.
\end{align*}
Thus,
\begin{align*}
\frac{\d}{\d t}h_t(\Phi)
&=
c[ Q^+ [h_t](\Phi) -  Q^-_{t}[h_t](\Phi) ]
+ c \mathds{1}_{t=\tau(\Phi)} P_t(\bar{\Phi})g_0(\bar{v})\mathcal{B}(v^{\varepsilon}(\tau^-) - \bar{v} ,\omega).
\end{align*}
Now, let us define the forcing term $\Delta^{\varepsilon}(t):=\mathds{1}_{t=\tau(\Phi)} P_t(\bar{\Phi})g_0(\bar{v})\mathcal{B}(v^{\varepsilon}(\tau^-) - \bar{v} ,\omega).$
Then the above equation becomes
\begin{align*}
\frac{\d}{\d t}h_t(\Phi) &=c[ Q^+ [h_t](\Phi) - Q^-_{t}[h_t](\Phi)]
+c \Delta^{\varepsilon}(t).
\end{align*}
By using the variation of constant formula, we have
\begin{equation*}
h_t(\Phi) = P_t(h_0(\Phi)) + c \int_{0}^{t} P_{t-s} (  \Delta^{\varepsilon} (s))\, \d s,
\end{equation*}
where $P_t$ is the solution semigroup of equation $\eqref{eq-id}$. Integrating over $\mathcal{MT}$ we obtain
\begin{equation}\label{1.4}
\int_{\mathcal{MT}} h_t(\Phi) \, \d \Phi = \int_{\mathcal{MT}} P_t (h_0(\Phi))\, \d \Phi
+c \int_{\mathcal{MT}} \int_{0}^{t}P_{t-s} (  \Delta ^{\varepsilon}(s))\, \d s\, \mathrm{d}\Phi,
\end{equation}
with
\begin{align*}
\int_{\mathcal{MT}} \int_{0}^{t}P_{t-s} (\Delta ^{\varepsilon}(s))\, \d s\, \mathrm{d} \Phi & = \int_{\mathcal{MT}} \left| \int_{0}^{t}P_{t-s} (\Delta ^{\varepsilon}(s))\, \d s \right| \, \mathrm{d} \Phi
=  \int_{0}^{t} \left| \int_{\mathcal{MT}} P_{t-s} (\Delta ^{\varepsilon}(s)) \d \Phi \right|\, \d s \\
& =  \int_{0}^{t} \big \| P_{t-s} (\Delta ^{\varepsilon}(s)) \big\|_{L^1(\mathcal{MT})}\, \d s 
= \int_{0}^{t} \big \| \Delta ^{\varepsilon}(s) \big\|_{L^1(\mathcal{MT})}\, \d s.
\end{align*}
Here, in the second equality we used Fubini's Theorem and in the last equality that the solution semigroup $P_{t-s}$ preserves the measure. Now we will bound the term $\| \Delta ^{\varepsilon}(s) \|_{L^1(\mathcal{MT})}$,
then 
\begin{align*}
\| \Delta &^{\varepsilon}(s)\|_{L^1(\mathcal{MT})}  = \int_{\mathcal{MT}} \mathds{1}_{s=\tau(\Phi)} P_s(\bar{\Phi})g_0(\bar{v}) \mathcal{B}(v^{\varepsilon}(\tau^-) - \bar{v} ,\omega)\, \mathrm{d} \Phi\\
=&
\int_{\mathcal{MT}} \int_{\mathds{R}^3} \int_{\mathbb{S}^2} P_s(\bar{\Phi})g_0(\bar{v})
\mathcal{B}(v^{\varepsilon}(\tau^-) - \bar{v} ,\omega) \, \d \omega\, \d \bar{v}\, \mathrm{d} \bar{\Phi} \leq C
\end{align*}
using the uniform boundedness of $\mathcal{B}$ from Lemma \ref{lem:Bbounded}  and Proposition \ref{P3.2} as well as that $P_s$ is a probability measure and that $g_0 \in L^1(\R^3)$. 

Then, we get
\begin{align*}
\int_{\mathcal{MT}} h_t(\Phi)\, \d \Phi &=1
+c \int_{0}^{t} \big \| \Delta ^{\varepsilon}(s) \big\|_{L^1(\mathcal{MT})} \, \d s
\le 1 + C c t , \quad \textrm{for \ all \ } t\ge 0.
\end{align*}
Thus, the expected number of collisions grows linearly with $t$. That is
\begin{align*}
\mathbb{E}(n(\Phi))(t)
& \le 1 + C c t , \quad \textrm{for \ all \ } t\ge 0 \ \textrm{and \ for\  } 3<n\le5.
\end{align*}
This completes the proof of the proposition.
\end{proof} 
\subsection{The set of re-collisions }
The aim of this section is to estimate the probability of the set of re-collisions, with the number of collisions $n(\Phi)$ in the collision history $\Phi$ be bounded by a decreasing function $M(\varepsilon)$ and with the maximal velocity in the tree bounded by another decreasing function $V(\varepsilon)$, with $\lim_{\varepsilon \to 0} M(\varepsilon) = \lim_{\varepsilon \to 0} V(\varepsilon) = \infty$.
\begin{proposition} \label{2.1.2.} Let $\tilde{\eta} >0$ sufficiently small, $\tilde{\delta}>0$ and let the decreasing functions $V, M : (0,\infty) \to [0, \infty)$ such that $\lim_{\varepsilon \to 0}V(\varepsilon) = \lim_{\varepsilon \to 0}M(\varepsilon) = \infty,$ then the probability of the set of re-collisions with bounded number of collisions and bounded velocities is bounded. That is
\begin{align*}
 P_t \big(&\text{ re-collisions with} \  n(\Phi)  \le M(\varepsilon) \ and\ \mathcal{V}(\Phi) <V(\varepsilon) \big)\\
 &\le
C M(\varepsilon) 
 \left( \tilde \delta+ \frac{\tilde \eta V(\varepsilon) \varepsilon}{\tilde{\delta}^2}  +  \frac{V(\varepsilon)\varepsilon }{
\tilde{\eta}+ \tilde{\delta} } +  
M(\varepsilon) \left(  \frac{\tilde \eta  V(\varepsilon) \varepsilon}{\tilde \delta} + 
\frac{V(\varepsilon) \varepsilon}{\tilde{\eta}+ \sqrt{\tilde{\delta}}  }   \right)\right).
  \numberthis  \label{4.1}
 \end{align*}
\end{proposition}
\begin{proof}
The main strategy here is to identify variables enforcing the re-collision  and then estimate the volume of possible velocities of the background particles or collision parameters. It is enough to consider their volumes as we have uniform $L^\infty$ bounds on their densities and on the collision kernel $\mathcal{B}$. We will consider various cases. Some estimates are easier than the qualitative estimates in \cite{Matthies2018} because  we are on the domain $\R^3$ such that there will be no re-collisions due to periodic boundary conditions on the torus. As we are estimating events for the idealised distribution,  we consider idealised trajectories with instantaneous collision events for point particles, but that come close within a distance $\varepsilon$.

There cannot be any re-collisions in $\mathcal{MT}_0 :=\{ \Phi \in \mathcal{MT}  :\ n(\Phi) = 0\}$ as there are no collisions and
 $\mathcal{MT}_1 := \{ \Phi \in \mathcal{MT} :\ n(\Phi)=1\}$ as the tagged particle and the only colliding particle are moving away from each other for all times.

Now, let $2\le j \le n(\Phi)$ and $\Phi \in \mathcal{MT}_j$ with $$\Phi= ((x_0, v_0), (t_1, \nu_1, v_1),...,(t_j, \nu_j, v_j),...,(t_n, \nu_n, v_n)).$$
Let $\Phi \in \mathcal{MT}_j \setminus R(\varepsilon)$, i.e. the collision history $\Phi$ consists of $j$ collisions and it is not re-collision free. As above the re-collision cannot occur for $k=j$ as the tagged particle and the final colliding background particle move away from each other.

In the other cases there is a re-collision with the background particle number $k$, with $1\le k <j $. In this situation, there exists re-collision, since after collision with the background particle $k$, there is always a collision with another background particle. Therefore, the tagged particle changes direction and it can collide with the background particle number $k$ again.

There is a special situation if $\Phi \in R(0)$.  Then two of the collisions inside the collision history $\Phi$, correspond to the same background particle. We have that there exist numbers $l,k$ with $2\le l\le j$ and $1\le k< l$ such that the $l^{th}$ and the $k^{th}$ collision corresponds to the same background particle. This implies that
\begin{align*}
v_l = v_k + \nu _l (v(t_l^{-}) - v_k) \cdot \nu_l,
\end{align*}
where $v_l$ and $v_k$ are the incoming velocities of the background particles before the $l^{th}$ and the $k^{th}$ collision respectively, $\nu_l$ is the collision parameter and $v(t_l^{-}) $ is the velocity of the tagged particle before the $l^{th}$ collision. Thus $v_l$ is determined by $v_k$, $\nu_l$ and $ v(t_l^{-}) $, so $v_l$ can only be in a set of zero probability with respect to $P_t$.

The main case is a  re-collision with the background particle number $k$ will happen at some time $s \in (t_{j}, T]$. It is enough to consider trees $\Phi$, where the re-collision will be at some time after the time of the last collision inside the collision history $\Phi$. The probability of all other trees with re-collisions is dominated by such trees. There are at most $M(\varepsilon)$ possibilities for $j$.
Thus there exist $s\in (t_j, T],\ \tilde{\nu}\in \mathbb{S}^2$ and $k \in \{1, \ldots, j-1\}$ such that
\begin{align}
    x(s)  +\varepsilon \tilde{\nu} =x_k  (s), \label{eqn:recoll}
\end{align}
where $ x(s)=x_0 +t_1v_0 +(t_2-t_1) v(t_1) +...+ (s-t_j)v(t_j)$
 is the position of the tagged particle and $x_k(s)= x_k(t_k) +\varepsilon \nu +(s-(t_k))\bar{v}_k$
 is the position of the re-colliding background particle.
We consider two main cases with several subcases. We split the interval $(t_j,T]$ in two parts. The first part is $(t_j,t_j+\tilde{\eta})$ which covers re-collision shortly after the last collision with $\tilde{\eta} >0$ sufficiently small to be chosen later and the bulk part $[t_j+\tilde{\eta}, T]$.

\subsubsection*{Case 1: $s\in(t_j,t_j+\tilde{\eta}) $ }

We  consider the case where $ 2 \le j \le n(\Phi)$, $1\le k < j$ and  $s \in ( t_j, t_j + \tilde{\eta} )$, with $\tilde{\eta}>0 $ small to be chosen later. In this interval, re-collisions are more likely, so we need more careful estimates.
Thus there exist $s\in (t_j,t_j+\tilde{\eta}),\ \bar{\nu}$ and $k:\ 1\le k < j$ such that
\begin{align}\label{4.8}
x(s)  +\varepsilon \bar{\nu} =x_k  (s),
\end{align}
where $ x(s)=x_0 +t_1v_0 +(t_2-t_1) v(t_1) +...+(t_k - t_{k-1})v(t_{k-1})+ (t_{k+1}- t_k) v(t_k)+...+(s-t_j)v(t_j)$ and $x_k(s)= x_k(t_k) +(s-t_k)\bar{v}_k$  with $\bar{v}_k$ the post-collisional velocity of the $k$-th background particle.
 We rewrite  \eqref{4.8} by shifting the origin to $x(t_k)$, such that we have
\begin{align*}
    (t_{k+1} - t_k)v(t_k)+(t_{k+2} - t_{k+1})v(t_{k+1})&+...+(s-t_j)v(t_j) +\varepsilon \bar{\nu}  =
     (s - t_k )\bar{v}_k. \numberthis \label{eqn:jkrecol}
\end{align*}
We note that the post-collisional velocity of the $k$-th background particle and of the tagged particle after the $k$-th collision are given by
\begin{align}\label{4.9}
    \bar{v}_k& = v_k + \nu_k (v(t_k^-)-v_k)\cdot \nu_k,\\
    v(t_k)&= v(t_k^-) - \nu_k (v(t_k^-)-v_k)\cdot \nu_k
\end{align}
and that  the later root velocities only depend on $v(t_k)$. So for fixed $\nu_k$, this implies that the left hand side of \eqref{eqn:jkrecol} only depends on $\nu_k \cdot v_k$, while the right hand side only depends on the components of $v_k$ which are perpendicular to $\nu_k$. We fix $\nu_k \cdot v_k$, then the left hand side of \eqref{eqn:jkrecol} defines a
cylinder when we vary $s \in [t_j, t_j+\tilde \eta] $ and $\tilde{\nu} \in  \mathbb{S}^2$. The three-dimensional volume of the cylinder is proportional to $\tilde \eta V(\varepsilon) \varepsilon^2$. Projecting this cylinder to the plane $\nu_k^\perp$  yields the constraint for the orthogonal components of $v_k$. Its area  in $\nu_k^\perp$  is bounded by $C\tilde \eta V(\varepsilon) \varepsilon$.

The factor  $s-t_k$ can be easily bounded from below by some $\tilde \delta>0$ with a probability $1-\tilde \delta$ for the case  $k=j-1$ and with the same probability by $\sqrt{\tilde{\delta}}>0$ for $k<j-1$.
 Then the overall volume of suitable $v_k$ and $\nu_k$ can be estimated for $k=j-1$
\begin{equation}
\mathrm{vol}(v_k,\nu_k) \le C \frac{1}{\tilde{\delta}^2}\tilde \eta V(\varepsilon) \varepsilon 
\end{equation}
and for the other $k$ by 
\begin{equation}\label{eqn:jkvolest}
\mathrm{vol}(v_k,\nu_k) \le C \frac{1}{\tilde{\delta}}\tilde \eta V(\varepsilon) \varepsilon 
\end{equation}
The overall contribution is obtained by summing over all possible $k$.

\subsubsection*{Case 2: $s\in [t_j+\tilde{\eta}, T]$ } We restrict the interval to  $s\in [t_j+\tilde{\eta}, T]$.
Then \eqref{eqn:recoll}
is equivalent to
 \begin{align*}
Y+s v(t_j)+\varepsilon \tilde{\nu} = (s-t_k)\bar{v}_k,
\end{align*}
where $Y : = (t_k -t_{k-1})v(t_{k-1}) +... +(t_j -t_{j-1})v(t_{j-1})-t_j v(t_j)$. This can be solved for the terminal velocity of the recolliding particle and projecting onto $\nu_k^\perp$ again we obtain constraints in the "cylinder" with varying radius $\frac{\varepsilon}{s-t_k}$ around the curve defined by 
\begin{align} \label{eqn:r}
    r(s):= \frac{Y+s v(t_j)}{s- t_k},
\end{align}
for $1\le k< j$.  We observe that $|Y+s v(t_j)|\leq (s-t_k)  V(\varepsilon) $. Then  we estimate its length by calculating
\begin{align}
    \frac{\d}{\d s}r(s)= -\frac{Y +s v(t_j)}{(s-t_k)^2}+ \frac{v(t_j)}{s- t_k}. \label{eqn:drdy}
\end{align}
Therefore the $2d$ volume around  the curve for $s\in [t_j+\tilde{\eta}, T]$ is bounded by
\begin{align*}
\int_{t_j+\tilde{\eta}}^{T}\left|\frac{\d}{\d s}r(s)\right|  \frac{\varepsilon}{s-t_k} \, \d s & \le \int_{t_j + \tilde{\eta}}^{T} 
\frac{(s-t_k) V(\varepsilon) \varepsilon}{(s-t_k)^3}\, \d s + \int_{t_j+ \tilde{\eta}}^{T}\frac{V(\varepsilon)\varepsilon}{(s-t_k)^2}\, \d s \\ 
&\le \int_{t_j + \tilde{\eta}}^{T} 
\frac{  V(\varepsilon) \varepsilon}{(s-t_k)^2}\, \d s + \int_{t_j+ \tilde{\eta}}^{T}\frac{V(\varepsilon)\varepsilon}{(s-t_k)^2}\, \d s  \le 
\frac{C V(\varepsilon)\varepsilon }{
\tilde{\eta}+ (t_j-t_k) 
},
\end{align*}
where again $t_j-t_{j-1} \geq \tilde{\delta}$ and  $t_j-t_{k} \geq \sqrt{\tilde{\delta}}$ for $k< j-1$.

After having investigated all the possible cases, take the probability of the set of re-collisions, with the number of collisions $n(\Phi)$ in the collision history $\Phi$ be bounded by a decreasing function $M(\varepsilon)$, with $\lim_{\varepsilon \to 0} M(\varepsilon) = \infty$, we have
\begin{align*}
 P_t \big(&\text{ re-collisions with} \ n(\Phi)  \le  M(\varepsilon) \ and \ \mathcal{V}(\Phi) <V(\varepsilon) \big)\\
 & \le  \sum_{j\ge 0} P_t \big( ( \mathcal{MT}_j \setminus \mathcal{R}(\varepsilon) \big)\\
 &=
 P_t \big( ( \mathcal{MT}_1 \setminus \mathcal{R}(\varepsilon) \big) + \sum_{j= 2}^{n(\Phi)} P_t \big( ( \mathcal{MT}_j \setminus \mathcal{R}(\varepsilon) \big)\\
 & \le
 C \sum_{j= 2}^{n(\Phi)} \left( \tilde \delta+ \frac{\tilde \eta V(\varepsilon) \varepsilon}{\tilde{\delta}^2}  +  \frac{ V(\varepsilon)\varepsilon }{
\tilde{\eta}+ \tilde{\delta}} +  
M(\varepsilon) \left(  \frac{\tilde \eta  V(\varepsilon) \varepsilon}{\tilde \delta} + 
\frac{V(\varepsilon) \varepsilon}{\tilde{\eta}+ \sqrt{\tilde{\delta}}  }   \right)\right)\\
 & \le C M(\varepsilon) 
 \left( \tilde \delta+ \frac{\tilde \eta V(\varepsilon) \varepsilon}{\tilde{\delta}^2}  +  \frac{V(\varepsilon)\varepsilon }{
\tilde{\eta}+ \tilde{\delta} } +  
M(\varepsilon) \left(  \frac{\tilde \eta  V(\varepsilon) \varepsilon}{\tilde \delta} + 
\frac{V(\varepsilon) \varepsilon}{\tilde{\eta}+ \sqrt{\tilde{\delta}}  }   \right)\right),
 \end{align*}
 as required.
\end{proof}


\subsection{Initial overlap}
 We want to estimate the set of initial overlaps  of the  supports of potentials. For the empiric distribution this can be estimated using the properties of the spatial Poisson process. Here, we estimate a potential  overlap if the idealised description had particles of diameter $\varepsilon$. For this, we have the following lemma.
\begin{lemma} \label{lem:overlap}
Let $\delta>0$, then the probability of initial overlap of the  supports of potentials at diameter $\varepsilon$ is bounded by
 \begin{align*}
&P_t\left(\{ \Phi \in \mathcal{MT} : n(\Phi) \leq M(\varepsilon) \mbox{ and }  \exists j \in \{1,..., n(\Phi) \} : \ |x_0 - x_j| \le \varepsilon \}\right)\\&\le
C \varepsilon + C \delta^2   + C M(\varepsilon)  \left( \frac{\varepsilon}{\delta}\right)^3.\numberthis \label{est:overlap}
\end{align*}
\end{lemma}
\begin{proof}
Let $\Phi \in \mathcal{MT}$ such that $\Phi = ((x_0,v_0),(t_1, \nu_1, v_1),..., (t_n, \nu_n, v_n ) )$, where $(x_0,v_0)\in  \mathds{R}^3\times \mathds{R}^3 $, $t_i \in [0,T]$, $\nu_i \in \mathbb{S}^2$, $v_i \in \mathds{R}^3$, for $i=1,...,n$. The condition for initial overlap of supports of the potentials is
\[|x_j(0)- x_0 | \le \varepsilon. \]
Using $x_j(0)= x(t)- t_j v_j$ in the idealised setting, this yields the condition $v_j \in B_j $ with
\[   B_j= \{ v \mid   | v - \frac{x(t)-x_0}{t_j} | \le  \frac{\varepsilon}{t_j} \}. \]
Assume $t_j \geq \delta$, then the probability of $v_j \in B_j$ is given by
\begin{align} \label{eqn:overlap1}
    \int_{B_j} g_0(v)\, \d v \leq C \left( \frac{\varepsilon}{\delta}\right)^3.
\end{align}
The probability to have a tree with  $t_2 <\delta$ can be bounded by estimating the rate of collisions for the root particle with velocity $v_0$
\[|Q^-_t(\Phi)|=\left|\int_{\mathds{R}^3} \int_{\mathbb{S}^2} g_0(\bar{v}) \mathcal{B}(v_0 - \bar{v} ,\omega)\, \d{\omega} \, \d \bar{v}\right| \leq C, \]
due to the form of $\mathcal{B}$. Hence the probability to have two collisions is bounded by
\begin{align}\label{eqn:overlap2}
    C^2 \delta^2.
\end{align}
The probability that the first colliding particle has initial overlap combines the probability of a single collision and overlap  which can be estimated by
\begin{align}\label{eqn:overlap3}
\int_0^{T}  C \min\left(1, \frac{\varepsilon^3}{t^3} \right)\,  \d t \le C \varepsilon.
 \end{align}
Combining \eqref{eqn:overlap1}, \eqref{eqn:overlap2} and \eqref{eqn:overlap3} yields the result.
 \end{proof}

\subsection{The set of binary collisions }

Collisions of multiple background particles with the tagged particle are rare even for the empirical particle dynamics. According to \cite{Gallagher2013}[Prop. 8.2.1], individual collision times are bounded by  $\tau^* \varepsilon$ with $\tau_*$ uniformly bounded on compact sets of $\mathcal{E}_0$ in  $(0, \infty) $.  On the level of the idealised distribution we can quantify estimates in terms of the minimal time interval $I(\varepsilon)$. 

\begin{lemma}\label{lem:Ieps}
Let $I$  be a  monotone and continuous function $I:[0,\infty) \to [0,\infty)$ with $\lim_{\varepsilon \to 0} I(\varepsilon)= 0$. Then
\begin{align*}
P_t\bigl(\{ \Phi \in \mathcal{MT} : n(\Phi) \leq M(\varepsilon), \Phi \mbox{ is re-collision-free and }  \vartheta(\Phi) \geq I(\varepsilon)  \}\bigr)&\le
C M(\varepsilon) I(\varepsilon).\numberthis \label{est:timediff}
\end{align*}
\end{lemma}
\begin{proof}
The rate of collisions is bounded, hence the probability to have a collision within the time $\vartheta(\Phi)$ is proportional to that time. As there are at most $M(\varepsilon)$ collisions in $\Phi$ and as each collision  is independent we obtain the bound.    
\end{proof}

\subsection{Collisions with small relative velocities}
To obtain uniform bounds on the collision times of the empiric realisation of a tree $\Phi$, we require also lower bounds on the relative velocities in each collision. The requirements rules out a small ball in velocity space around the pre-collision velocity $v(t_j^-)$ of the tagged particle.

\begin{lemma}\label{lem:Jeps}
Let $J$  be a  monotone and continuous function $J:[0,\infty) \to [0,\infty)$ with $\lim_{\varepsilon \to 0} J(\varepsilon)= 0$. Then
\begin{align*}
P_t\left(\{ \Phi \in \mathcal{MT} : n(\Phi) \leq M(\varepsilon), \Phi \mbox{ is re-collision-free and }  \iota(\Phi) < J(\varepsilon)  \}\right)&\le
C M(\varepsilon) J^3(\varepsilon).\numberthis \label{est:veldiff}
\end{align*}
\end{lemma}
\begin{proof}
The probability to have a collision such that $|v_j -v(t_j^-)|< J(\varepsilon)$, means it is enough to estimate 
$$ \int_{|v-v(t_j^-)| \leq J(\varepsilon)} g_0(v)\, \d v \leq C \left(J(\varepsilon)\right)^3 $$
for each collision. There are at most $M(\varepsilon)$ collisions in $\Phi$,  which yields the bound.    
\end{proof}
\subsection{Quantitative estimates for bad trees}
In this section, we estimate the probability of the bad trees. To do so, we have previously estimated the probability of re-collisions in the collision history and the probability of initial overlaps. It remains to find a bound for the probability of the number of collisions being unbounded and also to find a bound for the probability that the maximum velocity in the collision history being unbounded while the number of collisions is bounded.
\begin{proposition} \label{2.3.1.}  For any decreasing functions $V, M : (0,\infty) \to [0, \infty)$ such that $\lim_{\varepsilon \to 0}V(\varepsilon) = \lim_{\varepsilon \to 0}M(\varepsilon) = \infty$ and for any $q>4$, $t\in [0,T]$, $\delta>0$, $\tilde{\delta}>0$, $\tilde{\eta}>0$ and $n \in (3,5]$, the probability of the set of bad trees has the following error estimate
\begin{align}\label{est:badtrees}
P_t (\mathcal{MT} \setminus \mathcal{G}(\varepsilon))
& \le
\textnormal{Ov}(\varepsilon)
+
\textnormal{Rec}(\varepsilon)
+
\textnormal{Hi}(\varepsilon)
+
\textnormal{Vel}(\varepsilon) + \textnormal{Tri}(\varepsilon)
+ \textnormal{Slow}(\varepsilon),
\end{align}
where
\begin{enumerate}
\item$\textnormal{Ov}(\varepsilon) \le
C \varepsilon + C \delta^2   + C M(\varepsilon)  \left( \frac{\varepsilon}{\delta}\right)^3$ 
is due to the initial overlaps,
\item$\textnormal{Rec}(\varepsilon) \le
C M(\varepsilon) 
 \left( \tilde \delta+ \frac{\tilde \eta V(\varepsilon) \varepsilon}{\tilde{\delta}^2}  +  \frac{V(\varepsilon)\varepsilon }{
\tilde{\eta}+ \tilde{\delta} } +  
M(\varepsilon) \left(  \frac{\tilde \eta  V(\varepsilon) \varepsilon}{\tilde \delta} + 
\frac{V(\varepsilon) \varepsilon}{\tilde{\eta}+ \sqrt{\tilde{\delta}}  }   \right)\right) $
is due to re-collisions,
\item$ \textnormal{Hi}(\varepsilon) \le   \frac{1}{M(\varepsilon)} (1+ C ct)$ is due to the unbounded number of collisions, 
\item$  \textnormal{Vel}(\varepsilon) \le
 \frac{M_{f_0}}{V(\varepsilon)}
 + ( C_g + \tilde{C}_{q,n} ct )\frac{ M(\varepsilon) }{V(\varepsilon)} $ is due to the unbounded velocities,
 \item$  \textnormal{Tri}(\varepsilon) \le
  C M(\varepsilon) I(\varepsilon) $ is to rule small minimal collision time differences $\vartheta(\varepsilon)$, and
  \item$  \textnormal{Slow}(\varepsilon) \le
  C M(\varepsilon) (J(\varepsilon))^3 $ is to rule out collisions with relative speeds smaller than $\iota(\varepsilon)$,
 \end{enumerate}
where $C_g:= \pi C_1 \bar{R}^4 + \frac{ 4\pi C_2}{ q-4} \bar{R}^{-q+4} $,
$M_{f_0}:=  \int_{\mathds{R}^3 \times \mathds{R}^3 \setminus B(0,V(\varepsilon))} f_0(x,v) |v| \, \d x\, \d v$ and $\tilde{C}_{q,n}>0$.
\end{proposition}
\begin{proof}
By the inclusion-exclusion principle we obtain
\begin{align*}
P_t (\mathcal{MT} \setminus \mathcal{G}(\varepsilon))& \le P_t (\mathcal{MT}\setminus S(\varepsilon)) + P_t ( n(\Phi)>M(\varepsilon) )\\ & + P_t \big( ( \mathcal{MT} \setminus \mathcal{G}(\varepsilon))\cap \{ \Phi :\ \Phi \in S(\varepsilon)\ \mathrm{with} \ n(\Phi) \le M(\varepsilon) \} \big)\\
& + P_t \big( \mathcal{V}(\Phi)>V(\varepsilon) \ \mathrm{with} \ n(\Phi) \le M(\varepsilon) \big)\\
&+ P_t\big( \vartheta(\Phi)< I(\varepsilon) \ \mathrm{with} \ n(\Phi) \le M(\varepsilon) \big) \\
&+ P_t\big( \iota(\Phi)< J(\varepsilon) \ \mathrm{with} \ n(\Phi) \le M(\varepsilon) \big) .
\end{align*}
The forth and fifth terms can be directly estimated by Lemmas  \ref{lem:Ieps} and \ref{lem:Jeps} respectively. We also observe that the third term can be written as
\begin{align*}
 P_t \big( ( \mathcal{MT} \setminus \mathcal{G}(\varepsilon))\cap \{ \Phi &:\ \Phi \in S(\varepsilon)\ \mathrm{with} \ n(\Phi) \le M(\varepsilon) \} \big)
 \\&=
 P_t \big(\mathrm{ recollisions \ with} \ n(\Phi) \le M(\varepsilon) \ and \ \mathcal{V}(\Phi) <V(\varepsilon) \big).\numberthis  \label{1.6}
 \end{align*}
Now, by the Markov's inequality and by $\eqref{3.4}$ we have 
\begin{align*}
P_t \big( n(\Phi)  >M(\varepsilon) \big) &\le \frac{\mathbb{E}(n(\Phi))}{M(\varepsilon)}
 \le \frac{1}{M(\varepsilon)} (1 +Cct), \numberthis \label{2.9}
\end{align*}
for all $t\ge 0$ and $n \in (3,5]$.
Also, by Proposition \ref{2.1.2.} we have
\begin{align*}
 P_t \big(\mathrm{ recollisions \ with}& \  n(\Phi)  \le M(\varepsilon) \ and \ \mathcal{V}(\Phi) <V(\varepsilon) \big)
\\ & \le
C M(\varepsilon) 
 \left( \tilde \delta+ \frac{\tilde \eta V(\varepsilon) \varepsilon}{\tilde{\delta}^2}  +  \frac{V(\varepsilon)\varepsilon }{
\tilde{\eta}+ \tilde{\delta} } +  
M(\varepsilon) \left(  \frac{\tilde \eta  V(\varepsilon) \varepsilon}{\tilde \delta} + 
\frac{V(\varepsilon) \varepsilon}{\tilde{\eta}+ \sqrt{\tilde{\delta}}  }   \right)\right). \numberthis \label{2.6}
 \end{align*}
Now, we want to find a bound for the probability $P_t \big( \mathcal{V}(\Phi)>V(\varepsilon) \ \mathrm{with} \ n(\Phi) \le M(\varepsilon) \big)$, where $\mathcal{V}(\Phi)$ is the maximum velocity on the history $\Phi$ and $V$ is a decreasing function of $\varepsilon$ with $\lim_{\varepsilon \to 0}V(\varepsilon) = \infty.$ Let us consider for the background particles that the velocity $v_j$ is i.i.d. with respect to $g_0$. Then
\begin{align*}
\int_{\mathds{R}^3 \setminus B(0,V(\varepsilon))} g_0(v) \, \d v
& \le \frac{1}{V(\varepsilon)}
  \int_{\mathds{R}^3 \setminus B(0,V(\varepsilon))} g_0(v) |v|\, \d v 
  \le \frac{1}{V(\varepsilon)}  \int_{\mathds{R}^3} g_0(v) |v|\, \d v \\
  & =
   \frac{1}{V(\varepsilon)} [ \pi C_1 \bar{R}^4 + \frac{ 4\pi C_2}{ q-4} \bar{R}^{-q+4} ]
    =:   \frac{C_g}{V(\varepsilon)},
\end{align*}
for q>4.
Therefore
\begin{align*}
\int_{\mathds{R}^3 \setminus B(0,V(\varepsilon))} g_0(v)\, \d v
& \le   \frac{C_g}{V(\varepsilon)}, \ q>4.
\end{align*}
Thus
\begin{align*}
P_t \big( |v_j| \le V(\varepsilon)  \big)= 1 - P_t \big( |v_j| >  V(\varepsilon)  \big)
\ge
 1 - \frac{ C_g}{V(\varepsilon)}.
\end{align*}
Let us consider for the tagged particle that the initial velocity $v_0$ is i.i.d. with respect to $f_0$ and assume furthermore that \[M_{f_0}:=  \int_{\mathds{R}^3 \times \mathds{R}^3 \setminus B(0,V(\varepsilon))} f_0(x,v) |v| \, \d x \, \d v \ \mathrm{and} \  \tilde{M}_{f}(t):=  \int_{\mathds{R}^3 \times \mathds{R}^3 \setminus B(0,V(\varepsilon))} f_t(x,v) |v|\, \d x\, \d v\] for $t\in [0,T]$.
Then
\begin{align*}
\int_{\mathds{R}^3 \times \mathds{R}^3 \setminus B(0,V(\varepsilon))} f_0(x, v) \, \d x \, \d v
 \le
 \frac{1}{V(\varepsilon)} \int_{\mathds{R}^3 \times \mathds{R}^3 \setminus B(0,V(\varepsilon))} f_0(x, v) |v|\, \d x \, \d v
 =
  \frac{M_{f_0}}{V(\varepsilon)}
\end{align*}
and for $t \in [0,T]$
\begin{align*}
\int_{\mathds{R}^3 \times \mathds{R}^3 \setminus B(0,V(\varepsilon))} f_t(x,v)\, \d x \, \d v
\le \frac{1}{V(\varepsilon)} \int_{\mathds{R}^3 \times \mathds{R}^3 \setminus B(0,V(\varepsilon))} f_t(x, v) |v|\, \d x\, \d v
=
\frac{\tilde{M}_f(t)}{V(\varepsilon)}.
\end{align*}
We made the above estimates, since we know that the distribution function $f$ has first moment. Thus
\begin{align*}
P_t \big( |v_0| >  V(\varepsilon) \ \mathrm{or} \ |v_1|> V(\varepsilon)\ \mathrm{or}\ ...\ \mathrm{or} \ |v_{n(\Phi)}| >  V(\varepsilon) \big)
&\le
P_t \big( |v_0| >  V(\varepsilon) \big) +   P_t \big( |v_1|> V(\varepsilon) \big) \\
&+ ... + P_t \big( |v_{n(\Phi)}| >  V(\varepsilon) \big)\\
& \le
 \frac{M_{f_0}}{V(\varepsilon)} + n(\Phi) \frac{\tilde{M}_f(t)}{V(\varepsilon)}, \ t\in [0,T].
\end{align*}
Hence
\begin{align*}
P_t \big( \mathcal{V}(\Phi)>V(\varepsilon) \ \mathrm{with} \ n(\Phi)
 \le M(\varepsilon) \big) &\le P_t \big( |v_j|>V(\varepsilon) \big)
+  P_t \big( |v(s)|>V(\varepsilon) \big)\\
&\le
 \frac{M_{f_0}}{V(\varepsilon)}
+ (C_g +  \tilde{M}_f(t) ) \frac{M(\varepsilon) }{V(\varepsilon)} \\
& \le
\frac{M_{f_0}}{V(\varepsilon)}
+ (C_g + \tilde{C}_{q,n} ct ) \frac{M(\varepsilon) }{V(\varepsilon)}, \ t \in [0,T], \numberthis \label{6.4.}
\end{align*}
by using Proposition \ref{momentum} in the last inequality, since $\tilde{M}_f(t) \le D_f(t)$.
Therefore, by summing \eqref{est:overlap}, \eqref{2.9}, \eqref{4.1} and \eqref{6.4.} and the estimates in  Lemmas  \ref{lem:Ieps} and \ref{lem:Jeps},  we obtain the desired estimate
 \begin{align*}
P_t (\mathcal{MT} \setminus \mathcal{G}(\varepsilon))&
 \le
 C \varepsilon + C \delta^2   + C M(\varepsilon)  \left( \frac{\varepsilon}{\delta}\right)^3
 +
\frac{1}{M(\varepsilon)} (1 + C ct) \\
& \quad +
C M(\varepsilon) 
 \left( \tilde \delta+ \frac{\tilde \eta V(\varepsilon) \varepsilon}{\tilde{\delta}^2}  +  \frac{V(\varepsilon)\varepsilon }{
\tilde{\eta}+ \tilde{\delta} } +  
M(\varepsilon) \left(  \frac{\tilde \eta  V(\varepsilon) \varepsilon}{\tilde \delta} + 
\frac{V(\varepsilon) \varepsilon}{\tilde{\eta}+ \sqrt{\tilde{\delta}}  }   \right)\right)
\\
 &\quad +C M(\varepsilon)\left( I(\varepsilon) + (J(\varepsilon))^3 \right)
\\
 &\quad
+ \frac{M_{f_0}}{V(\varepsilon)}
+(  C_g  + \tilde{C}_{q,n} ct )  \frac{M(\varepsilon) }{V(\varepsilon)}, \ t \in [0,T] \ \textrm{and} \ n \in (3,5].
\end{align*}
\end{proof}

\subsection{The Empirical Distribution}
On the set of good histories the particle dynamics can be stated in a similar way to \eqref{eq-id}.
We consider now the empirical distribution on collision histories $\hat{P_t^{\varepsilon}}$ defined by the dynamics of the particle system for particles with diameter $\varepsilon$. We will write $\hat{P_t}$ instead of $\hat{P_t^{\varepsilon}}$. The main result in this subsection is that $\hat{P_t}$ solves a differential equation which is similar to the idealised equation $\eqref{eq-id}$ when restricted to the set of good trees $ \mathcal{G}(\varepsilon)$. We follow the approach in  \cite{Matthies2018} which derived these equations for hard-sphere interactions and adaptions made for finite and infinite range interactions in \cite{Egginton2017}. As initial conditions are obtained from a Poisson process the equations simplify as the terms due to conditioning do not appear.

Now we define the operator $\hat{\mathcal{Q}}_t$ which mirrors the idealised operator in the empirical case. For given history $\Phi$, a time $0<t<T$ and $\varepsilon >0$, we define the function $\mathds{1}^{\varepsilon}_t[\Phi] : \R^3 \times{\mathds{R}^3} \to \{ 0,1\}$ by
\begin{align}
\mathds{1}^{\varepsilon}_t[\Phi] (\bar{x}, \bar{v}) :=
\begin{cases}
1 & \text{if\ for\ all\ } s\in (0,t),\ |x(s)-(\bar{x} + s \bar{v})|>\varepsilon \text{ or }
\\ &  \quad \text{if\ for\ all\ } s\in (0,t) \text{ with } |x(s)-(\bar{x} + s \bar{v})|\leq \varepsilon:  
\beta(\bar{x} + s \bar{v},\bar{v},s) =0 \\
0 &\text{else}.
\end{cases}
\end{align}
That is to say, $ \mathds{1}^{\varepsilon}_t[\Phi](\bar{x}, \bar{v})$ is $1$ if a background particle starting at the position $(\bar{x}, \bar{v})$ avoids colliding with the tagged particle defined by the collision history $\Phi$ up to time $t$. For a history $\Phi$, $t\ge 0$ and $\varepsilon>0$ we define the gain operator, 
\begin{align*}
\hat{\mathcal{Q}}^+_t[\hat{P_t}](\Phi):=
\begin{cases}
\delta (t- \tau) \hat{P_t}(\bar{\Phi})  g_0(v') \mathcal{B}(v(\tau^-) - v', \omega) 
& \text{if}\ n(\Phi)\ge 1\\
0 & \text{if}\ n(\Phi)=0.
\end{cases}
\end{align*}

Next we define the loss operator, which uses that there are no triple collisions for the particle dynamics. We denote by $\tau(\Phi)$ the begin of the final collision and by $T^*(\Phi) \in \mathcal{O}(\varepsilon)$ the duration of this last collision.     
\begin{align*}
\hat{\mathcal{Q}}^-_t[\hat{P_t}](\Phi):= \mathds{1}_{t\geq \tau(\Phi)+ T^*(\Phi)} 
\hat{P_t}(\Phi) 
\int_{\mathbb{S}^2} \int_{\mathds{R}^3}  g_0(\bar{v}) \mathcal{B}(v(\tau) - \bar{v}, \omega) 
\mathds{1}^{\varepsilon}_t[\Phi] (x(t)+ \varepsilon \omega, \bar{v})
\, \d \bar{v} \, \d \omega.
\end{align*}
Finally define the operator $\hat{\mathcal{Q}_t}$ as
\begin{equation*}
\hat{\mathcal{Q}_t} := \hat{\mathcal{Q}}_t^+ - \hat{\mathcal{Q}}_t^-.
\end{equation*}
\begin{theorem} For $\varepsilon$ sufficiently small and for $\Phi \in \mathcal{G(\varepsilon)}$, $\hat{P_t}$ solves the following equation
\begin{align}\label{emp-eq}
\begin{cases}
\partial_t{\hat{P_t}}(\Phi) &= c \hat{\mathcal{Q}_t}[\hat{P_t}] (\Phi) \\
\hat{P_0}(\Phi) &= \zeta(\varepsilon) f_0 (x_0, v_0) \mathds{1}_{n(\Phi)=0}
\end{cases}
\end{align}
with a correction factor for initial overlaps 
\begin{align} \label{eqn:zeta}
    \zeta(\varepsilon):=\exp(- \tfrac{4 \pi}{3}N \varepsilon^3)
\end{align} 
due to the spatially Poisson distributed initial data. 
\end{theorem}

This is a direct adaption of \cite[Theorem 4.6.]{Matthies2018} and \cite{Stone2017} for the version of Rayleigh in \eqref{eq:partic}. The same approach was also used in \cite{Matthies2018a,Egginton2017,Egginton2018}. 
Following \cite[Prop 8.2.1]{Gallagher2013}, the collision intervals for the empiric can be bounded by $C \varepsilon (\mathcal{V}(\Phi) + \frac{1}{\iota(\Phi)} ) \leq C  \varepsilon (V(\varepsilon)  + \frac{1}{J(\varepsilon)})$, if this bound is considerably smaller than $\vartheta(\Phi)$, e.g. if
\begin{align}
    \label{eqn:notriple}
    C  \varepsilon (V(\varepsilon)  + \frac{1}{J(\varepsilon)}) \ll I (\varepsilon),
\end{align}
then $\Phi$  will only undergo binary collisions.

There are two substantial simplification compared to \cite{Matthies2018}: firstly due to the Poisson set-up we do not have terms due to conditioning on the history or due to a reduction of available particles and similarly initial overlaps can be calculated directly. In the next steps we use that
the evolution defined by \eqref{emp-eq} is linear in $f_0$. We first rescale with the factor $\zeta(\varepsilon)^{-1}$ and get a simple comparison between $P_t(\Phi)$ and the rescaled $\hat{R}_t (\Phi):=\zeta(\varepsilon)^{-1}\hat{P_t}(\Phi)$ 
\begin{lemma}
    For $\varepsilon$ sufficiently small and for $\Phi \in \mathcal{G(\varepsilon)}$
    \begin{equation} \label{eqn:hP-P}
        \hat{R_t}(\Phi) \geq P_t(\Phi).
    \end{equation}
\end{lemma}
 \begin{proof}
We use the differential equations \eqref{eq-id} and \eqref{emp-eq} for the idealised and the empirical evolution.  For a fixed $\Phi$ we observe that  
\begin{equation} \label{eq:Q+identity}
    \hat{\mathcal{Q}}^+_t[P_t](\Phi)= \mathcal{Q}^+_t[\hat{R_t}](\Phi)
\end{equation}
and that if we write
     \begin{align*}
     \hat{\mathcal{Q}}^-_t[\hat{R_t}](\Phi)&= \hat{L}_t(\Phi)\hat{R_t} (\Phi), \\
     \mathcal{Q}^-_t[P_t](\Phi)&= L_t(\Phi) P_t (\Phi),
     \end{align*}
where $$\hat{L}_t(\Phi):= \mathds{1}_{t\geq \tau(\Phi)+ T^*(\Phi)}    
\int_{\mathbb{S}^2} \int_{\mathds{R}^3}  g_0(\bar{v}) \mathcal{B}(v(\tau) - \bar{v}, \omega) 
\mathds{1}^{\varepsilon}_t[\Phi] (x(t)+ \varepsilon \omega, \bar{v})
\, \d \bar{v} \, \d \omega$$
and 
$$ L_t(\Phi) := \int_{\mathds{R}^3} \int_{\mathbb{S}^2} g_0(\bar{v})  \mathcal{B}(v^{\varepsilon}(\tau) - \bar{v} ,\omega)\, \d{\omega}\, \d \bar{v},$$
then for $t>\tau(\Phi)$ \[  L_t(\Phi) \geq \hat{L}_t(\Phi). \]
Then for a fixed $\Phi$ and  $t>\tau(\Phi)$ we obtain a scalar ordinary differential inequality for  $\hat{R_t}(\Phi) -P_t(\Phi)$
\begin{align*}
    \partial_t  (\hat{R_t}(\Phi) -P_t(\Phi))&=- c\hat{L}_t(\Phi)\hat{R_t} (\Phi)
   + c L_t(\Phi) P_t (\Phi) \\&  
   = -cL_t(\Phi) (\hat{R_t}(\Phi) - P_t(\Phi)) 
+ c[L_t(\Phi) -\hat{L}_t(\Phi)]  \hat{R_t}(\Phi) \\& \geq   -cL_t(\Phi) (\hat{R_t}(\Phi) -P_t(\Phi)). 
\end{align*}
The differential inequality yields then 
\begin{equation} \label{eqn:pineq}
    \hat{R_t}(\Phi) -P_t(\Phi)\geq \exp\left(-c\int_\tau^t L_s(\Phi) \d s\right) (\hat{R_\tau}(\Phi) -P_\tau(\Phi)). 
\end{equation}
So we obtain \eqref{eqn:hP-P} if $\hat{R_\tau}(\Phi) -P_\tau(\Phi)\geq 0$. The $n(\Phi)=0$ holds for $\tau=0$ due to the initial conditions in \eqref{eq-id} and \eqref{emp-eq}. For all other good trees this follows by induction over $n(\Phi)$ with  \eqref{eq:Q+identity} as the gain operator $\mathcal{Q}^+_t[P](\Phi) $ is monotone in $P$.
 \end{proof}   
We then obtain quantitative convergence estimates.
\begin{proposition}
\label{Bound}
There is a uniform constant $C$  such that for every $t \in [0,T]$ and $S \subset \mathcal{MT}$ measurable,
\begin{align*}
 \sup_{S \subset \mathcal{MT}  }|P_t(S)- \hat{P_t}(S)|
&\le 
 C \varepsilon + C \delta^2   + C M(\varepsilon)  \left( \frac{\varepsilon}{\delta}\right)^3
 +
\frac{1}{M(\varepsilon)} (1 + C ct) \\
& \quad +
C M(\varepsilon) 
 \left( \tilde \delta+ \frac{\tilde \eta V(\varepsilon) \varepsilon}{\tilde{\delta}^2}  +  \frac{V(\varepsilon)\varepsilon }{
\tilde{\eta}+ \tilde{\delta} } +  
M(\varepsilon) \left(  \frac{\tilde \eta  V(\varepsilon) \varepsilon}{\tilde \delta} + 
\frac{V(\varepsilon) \varepsilon}{\tilde{\eta}+ \sqrt{\tilde{\delta}}  }   \right)\right)\\
 &\quad
+ \frac{M_{f_0}}{V(\varepsilon)}
+(  C_g  + C ct )  \frac{M(\varepsilon) }{V(\varepsilon)} \\
 &\quad + C M(\varepsilon) ( I(\varepsilon) + (J(\varepsilon)^3)
+ (1- \zeta(\varepsilon)). \numberthis \label{bound}
\end{align*}    
\end{proposition} 
\begin{proof}
     Let  $S \subset \mathcal{MT}$ be measurable, then
\begin{align}
P_t(S)- \hat{P_t}(S) &= P_t(S \cap \mathcal{G(\varepsilon)}) + P_t(S \setminus \mathcal{G(\varepsilon)}) - \hat{P_t}(S \cap \mathcal{G(\varepsilon)}) - \hat{P_t}(S \setminus \mathcal{G(\varepsilon)})  \nonumber\\
& \le  P_t(S \cap \mathcal{G(\varepsilon)}) + P_t(S \setminus \mathcal{G(\varepsilon)}) - \hat{P_t}(S \cap \mathcal{G(\varepsilon)}) \nonumber\\&=  P_t(S \setminus \mathcal{G(\varepsilon)})
+ \underbrace{P_t(S \cap \mathcal{G(\varepsilon)})-\hat{R_t}(S \cap \mathcal{G(\varepsilon)})}_{\leq 0}  + (\hat{R}_t(S \cap \mathcal{G(\varepsilon)})-\hat{P_t}(S \cap \mathcal{G(\varepsilon)}) ) \nonumber\\
&\leq P_t( \mathcal{MT} \setminus \mathcal{G(\varepsilon)}) +(1- \zeta(\varepsilon)) \label{eqn:boundbybad}
.
\end{align}
We obtain further bounds  by applying  \eqref{eqn:boundbybad} to complements. 
\begin{align}
\hat{P_t}(S) -P_t(S) &
=  (1 - \hat{P_t}( \mathcal{MT} \setminus S  ) ) - (1- P_t (  \mathcal{MT} \setminus S )  )  \nonumber\\
& =  P_t (  \mathcal{MT} \setminus S ) - \hat{P_t}( \mathcal{MT} \setminus S  )  \nonumber\\
& \leq 
P_t( \mathcal{MT} \setminus \mathcal{G(\varepsilon)})  +(1- \zeta(\varepsilon)) 
\label{eqn:lower}.
\end{align}
By Proposition \ref{2.3.1.} for $\varepsilon$ sufficiently small,
\begin{align}
P_t(\mathcal{MT} \setminus \mathcal{G}(\varepsilon))
& \le
 C \varepsilon + C \delta^2   + C M(\varepsilon)  \left( \frac{\varepsilon}{\delta}\right)^3
 +
\frac{1}{M(\varepsilon)} (1 + C ct)  \nonumber\\
& \quad +
C M(\varepsilon) 
 \left( \tilde \delta+ \frac{\tilde \eta V(\varepsilon) \varepsilon}{\tilde{\delta}^2}  +  \frac{V(\varepsilon)\varepsilon }{
\tilde{\eta}+ \tilde{\delta} } +  
M(\varepsilon) \left(  \frac{\tilde \eta  V(\varepsilon) \varepsilon}{\tilde \delta} + 
\frac{V(\varepsilon) \varepsilon}{\tilde{\eta}+ \sqrt{\tilde{\delta}}  }   \right)\right) \nonumber\\
 &\quad
+ \frac{M_{f_0}}{V(\varepsilon)}
+(  C_g  + \tilde{C}_{q,n} ct )  \frac{M(\varepsilon) }{V(\varepsilon)}
+ C M(\varepsilon) ( I(\varepsilon) + (J(\varepsilon)^3). \label{eq:upper}
\end{align}
Then the estimate \eqref{bound} follows by combining \eqref{eqn:boundbybad} and \eqref{eqn:lower} with \eqref{eq:upper}.
 \end{proof}
\subsection{Making time $T$ large}\label{largeT}
The aim of this section is to make time $T$ large, as $\varepsilon$ tends to be very small. We are doing this by equating the leading order terms of the bound of the probability of bad trees in \eqref{bound} and making them small whereas also ensuring that \eqref{eqn:notriple} holds. Therefore we have the following lemma. 

\begin{lemma}
We can choose the time $T$ and the parameter $c$ in Proposition $\ref{Bound}$ coupled with the form
\begin{align*}
c T = \varepsilon^{\frac{4 m}{3} - \frac{2}{9}}
\end{align*} which becomes large, as $\varepsilon$ becomes small, for the values of $m$ such that $ 0 < m < \frac{1}{6}.$
\end{lemma}
\begin{proof}
In this step, we want to make the time $cT$ large such that $\lim_{ \varepsilon \to 0} cT(\varepsilon) = \infty.$
By considering only the leading order terms in \eqref{bound} and for $\tilde \eta = \tilde \delta$ we have
 \begin{align}\label{balance}
\frac{cT}{M(\varepsilon)}
= M(\varepsilon)\tilde{\delta}
= \frac{M^2(\varepsilon) V(\varepsilon)\varepsilon}{\sqrt{\tilde{\delta}}}
=  \frac{cTM(\varepsilon)}{V(\varepsilon)}.
\end{align}
We solve this system algebraically and start by
\begin{align*}
\frac{cT}{M(\varepsilon)}
=
\frac{c T M(\varepsilon)}{V(\varepsilon)}
&\Leftrightarrow
M(\varepsilon)  =  \sqrt{ V(\varepsilon)} \numberthis \label{M},
\end{align*}
and consider
\begin{align}
M(\varepsilon)\tilde{\delta}
=
\frac{\varepsilon M^2(\varepsilon) V(\varepsilon)}{\sqrt{\tilde \delta}}
&\Leftrightarrow
 \tilde \delta^{\frac{3}{2}}   = \varepsilon M^3(\varepsilon)\Leftrightarrow
 \tilde \delta  =  \varepsilon^{\frac{2}{3}} M^2(\varepsilon). \label{delta}
\end{align}
Then with \eqref{M} and \eqref{delta} we have
\begin{align*}
\frac{cT}{M(\varepsilon)}
=
M(\varepsilon)\tilde{\delta} &
\Leftrightarrow cT=M^4(\varepsilon) \varepsilon^{\frac{2}{3}} \\
 &
\Leftrightarrow
M(\varepsilon) =  (cT)^{\frac{1}{4}} \varepsilon^{-\frac{1}{6}}.
\numberthis \label{V}
\end{align*}
Now we want to make all the terms in \eqref{balance}, i.e. the leading terms of \eqref{bound}, small. We observe that by the choices of $M(\varepsilon)$, $V(\varepsilon)$, $\tilde \eta$, and $\tilde \delta$
above, all the terms in \eqref{balance} are equal to
$$\frac{cT}{M(\varepsilon)}
=(cT)^{\frac{3}{4}} \varepsilon^{\frac{1}{6}}.$$
Thus, in order to make it small, we equate it with an $\varepsilon ^{m}$ small, i.e.
\begin{align}\label{7.5}
\varepsilon^{m} =
( cT)^{\frac{3}{4}} \varepsilon^{\frac{1}{6}}
\Leftrightarrow
c T = \varepsilon^{\frac{4 m}{3} - \frac{2}{9}}.
\end{align}
Next, we find the values of $m$ so that $cT$ is large. Thus it needs to satisfy
\begin{align*}
 \frac{4 m}{3} < \frac{2}{9}
 \Leftrightarrow
 0 <   m < \frac{1}{6}.
\end{align*}
The term $0<1-\zeta(\varepsilon) \leq C c \varepsilon$ is smaller than 
$\varepsilon^{1/6}$ for any admissible choice of $c$. It can be easily checked that all further terms in \eqref{bound} are of higher order by choosing $\delta= \sqrt{\varepsilon}$. Similarly, if $I(\varepsilon) = \sqrt{\varepsilon}$ and $J(\varepsilon)= \varepsilon^{1/3}$  then the  $\textnormal{Tri}(\varepsilon)$ and
$\textnormal{Slow}(\varepsilon)$ are also of higher order compared to the
 $\textnormal{Rec}(\varepsilon)$. We note that with this choice \eqref{eqn:notriple} holds. 
\end{proof}

Now, we have all the tools we need to prove Theorem $\ref{thm1}$.
\begin{proof}[Proof of Theorem $\ref{thm1}$]
We consider the set of all collision histories such that at time $t$ the tagged particle is in the set $\Omega\subset \mathds{R}^3 \times \mathds{R}^3$, that is
 $$S_t(\Omega) : = \{ \Phi \in \mathcal{MT} : (x(t),v(t))\in \Omega  \}. $$
By \cite[Theorem 3.1]{Matthies2018} it holds true that $$ \int_{\Omega} f_t(x,v) \, \d x\, \d v = \int_{S_t(\Omega)} P_t(\Phi)\, \d \Phi = P_t(S_t(\Omega))$$ and by the definition of the empirical distribution
 $$ \int_{\Omega} \hat{f}^N_t(x,v)\, \d x \, \d v = \int_{S_t(\Omega)} \hat{P}^{\varepsilon}_t(\Phi)\, \d \Phi = \hat{P}^{\varepsilon}_t(S_t(\Omega)).$$ Then by Theorem \ref{Bound} above, we get that for any $t\in [0,T_{\varepsilon}]$ and for any $ \Omega \subset \mathds{R}^3 \times \mathds{R}^3$
 \begin{align*}
 \lim_{N \to \infty}
\big| \int_{\Omega} \hat{f}^N_t(x,v) - f_t (x,v)\, \d x\, \d v \big|
&=\lim_{\varepsilon \to 0} |\hat{P}^{\varepsilon}_t (S_t(\Omega)) - P_t(S_t(\Omega)) | \\
&\le \lim_{\varepsilon \to 0} \sup_{S \subset \mathcal{MT}  }
| \hat{P}^{\varepsilon}_t (S) - P_t(S) | =0,
 \end{align*}
which completes the proof of the theorem, since for any $t\in [0,T_{\varepsilon}]$
\begin{align*}
\| \hat{f}^N_t - f_t \|_{L^1(\R^3 \times \R^3)}
&= 2 \int_{f_t^N \ge f_t}
\hat{f}^N_t (x,v) - f_t(x,v)\, \d x \, \d v \\
&\le
2 \sup_{\Omega \subset \R^3 \times \R^3  }
\big| \int_{\Omega} \hat{f}^N_t(x,v) - f_t (x,v)\, \d x\, \d v \big| \to 0,
\end{align*}
as $N \to \infty$, by above.
\end{proof}
\section{From Linear Boltzmann to the fractional diffusion equation}\label{pthm2}
The aim of this section is to give the proof of Theorem \ref{thm2}. We are doing so by adapting results from \cite{Mellet2011} and combine them with Theorem \ref{thm1}. First, we are splitting the linear Boltzmann operator into the gain and loss term and write the gain term via the Carleman representation in order to integrate over the $v_2'\in E_{v_1,v_1'}$ and $v_1' \in \mathds{R}^3$ variables instead of $\omega\in \mathbb{S}^2$ and $v_2\in \mathds{R}^3$. Then, we are finding the behaviour of the gain and loss term, $K(f)(v_1)$ and $\nu(v_1)$ respectively, for large velocities $|v_1|$ and for $\Theta \in (0, \frac{\pi}{2})$. Additionally, we appropriately choose $\alpha$ and $\beta$ so that \cite[Theorem 3.2]{Mellet2011} can be applied, which is a theorem which derives a fractional diffusion equation from the linear Boltzmann equation. Finally, we combine this result with Theorem \ref{thm1} and obtain the proof of the full derivation, i.e. the proof of Theorem \ref{thm2}.
The gain term of the linear Boltzmann collision operator is
\begin{align*}
Q^+(f)(v_1) = \int_{\R^3} \int_{\mathbb{S}^2} f(v_1') g_0(v_2') \ \mathcal{B} (v_1-v_2, \omega) \, \d  \omega \, \d v_2,
\end{align*}
which represents the probability that a particle after collision has velocity $v_1$. Furthermore, $\omega$ is the orthonormal vector and gives the direction of the apse line, which is the line that bisects the angle between $v_1-v_2$ and $v_1' - v_2'$.
The idea of Carleman representation, first introduced by Carleman \cite{carleman57}, is to parameterise the gain term $Q^+$ by the variables $v_1'$, $v_2'$ instead of $v_2$, $\omega$.~Rather than parameterising the sphere $\mathbb{S}_{v_1,v_2}$, which is defined as the set of all possible $v'$ which result from a collision between one particle
with velocity $v_1$ and another particle with velocity $v_2$, i.e. the sphere with a diagonal the line segment which connects $v_1$ and $v_2$, one considers the set $E_{v_1,v_1'}$ which is the set of all $v_2'$ which could be the velocity after collision of one particle, if the other velocity after collision is $v_1'$ and one of the velocities before collision is $v_1$. See Fig. \ref{figCarl}. The set $E_{v_1,v_1'}$ is the hyperplane orthogonal to $v_1-v_1'$ and containing $v_1$ and we are writing it as the set 
\begin{align}
    E_{v_1,v_1'} : = \{ v_2' \in \R^3 \ : \ (v_2'-v_1)\cdot (v_1-v_1') = 0 \}, \label{eqn:carlemandef}
\end{align}
where $(v_2-v_1)\cdot (v_1-v_1')$ is the usual inner product such that 
\[(v_2-v_1)\cdot (v_1-v_1') = \langle v_2-v_1, v_1-v_1' \rangle = |v_2-v_1| |v_1-v_1'| \cos (v_2-v_1, v_1-v_1').\]
See \cite[Section 3]{Wennberg1997} for a further discussion of this.

\begin{figure}[h!]
\begin{center}
\begin{tikzpicture}[
dot/.style={
  fill,
  circle,
  inner sep=1pt
  }
]

\begin{scope}[x=2cm,y=2cm]
  
\draw[-]  (-1.5,-1/2) -- (1.8,-1/2);  

\draw (0,0) circle [radius=1];

   \node[above right] at ( 0.9,0.4 )  { $v_1'$ };
  \node[ below right] at (0.9, -0.55) {$v_1$};
   \node[ above left] at (-0.95,0.4) {$v_2$};
  \node[below left] at (  -0.99,-1/2 ) {$v_2'$};
  
   \node[above] at ( 1.55, -1/2 ) {$E_{v_1, v_1'}$};

     \node[dot] at (0.87, -0.5)  {}; 
     
      \node[dot] at ( 0.89,0.45 )  {}; 
      
        \node[dot] at (-0.89,0.45) {}; 
      
       \node[dot] at (-0.865, -0.5)  {}; 
     
    \draw [-> ] (0.87, -0.5)  --  ( 0.89,0.45 )  ;

  \draw [->, line width=0.9] (0.87, -0.5)  --  ( 0.89,0.12 )  ;
  
   \node[left] at ( 0.89,0.09 )   { $\omega$ };

    \draw [-> ] ( -1.1,-1.7)  --  ( -0.865, -0.5 )  ;

     \draw [-> ] ( -1.1,-1.7)  --  ( -0.89,0.45 )  ;  
     
       \draw [-> ] ( -1.1,-1.7)  --  (  0.89,0.45)  ;  
       
        \draw [-> ] ( -1.1,-1.7)  --  (0.87, -0.5)  ;  
     
         \draw [- ] ( -0.89,0.45 )  --  (0.87, -0.5)  ;

 \node[above] at ( -0.8,0.8 ) {$\mathbb{S}_{v_1, v_2}$};

\end{scope}
\end{tikzpicture}
\end{center}\caption{Picture of the hyperplane $E_{v_1, v_1'}$.} \label{figCarl}
\end {figure}

In order to represent the gain term with the Carleman representation, we set $v_1' = v_1 + q \omega$, $q\in \R$ and $\omega \in \mathbb{S}^2$. Now, for any fixed $\omega$ and $v_1$ we can write $v_2' = v_2 - q \omega$, so $v_2 = v_2' + q \omega$. Then, with the change of variables $v_2 = v_2' + q \omega$, $\d v_2 = \d v_2' \d q$. Furthermore, by taking the polar coordinates of $v_1'$ centred at $v_1$ we get $$\d v_1' = q^2 \d q \d \omega.$$
So, by using the above formulations we have$$\d \omega \d v_2 = \d \omega \d v_2' \d q =\d v_2' \d q \d \omega = \frac{1}{q^2} \d v_2' \d v_1', $$
where $\d v_2'$ is the Lebesgue measure on the hyperplane $E_{v_1, v_1'}$ and $\d v_1'$ is the Lebesgue measure in $\R^3$. Also, we make the observation $q^2 = |q \omega |^2 = | v_1-v_1'|^2 $ using the formula for $v_1'$ above, see \cite[Section 2]{Wennberg1994}. Then the Carleman representation of the gain term is
 \begin{align*}
Q^+(f)(v_1) = 2 \int_{\R^3} \frac{f(v_1')}{|v_1-v_1'|^2}  \int_{E_{v_1,v_1'}} g_0( v_2') \ \mathcal{B}(v_1-v_2, \omega) \, \d v_2' \, \d v_1'.
\end{align*}
For an extended explanation of the change of variables see \cite[Appendix A]{Silvestre2016}.
Now in order to have the variables $v_1'$, $v_2'$ inside the collision kernel $\mathcal{B}$, we use the identity $v_1-v_2=2v_1-v_1'-v_2'$ which comes from the conservation of momentum and we take
 \begin{align*}
Q^+(f)(v_1) = 2 \int_{\R^3} \frac{f(v_1')}{|v_1-v_1'|^2}  \int_{E_{v_1,v_1'}} g_0( v_2') \ \mathcal{B}(2v_1-v_1'-v_2', \omega) \, \d v_2' \, \d v_1'.
\end{align*}
The factor $2$ is due to the fact that each plane is represented by two opposite directions $\omega$. Also, by the equality $v_1' = v_1 + q \omega$ we take that $\omega = \frac{v_1' - v_1}{q} = \frac{v_1' - v_1}{|v_1'-v_1|} \in \mathbb{S}^2$. Another reference for the Carleman representation is \cite[Section 1.4]{Villani2002}.

By the above analysis for the Carleman representation, the linear Boltzmann collision operator $Q$ can be written as
\begin{align}\label{LBO}
Q(f)(v_1)&:= \int_{\R^3} [\sigma(v_1,v_1')f(v_1') - \sigma(v_1',v_1)f(v_1)]\, \d v_1'
= K(f)(v_1) - \nu(v_1) f(v_1),
\end{align}
where the gain term is
\begin{align}
K(f)(v_1) &= \int_{\R^3} \sigma(v_1,v_1')f(v_1') \, \d v_1' \nonumber\\
& = \int_{\R^3} \left[ \frac{2}{|v_1 - v_1'|^2} \int_{E_{v_1,v_1'}} g_0(v_2') \mathcal{B} (2v_1-v_1'-v_2', \omega) \, \d v_2' \right] f(v_1')\, \d v_1' \label{eqn:defK}
\end{align}
and the loss term is
\begin{align}
\nu (v_1) &= \int_{\R^3} \sigma(v_1',v_1) d v_1' =
\int_{\R^3} \int_{\mathbb{S}^2} g_0(v_1') \mathcal{B} (v_1-v_1', \omega) \, \d \omega \, \d v_1'. \label{eqn:defnu}
\end{align}
\subsection{Estimate of the loss term of the collision operator}\label{S5.2}
In this subsection, we analyse the behaviour of the loss term $\nu$ of the Boltzmann collision operator $Q$ as the velocity magnitude $|v_1|\to \infty$. 
Our goal is to establish Proposition \ref{L5.2}, which characterizes the asymptotic decay of $\nu$ in the high-velocity regime. To proceed, we first state the following lemma which uses the notion of unimodal distribution. Unimodality refers to a distribution that has a maximum at $0$ and is decaying in any radial direction.

\begin{lemma}\label{2.4.} Let $f \in L^1(\R^3)$ be a radial and unimodal distribution. Suppose in addition that $f(v)(1+|v|)\in L^1(\R^3)$. Then, for any $\lambda < 0,$
\begin{align}
\int_{\R^3} f(v_1') |v_1 - v_1'|^{\lambda } \, \d v_1'
=
\int_{\R^3} f(v_1') |v_1 |^{ \lambda} \, \d v_1' + o(|v_1|^{\lambda }), \quad \text{as} \ |v_1| \to \infty.
\end{align}
\end{lemma}


\begin{proof}
We first note that $f\in L^\infty(\R^3\setminus B_1(0))$ 
where $B_1(0)$ is the unit ball as $f$ is a radial function which decays monotonically to $0$.

 Consider first the integral on the ball $B(v_1, \frac{|v_1|}{2})$, i.e.
\begin{align*}
\int_{ B(v_1, \frac{|v_1|}{2})  } |v_1|^{- \lambda} f(v_1') |v_1 - v_1'|^{\lambda }\, \d v_1'
&=
\int_{ B(v_1, \frac{|v_1|}{2}) }
f(v_1')
\left|\frac{v_1}{|v_1|} - \frac{v_1'}{|v_1|}\right|^{\lambda }\, \d v_1' \\
& \le
\|f\|_{L^{\infty}( B(v_1, \frac{|v_1|}{2}))} \int_{B(0, \frac{|v_1|}{2})}
\left|\frac{v_1'}{|v_1|}\right|^{\lambda }\, \d v_1'\\
&=
4\pi \|f\|_{L^{\infty}( B(v_1,\frac{|v_1|}{2}))}
\int_{0}^{\frac{|v_1|}{2}}
\frac{r^{\lambda + 2}}{|v_1|^{\lambda}} \, \d r\\
& =
\frac{4\pi}{ (\lambda + 3) 2^{\lambda + 3}} \|f\|_{L^{\infty}( B(v_1,\frac{|v_1|}{2}))} |v_1|^3,
\numberthis \label{2.7.}
\end{align*}
by H\"older's inequality and change of coordinates. This term is bounded for $|v_1|>2$ since $f \in L^{\infty}(\R^3\setminus B_1(0))$ and is a radial and  unimodal distribution.

Furthermore, consider the integral on the set $B(0, (1-\eta)|v_1|) $, for $\eta >0$ sufficiently small. Then
\begin{align*}
\int_{ B(0, (1-\eta)|v_1|) } |v_1|^{- \lambda} f(v_1') |v_1 - v_1'|^{\lambda }\, \d v_1'
& =  \int_{ B(0, (1-\eta)|v_1|) }
 f(v_1')   \frac{ 1 }{ | \e_1 - \frac{v_1' }{|v_1|}|^{|\lambda|}}  \, \d v_1' \\
 & \to  \int_{\R^3}  f(v_1')\, \d v_1'
\end{align*}
in the limit $|v_1|\to \infty$ which is finite since $f \in L^1(\R^3)$. Here we used the dominated convergence theorem, since the function $ f(v_1') \frac{ 1 }{ | \e_1 - \frac{v_1' }{|v_1|}|^{|\lambda|}}$ is dominated by $f(v_1') \frac{1}{\eta^{|\lambda|}}$, on the set $B(0, (1-\eta)|v_1|) )$.

Now, we consider the same integral but on the set $\R^3 \setminus B( 0, (1+\eta) |v_1|)$. That is
\begin{align*}
\int_{ \R^3 \setminus B( 0, (1+\eta) |v_1|) } |v_1|^{- \lambda} f(v_1') |v_1 - v_1'|^{\lambda } \, \d v_1'
&\le
\eta^{\lambda}
\int_{\R^3 \setminus B( 0, (1+\eta) |v_1|)} |v_1|^{- \lambda} f(v_1') |v_1|^{\lambda }\, \d v_1'\\
& =
\eta^{\lambda}
\int_{\R^3 \setminus B( 0, (1+\eta) |v_1|)}  f(v_1')\, \d v_1',
\end{align*}
which is finite since $f\in L^1(\R^3)$. Here we used the inverse triangle inequality and the fact that $v_1' \in  \R^3 \setminus B( 0, (1+\eta) |v_1|)$ and that $\lambda <0$.
Finally, we consider the integral on the set $\Omega := B(0, (1+\eta)|v_1|) \setminus  B(0, (1-\eta)|v_1|)$. Then
\begin{align*}
\int_{ \Omega} |v_1|^{- \lambda} f(v_1') |v_1 - v_1'|^{\lambda }\, \d v_1'
& =
\int_{ \Omega \cap B(v_1, \frac{|v_1|}{2} )} |v_1|^{- \lambda} f(v_1') |v_1 - v_1'|^{\lambda } \, \d v_1'  \\
&+
\int_{ \Omega \setminus B(v_1, \frac{|v_1|}{2})  } |v_1|^{- \lambda} f(v_1') |v_1 - v_1'|^{\lambda }\, \d v_1',
\end{align*}
where
\begin{align*}
\int_{ \Omega \cap B(v_1, \frac{|v_1|}{2} )} |v_1|^{- \lambda} f(v_1') |v_1 - v_1'|^{\lambda } \, \d v_1'
& \le
\int_{ B(v_1, \frac{|v_1|}{2} )} |v_1|^{- \lambda} f(v_1') |v_1 - v_1'|^{\lambda }\, \d v_1' \\
&\le
\frac{4\pi}{ (\lambda + 3) 2^{\lambda + 3}} \|f\|_{L^{\infty}( B(v_1,\frac{|v_1|}{2}))} |v_1|^3,
\end{align*}
by the relation $\eqref{2.7.}$.
Also
\begin{align*}
\int_{ \Omega \setminus B(v_1, \frac{|v_1|}{2})  } |v_1|^{- \lambda} f(v_1') |v_1 - v_1'|^{\lambda }\, \d v_1'
&=
 \int_{ \Omega \setminus B(v_1, \frac{|v_1|}{2})  }
 f(v_1')   \frac{ 1 }{ | \e_1 - \frac{v_1' }{|v_1|}|^{|\lambda|}}  \, \d v_1'\\
 & \le
 2 \int_{ \Omega \setminus B(v_1, \frac{|v_1|}{2})  }
 f(v_1') \, \d v_1',
\end{align*}
which is finite since $f\in L^1(\R^3)$. We used the fact that the term $ \frac{ 1 }{ | \e_1 - \frac{v_1' }{|v_1|}|^{|\lambda|}}$ is bounded by $2$ on the set $\Omega \setminus B(v_1, \frac{|v_1|}{2})$.
This completes the proof of the lemma.
\end{proof}


\begin{proposition}\label{L5.2}
The function  $\nu$ is bounded and is behaving as
\begin{align}\label{nu}
\nu (v_1)
 & =
 \tilde{C}
   \bigg[ |v_1|^{\frac{n-5}{n-1}} 
   +
   o(|v_1|^{\frac{n-5}{n-1}})
   \bigg], 
\end{align}
as $|v_1|\to \infty$, $q\in (4, \frac{4n}{n-1})$ and $n\in (3,5]$.
\end{proposition}
\begin{proof}
The formula for the loss term of the collision operator is given by
\begin{align*}
\nu (v_1) &= \int_{\mathbb{S}^2} \int_{\R^3}g_0 (v_1') \mathcal{B} (v_1-v_1', \omega) \, \d v_1'  \, \d \omega,
\end{align*}
then boundedness for bounded $v_1$
follows from the boundedness of $\mathcal{B}$ implied by Lemma \ref{lem:Bbounded} and Corollary \ref{CorCOB}. The estimates follow from 
\begin{align*}
\nu (v_1)&=
    \frac{2^{\frac{2}{n-1}}}{ a^2}
   \int_{\mathbb{S}^2} \int_{\R^3}
g_0(v_1')
\bigg[ |v_1-v_1'|^{\frac{n-5}{n-1}} + \mathcal{O}(|v_1-v_1'|^{\frac{n-7}{n-1}})
\bigg] \Theta
(\sin \Theta)^{-1}
\, \d v_1'\, \d \omega \\
 &=
   \frac{2^{\frac{2}{n-1}}}{ a^2}
   \int_{\mathbb{S}^1}
   \int_{\R^3}
g_0(v_1')
\bigg[ |v_1-v_1'|^{\frac{n-5}{n-1}} + \mathcal{O}(|v_1-v_1'|^{\frac{n-7}{n-1}})
\bigg]
  \int_{0}^{\frac{\pi}{2}} \Theta  \, \d\Theta
\, \d v_1' \, \d\psi \\
 & =
  \frac{2^{\frac{2}{n-1}}}{ a^2}  \pi \frac{\pi}{2}^2
   \int_{\R^3}
g_0(v_1')
\bigg[ |v_1-v_1'|^{\frac{n-5}{n-1}} + \mathcal{O}(|v_1-v_1'|^{\frac{n-7}{n-1}})
\bigg]\, \d v_1'  \\
 & =
 \underbrace{  \frac{2^{\frac{2}{n-1}}}{ a^2}
   \pi \frac{\pi}{2}^2}_{ =: \tilde{C}}
   \bigg[  \left[ |v_1|^{\frac{n-5}{n-1}}
+ \mathcal{O}(|v_1|^{\frac{n-7}{n-1}})
\right]
\underbrace{ \int_{\R^3}
g_0(v_1')
\, \d v_1' }_{=1}
 +  o(|v_1|^{\frac{n-5}{n-1}}) \bigg].
\end{align*}
Here we used the definition of the collision kernel \eqref{2.4} 
and Lemma \ref{2.4.} for $f=g_0$, and for $\lambda = \frac{n-5}{n-1}$. 
\end{proof}

\subsection{Estimate of the gain term of the collision operator}\label{S5.1}
The goal of this subsection is to state and prove a series of lemmas which are essential in order to find the behaviour of the gain term $K(f)$ of the Boltzmann collision operator $Q$ for large velocities, i.e. $|v_1|\to \infty$. 
This behaviour is presented in Proposition \ref{L5.1}. To proceed, we start with the following lemma. 

\begin{lemma} \label{lem:convolest}
There exists $C$ such that the function   satisfies
 \begin{align} \nonumber 
     \sigma(v_1,v_1')=& \frac{2}{|v_1 - v_1'|^2} \int_{E_{v_1,v_1'}} g_0(v_2') \mathcal{B} (2v_1-v_1'-v_2', \omega)\, \d v_2' \\
     \leq& \frac{C}{|v_1 - v_1'|^2} \min \bigl( 1,  \left|v_1 \cdot\frac{v_1-v_1'}{|v_1-v_1'|} \right|^{-q+2} \bigr).  \label{eqn:convolest}
 \end{align}   
\end{lemma}
\begin{proof}
   From Lemma \ref{lem:Bbounded} and Proposition \ref{P3.2} we know that   $\mathcal{B}$ is uniformly bounded. 
   Due to the specific form of $g_0$ the maximal integral over any plane is bounded by the integral over a plane through the origin
$$\int_{E_{v_1,v_1'}} g_0(v_2') \, \d v_2' \leq \int_{\R^2\times \{0 \}} g_0(v) \, \d v < \infty, $$
   which implies the constant bound.

To get better estimates for large $|v_1|$, we analyse the Carleman plane $E_{v_1,v_1'}$ more carefully. The point closest to the origin on  $E_{v_1,v_1'}$ is a multiple of the normalised normal vector
$\frac{1}{|v_1-v_1'|} (v_1-v_1')$ and satisfies the condition in \eqref{eqn:carlemandef}. Hence it is given by
\begin{align}
    \label{eqn:vmin}
    v_{min} (v_1,v_1')= \left(v_1 \cdot\frac{v_1-v_1'}{|v_1-v_1'|} \right) \frac{v_1-v_1'}{|v_1-v_1'|}. 
\end{align}
If $|v_{min}|> \bar{R}$, we improve the estimate for the integral. Let 
$e_1, e_2$ be two orthonormal vectors in $\R^3$ that satisfy $e_1, e_2 \perp (v_1-v_1').$ Then, we parametrise the plane $E_{v_1,v_1'}$  in polar coordinates $v_2':= v_{min} + r v_s$, $r \in (0, \infty)$ and $v_s := \cos s \ e_1 + \sin s \ e_2,$ $s \in (0, 2\pi)$.
\begin{align*}
   \int_{E_{v_1,v_1'}} g_0(v_2') \, \d v_2'
   =& \int_0^\infty \int_0^{2\pi} g_0(v_{min} + rv_s) \, \d s \, r \d r =
   \int_0^\infty \int_0^{2\pi} C_2 \left(\sqrt{|v_{min}|^2 + r^2 } \right)^{-q} \, \d s \, r \d r\\
   =& 2 \pi C_2 \int_0^\infty  \left(\sqrt{|v_{min}|^2 + r^2 } \right)^{-q}    r \,   \d r\\
=& 2 \pi C_2  \left[ \frac{-1}{(q-2 )} 
\left(\sqrt{|v_{min}|^2 + r^2 } \right)^{-q+2}
\right]_0^\infty =\frac{2\pi C_2}{ (q-2 )}  
 |v_{min}|^{-q+2}, 
 \end{align*}
as required.
\end{proof}

\begin{lemma} \label{lem:sharperconvolution}
Let $f \in L^\infty(\R^3)$, then for $|v_1|>1$ we restrict integration to the ball $B_L(v_1)$ of radius $L$ around $v_1$ and define
\begin{align}
    H(v_1, L,f) = &\int_{\R^3} \sigma(v_1,v_1')
\chi_{|v_1-v_1'|\leq L} f(v_1') \, \d v_1' 
    \label{eqn:H}.
\end{align}    
Then there exists a uniform constant $C$ such that
\begin{align*}
    |H(v_1, L,f)| \leq C L \frac{1}{|v_1|} \|f \|_{L^\infty(B_L(v_1))}.    
\end{align*}
\end{lemma}

\begin{proof}
    We use the estimates in Lemma \ref{lem:convolest}. We consider spheres in $v_1'$ around $v_1$ with radius
$|v_1-v_1'|=  \ell$ for $0 \leq \ell \leq L$. Taking $v_1$ as pointing to the north pole, we use spherical polar coordinates, $(\phi, \theta)
\in [0,2\pi) \times (0, \pi)$.   For  a given sphere the condition for $v_{min}$ as in \eqref{eqn:vmin} 
\[ \left(v_1 \cdot\frac{v_1-v_1'}{|v_1-v_1'|} \right) =\cos(\theta) |v_1| \mbox{ with } \theta \in (0,\pi) \]
defines a circle of radius $\ell \sin \theta$ such that the right hand side of \eqref{eqn:convolest} is constant and, as $|v_{min}|
= |\cos(\theta)|  |\ell| |v_1|$, is given by 
\[ I(\ell, \theta)=  \frac{C}{\ell^2} \begin{cases} 1 & |\cos(\theta)|\leq \frac{\bar{R}}{|v_1|}  \\ \frac{2 \pi C_2}{q-2}(|\cos \theta| |v_1|)^{-q+2} & \frac{\bar{R}}{|v_1|} \leq |\cos(\theta)|\leq 1.
\end{cases}\]
Taking the volume integral
\begin{align*}
    |H(v_1,L)| \leq & \int_{B_L(v_1)}
\sigma(v_1,v_1')
 \, \d v_1' \|f \|_{L^\infty(B_L(v_1))},
\end{align*}
we get by letting $u =\cos (\theta)$
\begin{align*}
   \int_{B_L(v_1)}
\sigma(v_1,v_1')
 \, \d v_1' &= \int_0^L \int_0^{2 \pi} \int_{0}^{\pi} I(\ell, \theta) \ell^2 \sin(\theta) \, \d \theta \, \d \phi \, \d \ell\\
& = C L  4 \pi  \left(\int_{0}^ {\frac{\bar{R}}{|v_1|}} 1 \, \d u + \int_{\frac{\bar{R}}{|v_1|}}^1 \frac{2\pi C_2}{q-2}(u |v_1|)^{-q+2} \, \d u  \right) \\
&= C L  4 \pi \left( \frac{\bar{R}}{|v_1|} +
\frac{2 \pi C_2}{q-2}|v_1|^{-q+2} \left[ \frac{u^{-q+3}}{3-q}\right]_{\frac{\bar{R}}{|v_1|}}^1
\right) \\
&= C L  4 \pi \left( \frac{\bar{R}}{|v_1|} +
\frac{2 \pi C_2}{(q-2) (q-3)} \left[ \frac{\bar{R}^{-q+3}}{|v_1|} - |v_1|^{-q+2} \right] \right),  
\end{align*}
which yields the bound.    
\end{proof}

\begin{lemma}\label{integrand:sigma} 
 Let  $K,L >0$. Then  for $|v_1|>K$
the integrand of the gain term can be represented 
as 
 \begin{align} \nonumber 
     \sigma(v_1,v_1')=& \frac{2}{|v_1 - v_1'|^2} \int_{E_{v_1,v_1'}} g_0(v_2') \mathcal{B} (2v_1-v_1'-v_2', \omega)\, \d v_2'
   \\  = & \frac{2^{\frac{2}{n-1}}}{ a^2} 4\pi C_2  |v_1|^{-q +\frac{n-5}{n-1}} \Bigl[ 
 \int_{0}^{\infty}
 \rho (1+\rho^2)^{\frac{n-5}{2(n-1)} -\frac{q}{2}} \frac{\arctan \rho}{ \sin (\arctan \rho)} \d \rho + R(v_1,v_1') \Bigr], \label{eq:sigma}
 \end{align}   
  where $R(v_1,v_1')$ is uniformly bounded for  $|v_1-v_1'|>L$, continuous and
 $R(v_1,v_1') \to 0 $ as $|v_1| \to \infty$ pointwise in $v_1'$.
\end{lemma}



\begin{proof}
We start with the integrand of the gain term 
\begin{align*}
\sigma(v_1, v_1') = \frac{2}{|v_1 - v_1'|^2} \int_{E_{v_1,v_1'}} g_0(v_2') \mathcal{B} (2v_1-v_1'-v_2', \omega)\, \d v_2',
\end{align*}
where $E_{v_1,v_1'} = \{ v_2' \in \R^3 \ : \ (v_2'-v_1)\cdot (v_1-v_1') = 0 \}$. 
We observe that if $|v_1 - v_1'| \to \infty$ then $|2v_1 - v_1' - v_2'| \to \infty$. That is because we can write 
\begin{align*}
    |2v_1 - v_1' - v_2'|^2 &= |v_1 - v_1'|^2 + 2(v_1 - v_1')\cdot(v_1 - v_2') + |v_1 - v_2'|^2\\
    & = |v_1 - v_1'|^2 + |v_1 - v_2'|^2,
\end{align*}
since $v_2' \in E_{v_1,v_1'}$. So, by the above observation, we are able to utilise Proposition \ref{P3.2}. 
Let us now estimate the integrand of $K(f)(v_1)$, that is 
\begin{align*}
\int_{E_{v_1,v_1'}} g_0(v_2') \mathcal{B}& (2v_1-v_1'-v_2', \omega)\, \d v_2' \\
& =
\frac{2^{\frac{2}{n-1}}}{ a^2} \bigg[
\int_{E_{v_1,v_1'}} g_0(v_2')
|2v_1-v_1'-v_2'|^{\frac{n-5}{n-1}} \Theta
(\sin \Theta)^{-1}\, \d v_2' \\
& \quad+
 \int_{E_{v_1,v_1'}} g_0(v_2')
\mathcal{O}(|2v_1-v_1'-v_2'|^{\frac{n-7}{n-1}}) \Theta
(\sin \Theta)^{-1}\, \d v_2' \bigg], \quad \text{as} \ |v_1 - v_1'|\to \infty.
\end{align*}
We know that
\begin{align*}
\frac{v_1-v_1'+v_1-v_2'}{|2v_1-v_1'-v_2'|} \cdot  \frac{v_1-v_1'}{|v_1-v_1'|} = \cos \Theta
& \Leftrightarrow \frac{|v_1-v_1'|^2 + (v_1-v_2')\cdot (v_1-v_1')}{|2v_1-v_1'-v_2'| |v_1-v_1'|} = \cos \Theta \\
& \Leftrightarrow |v_1-v_1'| = \cos \Theta |2v_1-v_1'-v_2'|\\
& \Leftrightarrow 
|v_1-v_1'|^2 = \cos^2 \Theta (|v_1-v_1'|^2 + |v_1 - v_2'|^2)\\
& \Leftrightarrow 
|v_1-v_2'|^2 = \tan^2 \Theta |v_1-v_1'|^2,
\end{align*}
where we used the fact that $ (v_1-v_2')\cdot (v_1-v_1')=0$, since $v_2' \in E_{v_1,v_1'}$. Furthermore, since the angle $\Theta$ is bounded above by $\frac{\pi}{2}$ and the function $\tan x$ is strictly increasing in $(0,\frac{\pi}{2})$, we conclude that $v_2'$ lies within a ball centred at $v_1$ with arbitrary large radius, or equivalently, the ball can be considered to have infinite radius.
Hence, we choose $e_1, e_2$ two orthonormal vectors in $\R^3$ that form a basis of the plane $E_{v_1, v_1'}$. This implies that $e_1, e_2 \perp (v_1-v_1').$ Then, we parametrise the plane in polar coordinates $v_2':= v_1 + r v_s$, $r \in (0, \infty)$ and $v_s := \cos s \ e_1 + \sin s \ e_2,$ $s \in (0, 2\pi)$ and the integral becomes
\begin{align*}
&\int_{E_{v_1,v_1'}} g_0(v_2') \mathcal{B} (2v_1-v_1'-v_2', \omega)\, \d v_2'
\\
& =
\frac{2^{\frac{2}{n-1}}}{ a^2}\bigg[
 \int_{0}^{\infty} \int_{0}^{2\pi}g_0(v_1 + r v_s)
|v_1-v_1'-rv_s|^{\frac{n-5}{n-1}} \Theta(r)(\sin \Theta(r))^{-1} \,
 \d s \,r \, \d r \\
 & \quad +
 \int_{0}^{\infty} \int_{0}^{2\pi}g_0(v_1 + r v_s)
\mathcal{O}(|v_1-v_1'-rv_s|^{\frac{n-7}{n-1}})\Theta(r)(\sin \Theta(r))^{-1} \,
 \d s \,r \, \d r \bigg]
 \\
& =
\frac{2^{\frac{2}{n-1}}}{ a^2}\bigg[
 \int_{0}^{\infty} \int_{0}^{2\pi}g_0(v_1 + r v_s)
(|v_1-v_1'|^2 +r^2 )^{\frac{n-5}{2(n-1)}} \Theta(r)(\sin \Theta(r))^{-1} \,
 \d s \,r \, \d r \\
 & \quad +
 \int_{0}^{\infty} \int_{0}^{2\pi}g_0(v_1 + r v_s)
\mathcal{O}((|v_1-v_1'|^2 +r^2 )^{\frac{n-7}{2(n-1)}})\Theta(r)(\sin \Theta(r))^{-1} \,
 \d s \,r \, \d r \bigg].
\end{align*}
Now, we use the change of variables 
$r \mapsto \rho := \frac{r}{|v_1 - v_1'| } $ and we obtain 
\begin{align*}
\frac{2^{\frac{2}{n-1}}}{ a^2}|v_1-v_1'|^2
\bigg[ \int_{0}^{\infty} \big[ |v_1 &- v_1'|^{\frac{n-5}{n-1}}
 \rho (1+\rho^2)^{\frac{n-5}{2(n-1)}} 
 + \mathcal{O} (|v_1 - v_1'|^{\frac{n-7}{n-1}})
 \rho (1+\rho^2)^{\frac{n-7}{2(n-1)}} \big] \\&
 \times  \frac{\Theta(\rho |v_1 - v_1'|)}{\sin \Theta(\rho |v_1 - v_1'|) } \int_{0}^{2\pi}g_0(v_1 + |v_1 - v_1'| \rho v_s) \,
 \d s \, \d \rho 
 \bigg]
\end{align*}
and we observe that 
$$ |v_1 - v_1'|^2 \frac{\sin^2 \Theta}{\cos ^2 \Theta} = |v_1 - v_2'|^2 = r^2 = \rho^2 |v_1 -v_1'|^2 $$
which implies that $$ \tan \Theta = \rho \Leftrightarrow \Theta = \arctan \rho.$$
Furthermore, we notice that, for large velocities $|v_1|$, the behaviour of the distribution function $g_0$ is $g_0(v_1) = C_2 |v_1|^{-q}$, thus
\begin{align*}
    \int_{0}^{2\pi} g_0(v_1 + |v_1-v_1'|\rho v_s ) \, \d s 
    &=
    C_2 |v_1|^{-q} 
    \int_{0}^{2\pi} 
    \left| \frac{v_1}{|v_1|} + \frac{|v_1 - v_1'|}{|v_1|}\rho v_s \right|^{-q} \, \d s,
\end{align*}
when $|v_1 + |v_1-v_1'| \rho v_s|\ge \bar{R}$
where 
\begin{align*}
\left| \frac{v_1}{|v_1|} + \frac{|v_1 - v_1'|}{|v_1|}\rho v_s \right|^2 
&= 1 + \frac{|v_1 - v_1'|^2}{|v_1|^2}\rho^2 
+
2 \rho \frac{|v_1-v_1'|}{|v_1|} \frac{v_1}{|v_1|} \cdot v_s \\
& = 1 + \rho^2 + \mathcal{O}(\frac{\rho}{|v_1|}), \quad \text{as} \ |v_1| \to \infty,
\end{align*}
for fixed $v_1' \in \R^3$. Here we used the fact that $v_s \cdot (v_1 - v_1') = 0$ which implies that $v_1 \cdot v_s = v_1' \cdot v_s$. So $v_1 \cdot v_s$ is constant in $v_1$. 
 Therefore, the above computations implies 
\begin{align*}
\int_{0}^{2\pi} g_0(v_1 + |v_1-v_1'|\rho v_s ) \, \d s 
&= 
2\pi C_2 |v_1|^{-q} \left( (1+\rho^2)^{-\frac{q}{2}} + \mathcal{O}\big(\frac{\rho}{|v_1|}\big)\right).
\end{align*} 
Hence, for fixed $v_1'$ and for $|v_1|\to \infty$ the whole integral becomes
\begin{align*}
&\frac{2}{|v_1 - v_1'|^2}\int_{E_{v_1,v_1'}} g_0(v_2') \mathcal{B} (2v_1-v_1'-v_2', \omega)\, \d v_2' \\
 &
 =  
 \frac{2^{\frac{2}{n-1}}}{ a^2} 4\pi C_2  |v_1|^{-q}  \bigg[
 |v_1 |^{\frac{n-5}{n-1}}
 \int_{0}^{\infty}
 \rho (1+\rho^2)^{\frac{n-5}{2(n-1)} -\frac{q}{2}} \frac{\arctan \rho}{ \sin (\arctan \rho)} \d \rho 
  \\ & \qquad \quad \quad \quad +
  \mathcal{O}(|v_1|^{\frac{n-7}{n-1}})
 \int_{0}^{\infty}
 \rho (1+\rho^2)^{\frac{n-7}{2(n-1)} -\frac{q}{2}} \frac{\arctan \rho}{ \sin (\arctan \rho)} \d \rho \\
 &\qquad \quad \quad \quad + |v_1|^{\frac{n-5}{n-1}}\mathcal{O}(\frac{1}{|v_1|}) \int_{0}^{\infty}
 \rho^2 (1+\rho^2)^{\frac{n-5}{2(n-1)} - \frac{q}{2} -1} \frac{\arctan \rho}{ \sin (\arctan \rho)} \d \rho \\ 
 & \qquad \quad \quad \quad + \mathcal{O}(|v_1 |^{\frac{n-7}{n-1}})\mathcal{O}(\frac{1}{|v_1|}) \int_{0}^{\infty}
 \rho^2 (1+\rho^2)^{\frac{n-7}{2(n-1)} -\frac{q}{2}-1}  \frac{\arctan \rho}{ \sin (\arctan \rho)} \d \rho 
  \bigg] \\
  &=: 
   \frac{2^{\frac{2}{n-1}}}{ a^2} 4\pi C_2  |v_1|^{-q + \frac{n-5}{n-1}}
 \bigg[ \int_{0}^{\infty}
 \rho (1+\rho^2)^{\frac{n-5}{2(n-1)} -\frac{q}{2}} \frac{\arctan \rho}{ \sin (\arctan \rho)} \d \rho + R(v_1, v_1')\bigg],
\end{align*}
where the remainder term, denoted by $R(v_1, v_1')$, is continuous in $(v_1, v_1')$  away from the diagonal $v=v'$ and $R(v_1, v_1') \to 0$, as $|v_1| \to \infty$, pointwise in $v_1'$.
 Uniform boundedness of $R$ for $|v_1-v_1'|>L$ follows  as the $v_1'$ dependence of all the remainder terms can be bounded by a negative power of $|v_1-v_1'|$. This completes the proof of the lemma.    
\end{proof}

\subsection{Existence of Equilibrium}
  \begin{proposition}\label{prop:equili}
    There exists a continuous, radial positive and unimodal equilibrium distribution $F$ of \eqref{LBE}.  It is the unique normalised positive equilibrium. 
\end{proposition}
\begin{proof}
    We will apply the Krein-Rutman theorem to the map
\begin{align}
    \label{eqn:defS}
    S(f)(v) = \frac{1}{\nu(v)} \int_{\R^3} \sigma(v,v') f(v')\,  \d v'
\end{align}    
on the Banach space 
\begin{align}
    \label{eqn:defX} X := \{ f \in C^0 ( \R^3) \mid \|f\|_X= \sup_{v \in \R^3} (1+|v|)^{q-\delta} |f(v)| < \infty, 
    \lim_{|v| \to \infty }(1+|v|)^{q-\delta} |f(v)|=0 \},
\end{align}
where we choose $0<\delta<1$ fixed.   

We first show that $S:X \to X$ is bounded. The function $S(f)$ is continuous as $\sigma$ is continuous away from the diagonal $v=v'$ by the standard arguments see e.g. 
Alt \cite[Sec. 10.16]{Alt}. We estimate $(1+|.|)^{q-\delta} |f(.)|$, for $|v|\leq 1$ and we use that $\frac{1}{\nu(v)}$ and  $(1+|v|)^{q-\delta}$ are bounded. Then it is enough to use the first part of Lemma \ref{lem:convolest} to estimate
\begin{align*}
    |S(f)(v)|& \leq \int_{\R^3} \frac{C}{|v-v'|^2} |f(v')| \, \d v' \\&=\int_{|v-v'|\leq L } \frac{C}{|v-v'|^2} |f(v')| \, \d v' +\int_{|v-v'|\geq L } \frac{C}{|v-v'|^2} |f(v') |\, \d v'   \\
    & \leq C \|f\|_{L^\infty} \int_{|v-v'|\leq L } \frac{1}{|v-v'|^2}  \, \d v' +
    \frac{C}{L^2} \int_{|v-v'|\geq L }  |f(v') |\, \d v'\\
    & \leq CL \|f\|_X + \frac{C}{L^2} \int_{|v-v'|\geq L } \|f\|_X \frac{1}{(1+|v'|)^{q-\delta}} \, \d v' \leq C\left( L + \frac{1}{L^2}\right) \|f\|_X. 
\end{align*}
For $|v|\geq 1$ we combine Lemma \ref{lem:sharperconvolution} for $v'$ near $v$ and Lemma \ref{integrand:sigma} for $v'$ away from $v$ for some distance $0<L<1$. We first estimate
\begin{align*}
    |S(f)(v)|& \leq \frac{1}{\nu(v) } \int_{\R^3} \frac{C}{|v-v'|^2} |f(v')| \, \d v' \\&=\frac{1}{\nu(v) } \left(\int_{|v-v'|\leq L } \frac{C}{|v-v'|^2} |f(v')| \, \d v' +\int_{|v-v'|\geq L } \frac{C}{|v-v'|^2} |f(v') |\, \d v' \right) \\
    & \leq C L \frac{1}{\nu(v) |v| }\|f\|_{L^\infty(B_L(v))}  +
    \frac{C}{\nu(v)} |v|^{-q+\frac{n-5}{n-1}} \int_{|v-v'|\geq L }  |f(v') |\, \d v'.
\end{align*}
Using the decay of $\nu$ as in Proposition \ref{L5.2} and that $ \|f\|_{L^\infty(B_L(v))} \leq (1+|v|-L)^{-q+\delta} \|f\|_X$ as well as $\|f\|_{L^1} \leq C \|f\|_X$, we can conclude
\begin{align*}
    (1+|v|)^{q-\delta} |S(f)(v)| \leq
    C L \frac{1}{\nu(v) |v| }\|f\|_X \frac{(1+|v|)^{q-\delta}}{(1+|v|-L)^{q-\delta}} + 
    C (1+|v|)^{
 -\delta} \|f\|_X \leq C(L) \|f\|_X,
\end{align*}
which is bounded and converges to $0$ for $|v| \to \infty$  as long as $\frac{1}{\nu(v)|v|} \to 0 $ for $|v| \to \infty$, which holds with Proposition \ref{L5.2} for $n \in (3,5]$. Hence $S$ is a bounded map on $X$. 

To show compactness of $S$ we use the Arzela-Ascoli Theorem for the weighted functions $(1+|.|)^{q-\delta} f(.)$ on one-point compactification of $\R^3$ with boundary data $0$ at $\infty$. The previous estimates ensure that the functions in the image under $S$ of the unit ball in $X$  are uniformly bounded. It remains to check 
equi-continuity. For the integral parts with $|v-v'| \leq L$, this follows as in \cite[Sec. 10.16]{Alt} for Schur integral operators. Whereas for $|v-v'| \geq L$, we use the continuity of $\sigma$ with respect to both variables and that Lemma \ref{integrand:sigma} ensures uniform continuity in $v$, uniformly in $v'$.

As a cone $C$ for the Krein-Rutman theorem we first use the positive functions in $X$, which satisfies $X= C+ (- C)$. This yields uniqueness of a positive normalised eigenvector $F$. To obtain additional qualitative properties, we observe that $S$ maps radial functions to radial functions such that we can restrict $X$ to the set $X_{rad}$ of radial functions and recover again a unique positive normalised eigenvector, which must coincide with $F$. Within in radial functions, we can use the cone of unimodal (monotonically decaying), positive functions. This cone $C_U$ satisfies that  $C_U + (-C_U)$ is dense in $X_{rad}$, as any element in       $X_{rad}$ can be approximated by a $C^1$ function with compact support and every $C^1$ function with compact support is the sum of a decaying and an increasing function, which can be chosen to be positive and negative respectively. Hence the normalised eigenvector  $F$ is  positive, radial and unimodal.

Then $F$ satisfies for some $\lambda>0$
\begin{align*}
    \lambda \nu(v) F(v)= \int_{\R^3} \sigma(v,v')F(v') \, \d v',
\end{align*}
integrating with respect to $v$ and using \eqref{eqn:defnu} yields 
\begin{align*}
\lambda\int_{\R^3} \nu(v) F(v) \d v= \int_{\R^3} \nu(v') F(v') \, \d v',
\end{align*}
i.e.  $\lambda=1$ as $\nu$ and $F$ are non-negative and non-zero. Hence $F$ is indeed the desired equilibrium.
\end{proof}

By construction in the proof, we already know that $F$ decays
at least as $|v|^{-q+\delta}$ for any fixed $\delta >0$.  With this we are now able to determine the exact decay of $K(F)(v_1)$ for $|v_1| \to \infty$.

\begin{proposition}
\label{L5.1} Let $F \in L^1(\R^3)$ be a radial, positive  and unimodal distribution, normalised such that $\|F\|_{L^1}=1$. 
Then the operator $K(F)$ satisfies the asymptotic behaviour
\begin{align}\label{10.2}
K(F)(v_1)
& =   C_{\alpha}  \bigg[ |v_1|^{-q+\frac{n-5}{n-1}} 
 + o(|v_1|^{-q +\frac{n-5}{n-1}}) \bigg],
\end{align}
as $|v_1|\to \infty$, where $q \in (4,\frac{4n}{n-1})$, $n\in (3,5]$ and $C_{\alpha} >0$ constant.  \end{proposition}

\begin{proof}
    The Carleman representation of the gain term of the linear Boltzmann operator $Q$ is given by
\begin{align*}
K(F)(v_1) &= \int_{\R^3} \left[ \frac{2}{|v_1 - v_1'|^2} \int_{E_{v_1,v_1'}} g_0(v_2') \mathcal{B} (2v_1-v_1'-v_2', \omega)\, \d v_2' \right] F(v_1')\, \d v_1',
\end{align*}
where $E_{v_1,v_1'} = \{ v_2' \in \R^3 \ : \ (v_2'-v_1)\cdot (v_1-v_1') = 0 \}$. 

To compute the behaviour of the operator $K(F)(v_1)$, for $|v_1|\to \infty$, we use Lemma \ref{integrand:sigma} for the main part where $|v_1-v_1'|>L$ and Lemma \ref{lem:sharperconvolution} for the remainder. That is, we plug in the given expression \eqref{eq:sigma} into the integral , so we can write 
\begin{align*}
   & K(F)(v_1) = \int_{\R^3} \sigma(v_1, v_1') F(v_1') \d v_1'\\
    & =\int_{|v_1'-v_1|>L} \frac{2^{\frac{2}{n-1}}}{ a^2} 2\pi C_2  |v_1|^{-q +\frac{n-5}{n-1}}  \Bigl[ 
 \int_{0}^{\infty}
 \rho (1+\rho^2)^{\frac{n-5}{2(n-1)} -\frac{q}{2}} \frac{\arctan \rho}{ \sin (\arctan \rho)} \d \rho + R(v_1,v_1') \Bigr]F(v_1') \d v_1'\\& \quad +H(v_1,L,F) .
\end{align*}
Now, we define the constant 
\begin{equation*}
A:=  \frac{2^{\frac{2}{n-1}}}{ a^2} 2\pi C_2 \int_{0}^{\infty}
 \rho (1+\rho^2)^{\frac{n-5}{2(n-1)} -\frac{q}{2}} \frac{\arctan \rho}{ \sin (\arctan \rho)} \d \rho. 
\end{equation*}
Hence, $K(F)$ can be written as
\begin{align*}
K(F)(v_1) =& A |v_1|^{-q + \frac{n-5}{n-1}} \underbrace{\int_{\R^3} F(v_1') \d v_1'}_{=1} + \frac{2^{\frac{2}{n-1}}}{ a^2} 2\pi C_2 |v_1|^{-q - \frac{n-5}{n-1}} \int_{|v_1'-v_1|>L} R(v_1, v_1') F(v_1') \d v_1' \\&
+ o(|v_1|^{-q +\frac{n-5}{n-1}}) +H(v_1,L, F).
\end{align*}
Furthermore, we know that $R(v_1, v_1') \to 0$, as $|v_1|\to \infty$ pointwise in $v_1'$ and that $R(v_1, v_1')$ is uniformly bounded in $(v_1, v_1')$. We also have that $F\in L^1(\R^3)$, then $$|R(v_1,v_1') F(v_1')|\le C |F(v_1')|.$$ Therefore, we can apply Dominated Convergence Theorem and get
\[ \int_{\R^3} R(v_1, v_1') F(v_1') \d v_1' \to 0, \quad \text{as} \ |v_1|\to \infty. \]
Additionally, $|H(v_1,L,F)| \leq CL |v_1|^{-q+\delta-1}$, it is indeed possible to choose $\delta$
such that 
\begin{align}
    \label{deltaconst} \delta -1 < \frac{n-5}{n-1}
\end{align}
for $n \in (3,5]$ such that the $H$ term is of higher order.

Hence, we have shown that the behaviour of the operator $K(F)$ is 
\begin{align*}
    K(F)(v_1) = C_{\alpha} |v_1|^{-q + \frac{n-5}{n-1} } + o(|v_1|^{-q + \frac{n-5}{n-1} }), \quad \text{as} \ |v_1|\to \infty,
\end{align*}
where $C_{\alpha}:= A\int_{\R^3} F(v_1') \d v_1'. $
\end{proof}

\subsection{Assumptions and proof of Theorem \ref{thm2}}
In the remaining subsection, we are using \cite[Theorem 3.2]{Mellet2011} which we state it here for the convenience of the reader. Before stating the theorem, we list here the key assumptions.

\textbf{Assumption (A1)} The cross-section $\sigma$ in \eqref{LBO} is locally integrable on $\mathds{R}^3 \times \mathds{R}^3$, non negative and the collision frequency $\nu$ is locally integrable on $\mathds{R}^3$ and satisfies 
\begin{align*}
\nu (-v) = \nu (v) >0 \quad \mathrm{for \ all\ } v \in \mathds{R}^3.
\end{align*}

\textbf{Assumption (A2)} There exists a function $0 \le F \in L^1(\nu)$ such that $|v|^2 \nu (v)^{-1} F$ is locally integrable and 
\begin{align*}
\nu (v) F(v) = K(F)(v) = \int_{\mathds{R}^3} \sigma(v,v')F(v') \ \d v' ,
\end{align*}
which means that $F$ is an equilibrium distribution, i.e. $Q(F)=0$. Moreover, $F$ is symmetric, positive and normalised to 1, that is
\begin{align*}
F (-v) = F (v) >0 \quad \mathrm{for \ all\ } v \in \mathds{R}^3 \quad \mathrm{and} \quad \int_{\mathds{R}^3} F(v) \, \d v =1. 
\end{align*}

Before stating the next assumption, we recall the definition of slowly varying functions. A measurable function $\ell:(0, \infty)\to(0, \infty)$ is called slowly varying (at infinity) if for all $a > 0$,
 $\lim _{x\to \infty }{\frac {\ell(ax)}{\ell(x)}}=1.$

\textbf{Assumption (B1)} There exists $\alpha >0$ and a slowly varying function $\ell$ such that
\begin{align*}
F (v) = F_0(v) \ell(|v|), \quad \mathrm{with} \quad |v|^{\alpha + 3} F_0(v) \longrightarrow \kappa_0\in (0,\infty) \quad \mathrm{as} \ |v|\to \infty .
\end{align*}

\textbf{Assumption (B2)} There exists $\beta \in \mathds{R}$ and a positive constant $\nu_0$ such that
\begin{align*}
|v|^{-\beta} \nu(v) \longrightarrow \nu_0 \quad \mathrm{as} \ |v|\to \infty .
\end{align*}

\textbf{Assumption (B3)} We assume that there exists a constant $M$ such that 
 \begin{align*}
\int_{\mathds{R}^3} F(v') \frac{\nu (v)}{ b(v,v')} \ \d v' + \left(\int_{\mathds{R}^3} \frac{F(v')}{\nu(v')} \frac{ b^2(v,v')}{\nu^2(v)} \ \d v' \right)^{\frac{1}{2}} \le M \quad \mathrm{for \ all\ } v\in \mathds{R}^3,
\end{align*}
with $b(v,v') := \sigma (v,v') F^{-1} (v).$ 

\begin{proposition}
    The linear Boltzmann collision operator \eqref{LBO} satisfies the Assumptions $(A_1 - A_2)$ and $(B_1 - B_3)$ above with $\alpha = q-3>0$, $\beta = \frac{n-5}{n-1}<1$ and the slowly varying function $\ell = 1$.
\end{proposition}
\begin{proof}
Assumption $(A_1)$: It is being satisfied by the way the cross-section $\sigma$ is defined and the fact that $g_0$ is a radial, non negative function. 

Assumption $(A_2)$: 
This follows directly from Proposition \ref{prop:equili}.

Assumption $(B_1)$:
 We know that the function $F$ is an equilibrium distribution of the linear Boltzmann collision operator \eqref{LBO}, i.e. $\nu(v_1)F(v_1) = K(F)(v_1)$.
Additionally, by the computations in subsections \ref{S5.1} and \ref{S5.2} above we have that the behaviour of $F$ for large velocities is
 \begin{align}\label{beh:F}
|v_1|^q F(v_1) = |v_1|^q \frac{K(F)(v_1)}{\nu(v_1)} \to \frac{16C_2C_S}{\pi^2} \quad as \ |v_1| \to \infty.
 \end{align}
Now, we aim to find the values of $\alpha$ and $\beta$ in our case, in order to be able to use Theorem \ref{Thm:MMM}. For $(B_1)$, choosing $\ell =1$, we need that the behaviour of the equilibrium distribution $F$ for large velocities is
 \begin{align}\label{beh:FM}
|v_1|^{3+\alpha } F(v_1) = |v_1|^{3+\alpha } \frac{K(F)(v_1)}{\nu(v_1)} \to \kappa_0 \quad as \ |v_1| \to \infty.
 \end{align}
By equating \eqref{beh:F} and \eqref{beh:FM}, we take that $\alpha = q-3 >0$, which implies $q>3$ (actually, that $q = \alpha +3$, $\alpha >0$) which is true since $g_0 \in L^1(\R^3)$ and also that $\kappa_0 = \frac{16C_2C_S}{\pi^2}$. 

Assumption $(B_2)$: By the behaviour of $\nu(v_1)$ for large velocities $v_1$, in \eqref{nu}, choosing $\beta = \frac{n-5}{n-1} <1$, since $n>1$ and the positive constant $\nu_0 = \tilde{C}$. Furthermore,
\begin{align*}
&\beta < \min \{ \alpha ; 2-\alpha \} = \min \{ q-3 ; 5-q \} = 5-q \\
\Leftrightarrow & \beta = \frac{n-5}{n-1} < 5-q
\Leftrightarrow  q< \frac{4n}{n-1}, \ \textrm{since} \ n>1.
\end{align*}
This yields an admissible  choice for   $\alpha$, $\beta$ and $q$. 

Assumption $(B_3)$: We start with the integral
\begin{align*}
\nu(v) \int_{\mathds{R}^3} \frac{F(v')}{b(v,v')} \, \d v' 
&= \nu(v) F(v) \int_{\mathds{R}^3}
\frac{F(v')}{\int_{\mathbb{S}^2} g_0(v) \mathcal{B}(v-v', \omega)\, \d \omega} \, \d v'\\
&=
\frac{K(F)(v)}{g_0(v)} 
\int_{\mathds{R}^3}
\frac{F(v')}{\int_{\mathbb{S}^2} \mathcal{B}(v-v', \omega)\, \d \omega} \, \d v',
\end{align*}
where we used the fact that $F$ is an equilibrium distribution, i.e. $K(F)(v) = \nu(v)F(v)$ and the cross section $\sigma$ by \eqref{LBO}. Now, by Corollary \ref{CorCOB} we find the lower bound for the integral 
\begin{align*}
\int_{\mathbb{S}^2} \mathcal{B}(v-v', \omega)\, \d \omega
&\ge
\frac{\pi^3}{8} \frac{(2K)^{\frac{2}{n-1}}}{a^2}  |v - v'|^{\frac{n-5}{n-1}}.
\end{align*}
Therefore, by Proposition \ref{2.4.} and the behaviour of $K(F)(v)$ for large velocities \eqref{10.2}
\begin{align*}
\nu(v) \int_{\mathds{R}^3} \frac{F(v')}{b(v,v')} \, \d v' 
&\le 
8\pi C_2 C_S \frac{|v|^{-q + \frac{n-5}{n-1}} + o(|v|^{ -q + \frac{n-5}{n-1} })}{|v|^{-q}} (|v|^{-\frac{n-5}{n-1}} + o(|v|^{-\frac{n-5}{n-1}}) )\\
& = 8\pi C_2 C_S.
\end{align*}
The computations for finding an upper bound for the integral 
$$\left( \int_{\mathds{R}^3} \frac{F(v')}{\nu(v')} \frac{ b^2(v,v')}{\nu^2(v)} \ \d v'\right)^{\frac{1}{2}}$$
follows similarly.
\end{proof}
We are now ready to state the theorem of \cite{Mellet2011} which is the following.
\begin{theorem}{\cite[Theorem 3.2]{Mellet2011}} \label{Thm:MMM}
Assume that Assumptions $(A_1-A_2)$ and $(B_1 - B_3)$ hold with $\alpha>0$ and $\beta < \min \{ \alpha ; 2-\alpha\}$. Define 
$$\gamma := \frac{\alpha - \beta}{1-\beta}, \quad \text{and} \quad \theta(\epsilon) := \ell(\epsilon^{-\frac{1}{1-\beta}})\epsilon^{\gamma}.$$
Assume furthermore that $f_0 \in L^2(F^{-1})$ and let $f^{\epsilon}$ be the solution of \eqref{LBE:MMM}, with that choice of $\theta$ and initial data $f_0$. Then, $(f^{\epsilon})$ converges in $L^{\infty}(0,T;L^2(\mathds{R}^3 \times \mathds{R}^3))-$weak$^*$ to a function $\rho(t,x) F(v)$, where $\rho(t,x)$ is the unique solution of the fractional diffusion equation of order $\gamma$, \eqref{FDE}.
\end{theorem}

In our case, we have chosen $\ell =1$, so that $\theta(\epsilon)$ becomes explicitly $\theta(\epsilon) = \epsilon ^{\gamma}$. Furthermore, the admissible choices for $\alpha$ and $\beta$ are \[ \alpha = q-3, \ \text{for} \ q\in(4, \frac{4n}{n-1}) \ \text{and} \ \beta=\frac{n-5}{n-1}, \ \text{for} \ n\in (3,5]. \]
Therefore, \[ \gamma = \frac{(q-4)(n-1)}{4} +1,\ \text{for} \ n\in (3,5], \ q\in (4, \frac{4n}{n-1}). \]
In this step we introduce the time scale in order to derive the rescaled linear Boltzmann equation of \cite[equation (3)]{Mellet2011}, which is the equation
\begin{align}\label{LBE:MMM}
    \epsilon^{\gamma} \partial_{\tau}f
    + \epsilon v \cdot \nabla_x f = Q(f),
\end{align}
with $\epsilon>0$ a small parameter different than the diameter $\varepsilon$ of each particle,
from the linear Boltzmann equation \eqref{LBE}, which is the equation
\begin{align*}
  \partial_{t}f
    + v \cdot \nabla_x f =c Q(f),
\end{align*}
where $c$ is a parameter which is the inverse of the mean free path of the microscopic particles.  
\begin{lemma}\label{timescale}
The macroscopic time scale for deriving the linear Boltzmann equation \eqref{LBE:MMM} from equation \eqref{LBE} is
\[ \tau = \epsilon^{\gamma-1} t
\quad \text{with} \quad  c = \frac{1}{\epsilon}.\]
\end{lemma}
\begin{proof}

Start with the linear Boltzmann equation \eqref{LBE} and set the rescale time $\tilde{\beta}\tau := t$, then the equation becomes
\begin{align*}
    \frac{1}{\tilde{\beta}} \partial_{\tau}f + v \cdot \nabla_x f = c Q(f),
\end{align*}
which is equivalent to
\begin{align*}
    \frac{1}{\tilde{\beta}c} \partial_{\tau}f + \frac{1}{c} v \cdot \nabla_x f =  Q(f).
\end{align*}
If we equate this equation with the equation \eqref{LBE:MMM}, we get
\[c = \frac{1}{\epsilon} \quad \mathrm{and} \quad \tilde{\beta} = \epsilon^{1-\gamma}, \]
which concludes the proof of the lemma.
\end{proof}
Now we have all the tools we need to give the proof of the remained main theorem, which is a theorem about the full derivation of a fractional diffusion equation from a short range potential Rayleigh particle system, Theorem \ref{thm2}.
\begin{proof}[Proof of Theorem \ref{thm2}]

By Theorem \ref{thm1} we have that for any $t \in [0,T_{\varepsilon}]$ the distribution of the tagged particle $\hat{f}_t^{N}(x,v)$ converges in ${L^1(\R^3 \times \R^3)}$-norm to the solution of the linear Boltzmann equation $f (t,x,v)$. This implies the strong convergence in the space $L^{\infty}(0,T; L^1(\R^3 \times \R^3))$ which immediately implies the weak$^*$ convergence in the same space. That is, for any test function $\phi \in L^1(0,T; L^{\infty}(\R^3 \times \R^3))$
\begin{align*}
   \int_{0}^{T} \int_{\R^3 \times \R^3} |(\hat{f}^{N}(t,x,v) - f(t,x,v) ) \cdot \phi(t,x,v) | \d x \d v \d t \to 0,
\end{align*}
as $N\to \infty$, which by the Boltzmann-Grad limit is equivalent to $\varepsilon \to 0$. Now by the Theorem \ref{Thm:MMM} we know that the rescaled solution of the linear Boltzmann equation $f^{\epsilon}(\epsilon^{\gamma}t, \epsilon x, v)$ converges weakly$^*$ in $ L^{\infty}(0,T; L^2_{F^{-1}}(\R^3 \times \R^3))$ to a function $ \rho(t,x) F(v)$, where $\rho (t,x)$ is the unique solution of the fractional diffusion equation \eqref{FDE} and the function $F$ is an equilibrium distribution of the linear Boltzmann operator \eqref{LBO}, i.e. $Q(F) = 0$. That is to say, for any test function $\psi \in L^1(0,T; L^2_{F^{-1}} (\R^3 \times \R^3))$
\begin{align*}
  \int_{0}^{T} \int_{\R^3 \times \R^3}  |(f^{\epsilon}(\epsilon^{\gamma} t, \epsilon x, v) - \rho(t,x)F(v) ) \cdot \psi(t,x,v) | \d x \d v \d t \to 0,
\end{align*}
as the parameter $\epsilon \to 0$ which is equivalent to $c \to \infty.$
Therefore, by combining these two theorems and using the triangle inequality, one can show that for any test function $\tilde{\varphi} \in L^1( 0,T; L^2_{F^{-1}} \cap L^{\infty} (\R^3 \times \R^3))$
\begin{align*}
 \int_{0}^{T} \int_{\R^3 \times \R^3}   |(\hat{f}^{N}(\epsilon^{1-\gamma} \tau,x,v) - \rho(\tau,x)F(v) ) \cdot \tilde{\varphi}(\tau,x,v) | \d x \d v \d \tau \to 0,
\end{align*}
in the limit $N \to \infty$, with $c = N \varepsilon^2 \to \infty.$
 \end{proof}

 \subsection*{Acknowledgements}
This work was supported through  The Leverhulme Trust research project grant  RPG-2020-107.  KM would like to thank the Isaac Newton Institute for Mathematical Sciences, Cambridge, for support and hospitality during the programme \emph{Frontiers in kinetic theory}  where work on this paper was undertaken. This work was supported by EPSRC grant no EP/R014604/1. We thank Josephine Evans, Milos Tasic and Florian Theil for helpful discussions.

\subsection*{Data Availibility}
Data sharing not applicable to this article as no datasets were generated or analysed during the current study.

\subsection*{Conflict of interest}
The authors have no relevant financial or non-financial interests to disclose.

\end{document}